\documentclass{article}

\usepackage[utf8]{inputenc}
\usepackage{subfigure}
\usepackage{amsmath}
\usepackage{amssymb}
\usepackage{amsthm}
\usepackage[pdftex]{hyperref}
\usepackage{tikz}
\usepackage{caption}
\usepackage[round]{natbib}
\usepackage{fancyhdr}
\usepackage{amsfonts}
\usepackage{times}
\usepackage{ifpdf}
\usepackage{latexsym}
\usepackage{graphicx}
\usepackage{enumerate}

\newtheorem{definition}{Definition}[section]
\newtheorem{proposition}[definition]{Proposition}
\newtheorem{theorem}[definition]{Theorem}
\newtheorem{lemma}[definition]{Lemma}
\newtheorem{corollary}[definition]{Corollary}
\newtheorem{remark}[definition]{Remark}
\newtheorem{openproblem}[definition]{Open problem}

\newlength{\taille} \makeatletter
\def\qed{%
  \ifmmode\vrule width .5\baselineskip height 0pt depth .5\baselineskip%
  \else{%
    \unskip\nobreak\hfil%
    \setlength{\taille}{\f@size\p@}%
    \penalty50\hskip1em\null\nobreak\hfil\vrule width .5\taille height
    0pt depth .5\taille
    \parfillskip=0pt\finalhyphendemerits=0\endgraf}%
  \fi} \makeatother

\newlength{\taillepreuve}
\newenvironment{demo}{%
  \setbox123=\hbox{Proof:}%
  \taillepreuve=\wd123%

  \vskip-\lastskip\nobreak\medskip\par\noindent\box123\list{}{\leftmargin
    .5\taillepreuve}\parindent=1em\item} {\qed\endlist\bigskip}

\DeclareMathOperator{\Domain}{Domain}
\DeclareMathOperator{\OrderFunction}{OrderFunction}

\DeclareMathOperator{\Height}{Height}
\DeclareMathOperator{\Level}{Level}

\DeclareMathOperator{\Inv}{Inv}
\DeclareMathOperator{\Next}{Next}
\DeclareMathOperator{\Length}{L}

\DeclareMathOperator{\Prelude}{Prelude}

\DeclareMathOperator{\MultiplyCost}{M}
\DeclareMathOperator{\Up}{Up}
\DeclareMathOperator{\Down}{Down}

\DeclareMathOperator{\DigitMin}{DigitMin}
\DeclareMathOperator{\Sup}{Supremum}
\DeclareMathOperator{\NeigbourhoodOrder}{NeBOr}
\DeclareMathOperator{\RMFTrunk}{RMFTrunk}

\DeclareMathOperator{\Identify}{Identify}
\DeclareMathOperator{\LinearExtensions}{LinearExtensions}
\DeclareMathOperator{\Fusion}{Fusion}
\DeclareMathOperator{\Label}{Label}

\newcommand\np{\ensuremath{\mathsf{NP}}}
\newcommand\diesep{\ensuremath{\mathsf{\#P}}}
\newcommand\rft{\ensuremath{\mathsf{RFT}}}
\newcommand\rmft{\ensuremath{\mathsf{RMFT}}}
\newcommand{\nn}{\ensuremath{\mathbb N}}
\newcommand{\zz}{\ensuremath{\mathbb Z}}

\newcommand\mycaption[1]{\caption{\protect\parbox[t]{\captionwidth}{#1}}}

\begin{document}

%\pagenumbering{arabic}
%\setcounter{page}{1}

%\pagestyle{fancy}
%\fancypagestyle{plain}{
%\fancyhf{}%
%\fancyhead[L]{Fraternal Journal of Mathematics and Informatics}%
%\fancyhead[R]{FJMI vol. ?}%
%\fancyfoot[L]{ISSN 111}%
%\fancyfoot[C]{\copyright by the author(s)}%
%\fancyfoot[R]{\newline Distributed under a Creative Commons Attribution 4.0 International License}%
%}

\author{Laurent Lyaudet\footnote{\url{https://lyaudet.eu/laurent/}, laurent.lyaudet@gmail.com}}
\title{On finite width questionable representations of orders}
%\keywords{order, universal, lexicographic, next, questionable representation}

%\publicationdetails{VOL}{2017}{ISS}{NUM}{SUBM}
\maketitle
\begin{abstract}
In this article, we study ``questionable representations'' of (partial or total) orders,
 introduced in our previous article ``A class of orders with linear? time sorting algorithm''.
(Later, we consider arbitrary binary functional/relational structures instead of orders.)
A ``question'' is the first difference between two sequences (with ordinal index) of elements of orders/sets.
In finite width ``questionable representations'' of an order \begin{math}O\end{math},
comparison can be solved by looking at the ``question'' that compares elements
of a finite order \begin{math}O'\end{math}.
A corollary of a theorem by Cantor (1895) is that
all countable total orders have a binary (width 2) questionable representation.
We find new classes of orders on which testing isomorphism or counting the number of linear extensions
can be done in polynomial time.
We also present a generalization of questionable-width, called balanced tree-questionable-width,
and show that if a class of binary structures has bounded tree-width or clique-width,
then it has bounded balanced tree-questionable-width.
But there are classes of graphs of bounded balanced tree-questionable-width
and unbounded tree-width or clique-width.
\end{abstract}

Current version : 2019/12/30 with masthead removed


\section{Introduction}
\label{section:introduction}

Apologies: We do science as a hobby, it is not our daily job and there is an impact on the quality of the bibliography.
For an unpublished work we did in 2015, we started doing bibliographic search during 9 months,
but all the gathered references were lost when a hacker erased all our files on our laptop.
Since then, we chose to publish our ideas on arXiv and correct the bibliography afterwards.
We tried our best, with this version, to correct missing references for well-known definitions.
Many results in this article still seem to be new, the main contribution being a generalization of first difference principle in a sequence,
to first difference principle in a tree.

In this article, we study ``questionable representations'' of (partial or total) orders,
 introduced in our previous article ``A class of orders with linear? time sorting algorithm''
 (\cite{DBLP:journals/corr/abs-1809-00954}).
A ``question'' is the first difference between two sequences (with ordinal index) of elements of orders/sets.
In a finite width ``questionable representation'' of an order \begin{math}O\end{math},
comparison can be solved by looking at the ``question'' that compares elements
of a finite order \begin{math}O'\end{math}.
A corollary of a theorem by \cite{Cantor1895} is that
all countable total orders have a binary (width 2) questionable representation.
We study the class of partial orders of questionable-width 2 and some related classes of orders,
exhibiting a wealth of structural results.
The following classes of orders appear naturally:
series parallel orders,
series parallel interval orders,
weak orders.
We also prove a few algorithmic results such as counting linear extensions for cedars,
testing order isomorphism for up-regular orders.
Last we study the links between questionable-width, tree-width and clique-width,
proving that tree-width and clique-width (resp. tree-width and questionable-width)
are incomparable with respect to weighted or labeled graphs.
However, questionable-width (for finite width and finite length) is strictly weaker than clique-width.
We also present a generalization of questionable-width, called tree-questionable-width,
that is too powerful since any binary structure (structure with binary relations or functions, such as a graph, an order, etc.)
has linear tree-questionable-width 2.
We limit this generalization to balanced tree-questionable-width and show that
if a class of binary structures has bounded tree-width or clique-width,
then it has bounded balanced tree-questionable-width.
We also show that there are classes of graphs of bounded balanced tree-questionable-width
and unbounded tree-width or clique-width.

Section \ref{section:definitions_and_notations} contains most of the definitions and notations used in this article.
In section \ref{section:strict_binary_questionable_representations_for_total_orders},
we review results of Hausdorff and Sierpi\'nski extending the theorem of Cantor to all total orders.
Section \ref{section:total_questionable_representations_for_partial_orders} characterizes
the class of partial orders with total questionable representation (of width 2);
this is the class of series parallel interval orders with an additional constraint when the order is infinite.
In section \ref{section:algorithms_for_finite_orders_with_property_itov}, we present a wealth of structural results for these partial orders,
using some particular induced weak orders.
Section \ref{section:questionable_representations_width2_for_partial_orders} characterizes the class of partial orders of questionable-width 2;
this is the class of series parallel orders.
In section \ref{section:questionable_representations_for_partial_orders}, we study the links between questionable-width, tree-width and clique-width.
It also presents (balanced) tree-questionable-width.

\section{Definitions and notations}
\label{section:definitions_and_notations}

Throughout this article, we use the following definitions and notations.
\begin{math}O\end{math} denotes an order (it may be either a partial, or a total/linear order),
in particular \begin{math}O^{0,1}\end{math} denotes the binary total order where \begin{math}0 < 1\end{math}.
We denote \begin{math}\Domain(O)\end{math}, the domain of the order \begin{math}O\end{math}
(for example, \begin{math}\Domain(O^{0,1}) = \{0,1\}\end{math}).
We write \begin{math}x < y\end{math}, and \begin{math}x > y\end{math} as usual to express the order between two elements;
we also write \begin{math}x \sim y\end{math} when two elements are incomparable in the partial order considered.
We denote \begin{math}\OrderFunction(O)\end{math}, the order function of the order \begin{math}O\end{math}
defined from \begin{math}\Domain(O)^2\end{math} to \begin{math}\{=,\sim,<,>\}\end{math}
(for example, \begin{math}\OrderFunction(O^{0,1}) = \{((0,0),=),~((0,1),<),~((1,0),>),~((1,1),=)\}\end{math}).
\begin{math}\mathcal{O}_i\end{math} denotes a sequence of orders indexed by the ordinal \begin{math}i\end{math},
\begin{math}i = \Length(\mathcal{O}_i)\end{math} is the length of \begin{math}\mathcal{O}_i\end{math},
in particular \begin{math}\mathcal{O}^{0,1}_{\omega} = (O^{0,1})_{\omega}\end{math}
denotes the sequence of binary orders repeated a countable number of times, \begin{math}\omega\end{math} is its length.
Given two ordinals \begin{math}i < j\end{math}, and a sequence of orders \begin{math}\mathcal{O}_{j}\end{math},
we denote \begin{math}\mathcal{O}_{j}[i] = O_{i}\end{math}, the item of rank \begin{math}i\end{math} in the sequence
(the ranks start at 0).
The reader might know Von Neumann's construction of the ordinals
(an ordinal can be seen as a set that contains exactly all ordinals that are strictly before it, the 0th ordinal is the empty set),
in which case we can consider that \begin{math}i \in j \Leftrightarrow i < j\end{math}.
We also use this notation, for example in \begin{math}\mathcal{O}_{i} = (O_j)_{j \in\ i}\end{math}.
While ordinals are frequently denoted by greek letters, 
we will try to keep using \begin{math}i, j, k, l\end{math} for this purpose,
so that it recalls finite indices to the reader.

We denote \begin{math}\Inv(O)\end{math}, the inverse order of \begin{math}O\end{math};
for example, \begin{math}\Inv(O^{0,1}) = O^{1,0}\end{math} is the order on 0 and 1
where \begin{math}1 < 0\end{math}.
We also denote \begin{math}\Inv(\mathcal{O}_i) = (\Inv(O_j))_{j \in\ i}, O_j = \mathcal{O}_i[j]\end{math},
the sequence of inverse orders of \begin{math}\mathcal{O}_i\end{math};
for example, if \begin{math}\mathcal{O}_3 = (O^{0,1}, O^{0,1}, O^{0,1,2})\end{math},
then \begin{math}\Inv(\mathcal{O}_3) = (O^{1,0}, O^{1,0}, O^{2,1,0})\end{math}
(each order in the sequence is inverted but the ranks of the items are preserved).
Note that \begin{math}\Inv\end{math} is an involution: \begin{math}\Inv(\Inv(O)) = O\end{math}, 
and \begin{math}\Inv(\Inv(\mathcal{O}_i)) = \mathcal{O}_i\end{math}.

\begin{definition}[Prelude sequence]
Given two ordinals \begin{math}i < j\end{math},
and two sequences of orders/sets \begin{math}\mathcal{O}_{i}, \mathcal{O}_{j}\end{math},
we say that \begin{math}\mathcal{O}_{i}\end{math} is a \emph{prelude (sequence)} of \begin{math}\mathcal{O}_{j}\end{math},
if and only if \begin{math}\mathcal{O}_{i}[k] = \mathcal{O}_{j}[k], \forall k \in i\end{math} (we have \begin{math}k < i < j\end{math}).
We note \begin{math}\Prelude(\mathcal{O}_{j})\end{math} the set of all prelude sequences
 of the order/set sequence \begin{math}\mathcal{O}_{j}\end{math}.
Remark that since an ordinal indexed sequence \begin{math}s\end{math} is just a total order,
a prelude sequence of \begin{math}s\end{math} is just a new name for a proper initial segment of \begin{math}s\end{math}.
We hope the reader will pardon us our nonconformism and the use of ``prelude'' instead of ``proper initial segment''.
\end{definition}

For example, \begin{math}(O^{0,1}, O^{0,1})\end{math} is a prelude of \begin{math}(O^{0,1}, O^{0,1}, O^{0,1,2})\end{math};
it is also a prelude of \begin{math}\mathcal{O}^{0,1}_{\omega}\end{math}.

We say that \begin{math}X\end{math} is an element or a word of \begin{math}\mathcal{O}_{i}\end{math} (denoted \begin{math}X \in \mathcal{O}_{i}\end{math}),
when \begin{math}X = (x)_{i}\end{math} is a sequence indexed by \begin{math}i\end{math}
with \begin{math}x_{k} \in \Domain(O_{k}) = \Domain(\mathcal{O}_{i}[k]), \forall k \in i\end{math}.
By convention, there is a unique sequence of orders/sets of length 0 \begin{math}\mathcal{O}_{0}\end{math} ; 
it has only one element denoted \begin{math}\epsilon\end{math} (the empty word).
\begin{math}\mathcal{O}_{0}\end{math} is a prelude sequence of any other order/set sequence.
We try to avoid confusion by distinguishing \emph{element} \begin{math}X\end{math} of \begin{math}\mathcal{O}_{i}\end{math}
from \emph{item} \begin{math}\mathcal{O}_{j}[i]\end{math} of \begin{math}\mathcal{O}_{i}\end{math};
for example, the word 002 is an element of \begin{math}\mathcal{O}_3 = (O^{0,1}, O^{0,1}, O^{0,1,2})\end{math},
whilst \begin{math}O^{0,1}\end{math} is the first and second (order-)item of \begin{math}\mathcal{O}_3\end{math}.
Given two ordinals \begin{math}i < j\end{math}, and an element/word \begin{math}X\end{math} of \begin{math}\mathcal{O}_{j}\end{math},
we denote \begin{math}X[i] = x_{i}\end{math}, the (element-)item of rank \begin{math}i\end{math} in the sequence
(the ranks start at 0).
We can say that the element/word \begin{math}X = 002\end{math} contains element-items \begin{math}X[0] = 0\end{math},
 \begin{math}X[1] = 0\end{math}, and \begin{math}X[2] = 2\end{math}.

\begin{definition}[Compatible (sequences)]
Given two sequences of orders/sets \begin{math}\mathcal{O}_{i}, \mathcal{O}_{j}\end{math},
we say that they are \emph{compatible} if they are equal, or one is a prelude sequence of the other.
Let \begin{math}X \in \mathcal{O}_{i}, Y \in \mathcal{O}_{j}\end{math},
we say that \begin{math}X, Y\end{math} are \emph{compatible} if \begin{math}\mathcal{O}_{i}, \mathcal{O}_{j}\end{math} are compatible.
\end{definition}

\emph{Compatible} elements are easily converted into \emph{comparable} elements.
Indeed, both element-items at the same rank in the two elements may be compared,
 since they belong to the domain of the same order-item.

\begin{definition}[Question]
Given two compatible elements \begin{math}X, Y\end{math} of sequences of orders/sets \begin{math}\mathcal{O}_{i}, \mathcal{O}_{j}\end{math},
we say that \begin{math}(k, x_k, y_k)\end{math} is the \emph{question} of \begin{math}X, Y\end{math},
if \begin{math}k\end{math} is the smallest ordinal such that \begin{math}x_k \neq y_k\end{math}.
If \begin{math}X \neq Y\end{math}, and neither \begin{math}X\end{math} is a prefix of \begin{math}Y\end{math},
nor \begin{math}Y\end{math} is a prefix of \begin{math}X\end{math},
 such a \begin{math}k\end{math} exists because ordinals are well-ordered. 
By contrapositive, if no such a \begin{math}k\end{math} exists, then \begin{math}X = Y\end{math},
 or either \begin{math}X\end{math} is a prefix of \begin{math}Y\end{math},
 or \begin{math}Y\end{math} is a prefix of \begin{math}X\end{math}.
 Thus if no such a \begin{math}k\end{math} exists, then \begin{math}X = Y\end{math},
 or \begin{math}\Length(X) \neq \Length(Y)\end{math}.
\end{definition}

\begin{lemma}
Let \begin{math}S\end{math} be a set of pairwise compatible order/set sequences,
then there exists an order/set sequence \begin{math}\mathcal{O}_{j}\end{math} of which all order/set sequences of \begin{math}S\end{math} are preludes.
\end{lemma}

\begin{definition}[Next partial order]
Given two ordinals \begin{math}i < j\end{math},
and a sequence of orders \begin{math}\mathcal{O}_{j} = (O_k)_{k \in\ j}\end{math},
the \emph{next partial order} denoted \begin{math}\Next(i, j, \mathcal{O}_{j})\end{math}
is a partial order defined on the set of all elements of the preludes of \begin{math}\mathcal{O}_{j}\end{math},
such that these preludes have length at least \begin{math}i\end{math}.
This is the partial order satisfying \begin{math}\forall X, Y \in \Domain(\Next(i, j, \mathcal{O}_{j}))\end{math}
(the lengths \begin{math}\Length(X), \Length(Y)\end{math} of \begin{math}X\end{math} and \begin{math}Y\end{math}
are such that \begin{math}i \leq \Length(X), \Length(Y) < j\end{math}),
\begin{itemize}
\item if \begin{math}X, Y\end{math} have a question \begin{math}(k, x_k, y_k)\end{math}, then:
\begin{itemize}
  \item if \begin{math}x_k < y_k\end{math} in \begin{math}O_{k}\end{math}, then \begin{math}X < Y\end{math},
  \item if \begin{math}x_k > y_k\end{math} in \begin{math}O_{k}\end{math}, then \begin{math}X > Y\end{math},
  \item if \begin{math}x_k \sim y_k\end{math} (\begin{math}x_k\end{math} and \begin{math}y_k\end{math} are incomparable)
        in \begin{math}O_{k}\end{math}, then \begin{math}X \sim Y\end{math},
\end{itemize}
\item if they don't have a question, then elements are not ordered (\begin{math}X \sim Y\end{math}).
\end{itemize}
\end{definition}

The partial order \begin{math}\Next(i, j, \mathcal{O}_{j})\end{math}
corresponds to the simple idea for comparing sequences
``if current items are equal, compare next items''.
It illustrates the first difference principle.
Lexicographic and contre-lexicographic orders are the two simplest linear extensions of next partial order.

\begin{definition}[Order embedding]
An \emph{order embedding} is an injective mapping \begin{math}f\end{math} such that order is preserved.
``Order is preserved'' means that no order relation is removed or added between injected elements.
(\begin{math}\forall x,y \in \Domain(O), f(x) < f(y) \Leftrightarrow x < y, \text{and} f(x) > f(y) \Leftrightarrow x > y\end{math})
\end{definition}

\begin{definition}[Universal order]
We say that an order \begin{math}O\end{math} is \emph{universal} for a class of orders \begin{math}\mathcal{A}\end{math},
 if for any order \begin{math}O' \in \mathcal{A}\end{math},
 there exists an order embedding of \begin{math}O'\end{math} in \begin{math}O\end{math}.
\end{definition}

\begin{theorem}[Cantor 1895]
\begin{math}\Next(1, \omega, \mathcal{O}^{0,1}_{\omega})\end{math} is universal for countable total orders.
\end{theorem}

\begin{definition}[Questionable representation]
We say that an order embedding from the domain of an order \begin{math}O\end{math}
 to the domain of the partial order \begin{math}\Next(i, j, \mathcal{O}_{j})\end{math}
is a \emph{questionable representation} of the order \begin{math}O\end{math}.
Indeed, two elements of the order \begin{math}O\end{math} are ordered if and only if they have a question.
It is a \emph{strict questionable representation} if any two images of two elements of
\begin{math}O\end{math} have a question.
Thus, if \begin{math}O\end{math} is a total order, it may only have strict questionable representations.
In this article, we may consider uniform sequences of orders \begin{math}\mathcal{O}_{j} = (O')_{j}\end{math},
 and we define the cardinal of \begin{math}O'\end{math} as the \emph{width} of the questionable representation.
If the sequence is not uniform, then the \emph{width} of the questionable representation
 is the supremum of the cardinals of the used orders.
\begin{math}j\end{math} is the \emph{length} of the questionable representation.
A questionable representation is a \emph{total questionable representation},
if \begin{math}O'\end{math} is a total order.
If the questionable representation is a \emph{total questionable representation of width 2},
we say that it is a \emph{binary questionable representation}.
\end{definition}

\begin{corollary}
Any countable total order has a strict binary questionable representation of length at most \begin{math}\omega\end{math}.
\end{corollary}

\begin{remark}
\label{remark:inv_qr}
If \begin{math}O\end{math} has a questionable representation,
then \begin{math}\Inv(O)\end{math} has a similar questionable representation,
where the injective mapping of elements of \begin{math}O\end{math} and \begin{math}\Inv(O)\end{math}
are identical and \begin{math}\Next(i, j, \mathcal{O}_{j})\end{math} is replaced by
\begin{math}\Next(i, j, \Inv(\mathcal{O}_{j}))\end{math}.
\end{remark}

\section{Strict binary questionable representations for total orders}
\label{section:strict_binary_questionable_representations_for_total_orders}

In this section, we give an affirmative answer to Open problem 6.6 in \cite{DBLP:journals/corr/abs-1809-00954}.
Any total order has a strict binary questionable representations.
Such a strict questionable representation of width 3 was found by \cite{Hausdorff1907},
and later Sierpi\'nski proved that a strict binary questionable representation exists (see \cite{Sierpinski1932}).
However, the proof by Hausdorff can be modified to replace ternary digits
by words of length 2 of binary digits (-1 becomes 00, 0 becomes 10 or 01, and 1 becomes 11).
Thus in both cases, a strict binary questionable representation is proven to exist of the same limit ordinal length.
We note that the proof of Hausdorff proves that the binary words may be
ultimately periodic with a finite period in \begin{math}(0|1)^* \setminus (0^* \cup 1^*)\end{math},
whilst the proof by Sierpi\'nski (and its mirror proof) proves that the binary words may be ultimately periodic of any finite period.
We give a partially new proof of the result by Hausdorff and Sierpi\'nski.

\begin{theorem}[Hausdorff 1907, Sierpi\'nski 1932]
\label{theorem:total_order_to_qr}
Any total order of cardinal \begin{math}\aleph\end{math} has a strict binary questionable representation
 of length at most \begin{math}\alpha(\aleph) + 1\end{math},
where \begin{math}\alpha(\aleph)\end{math} is the first ordinal of cardinal \begin{math}\aleph\end{math}.
(If \begin{math}\aleph = \aleph_{\beta}\end{math} is an infinite cardinal,
then \begin{math}\alpha(\aleph) = \omega_{\beta}\end{math},
but we also include the case where \begin{math}\aleph\end{math} is a finite cardinal.
It is length \begin{math}\alpha(\aleph) + 1\end{math} instead of
\begin{math}\alpha(\aleph)\end{math} because of the strict inequality
in the definition of next partial order for the ordinal upper bound.)
\end{theorem}
\begin{demo}
Let \begin{math}O\end{math} be a total order of cardinal \begin{math}\aleph\end{math}.
Since \begin{math}\alpha(\aleph)\end{math} is an ordinal of cardinal \begin{math}\aleph\end{math},
by Zermelo's axiom, we consider a bijection \begin{math}f\end{math} between \begin{math}\Domain(O)\end{math}
 and \begin{math}\Domain(\alpha(\aleph))\end{math}
(this bijection does not respect order;
it maps any element of \begin{math}\Domain(O)\end{math} to an ordinal strictly less than \begin{math}\alpha(\aleph)\end{math}).

We will now consider the sequence of total orders \begin{math}\mathcal{O}^{0,1}_{\alpha(\aleph)}\end{math}
that we will order partially with 
\begin{math}\Next(1, \alpha(\aleph) + 1, \mathcal{O}^{0,1}_{\alpha(\aleph) + 1})\end{math}.
The idea is simply to associate one bit of information to each element of \begin{math}O\end{math},
when needed to break ties.

We now proceed by transfinite induction.
Our induction hypothesis at ordinal rank \begin{math}j\end{math} is that:
\begin{itemize}
\item For any ordinal \begin{math}0 < i \leq j\end{math}, the first \begin{math}i\end{math} elements of \begin{math}O\end{math}
(the elements with rank at least 0 and less than \begin{math}i\end{math}),
according to the bijection \begin{math}f\end{math}, were associated to elements of
\begin{math}\mathcal{O}^{0,1}_{l}, l \leq i\end{math} such that order on these associated elements
 given by \begin{math}\Next\end{math} matches the induced suborder of \begin{math}O\end{math}
 on these first \begin{math}i\end{math} elements.
(Any two associated elements have a question since \begin{math}O\end{math} is a total order.)
We denote \begin{math}w_{i}\end{math} the current ``word-function'' that associates these elements of 
\begin{math}\mathcal{O}^{0,1}_{l}\end{math}
 to the first \begin{math}i\end{math} elements of \begin{math}O\end{math}.

\item Moreover, our induction hypothesis is strengthened by the fact that for any ordinals
\begin{math}i \leq j\end{math} and \begin{math}r < l\end{math},
and for any element \begin{math}x\end{math} among the first \begin{math}i\end{math} elements of \begin{math}O\end{math},
we have \begin{math}w_{i}(x)[r] = w_{j}(x)[r]\end{math}.
(The word associated to any element of \begin{math}O\end{math} is progressively lengthened as the induction progresses,
but without modifying its beginning.)
\end{itemize}
This two requirements are our induction hypothesis.
The proof using this induction hypothesis follows.
\begin{itemize}
\item The induction hypothesis is trivially true for \begin{math}j = 1\end{math},
since we associate the word \begin{math}0\end{math} to the first element.
Without loss of generality, we also ``add'' one element to the order \begin{math}O\end{math}:
\begin{math}M\end{math} which is more than all other elements of \begin{math}O\end{math}.
\begin{math}M\end{math} is initially associated to the word \begin{math}1\end{math}.

\item Let \begin{math}i + 1\end{math} be a successor ordinal,
 and let \begin{math}x\end{math} be the \begin{math}(i + 1)\end{math}th element (element with ordinal rank \begin{math}i\end{math}).
Assume, by induction, that the first \begin{math}i\end{math} elements of \begin{math}O\end{math}
(elements with ordinal rank strictly less than \begin{math}i\end{math}),
according to the bijection \begin{math}f\end{math}, were associated to elements of
\begin{math}\mathcal{O}^{0,1}_{l}, l \leq i\end{math} such that order on these associated elements
 given by \begin{math}\Next\end{math} matches the induced suborder of \begin{math}O\end{math}
 on these first \begin{math}i\end{math} elements.
Let \begin{math}\Down(x)\end{math} (resp. \begin{math}\Up(x)\end{math})
be all elements among the first \begin{math}i\end{math} elements of \begin{math}O\end{math}
that are less (resp. more) than \begin{math}x\end{math}.

We associate to \begin{math}x\end{math} a word of length \begin{math}l\end{math}
as follow: 
Let us define \begin{math}w_{i}(x) = \DigitMin(\{w_{i}(y) | y \in \Up(x)\})\end{math} the word where the digit
at rank \begin{math}r\end{math} is the minimum between all digits at rank \begin{math}r\end{math}
in the words \begin{math}\{w_{i}(y), y \in \Up(x)\}\end{math}.
(The minimum at each rank is well defined since there is only two possible values and we have the 
element \begin{math}M \in \Up(x)\end{math}.)
It is clear that if \begin{math}w_{i}(x)\end{math}
 has a question with \begin{math}w_{i}(y)\end{math}, for some \begin{math}y \in \Up(x)\end{math},
then this question orders \begin{math}x\end{math} so that it is less than \begin{math}y\end{math}.
If we now turn our attention to \begin{math}\Down(x)\end{math}, assume for a contradiction
that \begin{math}w_{i}(x)\end{math} has a question with \begin{math}w_{i}(y)\end{math},
for some \begin{math}y \in \Down(x)\end{math} and that question orders \begin{math}x < y\end{math}.
Let \begin{math}j\end{math} be the rank of the question between \begin{math}x\end{math} and \begin{math}y\end{math}.
Let \begin{math}z \in \Up(x)\end{math} be such that the digit of \begin{math}w(z)\end{math}
at rank \begin{math}j\end{math}, denoted by \begin{math}w_{i}(z)[j]\end{math},
equals the digit of \begin{math}w_{i}(x)\end{math} at rank \begin{math}j\end{math}, denoted by \begin{math}w_{i}(x)[j]\end{math}.
(Such a \begin{math}z\end{math} exists by definition of \begin{math}w_{i}(x)\end{math}.)
\begin{math}z\end{math} must have a question at rank \begin{math}k\end{math}, before rank \begin{math}j\end{math},
with \begin{math}y\end{math}, such that \begin{math}y < z\end{math}.
If \begin{math}w_{i}(x)[k] = w_{i}(z)[k]\end{math}, there is also a question
 between \begin{math}x\end{math} and \begin{math}y\end{math}, such that \begin{math}x > y\end{math}, at rank \begin{math}k\end{math}.
A contradiction.
If \begin{math}w_{i}(x)[k] < w_{i}(z)[k]\end{math}, then let \begin{math}z' \in \Up(x)\end{math}
 be such that \begin{math}w_{i}(x)[k] = w_{i}(z')[k]\end{math}.
We can repeat the argument with \begin{math}z'\end{math} and \begin{math}k\end{math}
instead of \begin{math}z\end{math} and \begin{math}j\end{math},
we obtain a new question at rank \begin{math}k'\end{math} before \begin{math}k\end{math},
and obtain either a contradiction or a \begin{math}z''\end{math} and a \begin{math}k''\end{math}, etc.
We are guaranteed to obtain a contradiction after a finite number of steps because ordinals are well-ordered
and \begin{math}j, k, k', k''...\end{math} are a strictly decreasing sequence of ordinal ranks.
Thus, either the order between \begin{math}x\end{math} and \begin{math}\Down(x)\end{math} is preserved,
or \begin{math}w_{i}(x)\end{math} is incomparable with some \begin{math}w_{i}(y), y \in \Down(x)\end{math}.

Hence, since all words associated to elements have the same length,
either \begin{math}w_{i}(x) = \DigitMin(\{w_{i}(y) | y \in \Up(x)\})\end{math}
is distinct from \begin{math}w_{i}(y), y \in \Down(x) \cup \Up(x)\end{math},
it has a question with all other words and we do not need to lengthen the words;
or, there is at most one \begin{math}y_{=} \in \Down(x) \cup \Up(x)\end{math}
such that \begin{math}w_{i}(y_{=}) = w_{i}(x)\end{math}.
In that last case, first we extend the words of length \begin{math}l\end{math} associated to the
first \begin{math}i\end{math} elements into words of length \begin{math}l + 1\end{math}
by concatenating \begin{math}0\end{math} (resp. \begin{math}1\end{math}) to all elements in
 \begin{math}\Down(x)\end{math} (resp. \begin{math}\Up(x)\end{math}).
Since all these elements have a question, the order between them is not changed.
Let \begin{math}w_{i+1}\end{math} be the extended ``word-function'' that associates these elements of
\begin{math}\mathcal{O}^{0,1}_{l + 1}\end{math}
 to the first \begin{math}i\end{math} elements of \begin{math}O\end{math}.
Next, we extend \begin{math}w_{i}(x)\end{math} into a word \begin{math}w_{i+1}(x)\end{math} of length
\begin{math}l + 1\end{math} by concatenating \begin{math}1\end{math}, respectively \begin{math}0\end{math}, at the end,
if \begin{math}y_{=} \in \Down(x)\end{math}, respectively \begin{math}y_{=} \in \Up(x)\end{math}.

\item Let \begin{math}j\end{math} be a limit ordinal.
Assume, by induction, that for any ordinal \begin{math}i < j\end{math}, the first \begin{math}i\end{math} elements of \begin{math}O\end{math},
according to the bijection \begin{math}f\end{math}, were associated to elements of
\begin{math}\mathcal{O}^{0,1}_{l}, l \leq i\end{math} such that order on these associated elements
 given by \begin{math}\Next\end{math} matches the induced suborder of \begin{math}O\end{math}
 on these first \begin{math}i\end{math} elements.
Since \begin{math}j\end{math} is a limit ordinal,
for any element \begin{math}y\end{math} among the first \begin{math}j\end{math} elements of \begin{math}O\end{math}
(elements with rank at least 0 and strictly less than \begin{math}j\end{math}),
\begin{math}y\end{math} is also an element among the first \begin{math}k < j\end{math} elements of \begin{math}O\end{math}.
We do not have new elements to consider.

According to the fact that the word-functions never contradict themselves,
for any element \begin{math}y\end{math} among the first \begin{math}j\end{math} elements,
we can define a word \begin{math}w_{j}(y)\end{math} such that,
for any ordinal \begin{math}r < \Sup({\Length(w_{i}(y)), i < j})\end{math},
then \begin{math}w_{j}(y)[r] = w_{i}(y)[r]\end{math},
(ranks of the digits start at 0 and are strictly less than the length of the word).
Clearly \begin{math}w_{j}(y)\end{math} does not contradict previous word-functions 
for \begin{math}y\end{math}.
Since for any two elements \begin{math}y, z\end{math} among the first \begin{math}j\end{math} elements,
there is some ordinal \begin{math}k < j\end{math} such that \begin{math}w_{k}(y)\end{math}
and \begin{math}w_{k}(z)\end{math} have a question that orders them like
\begin{math}y\end{math} and \begin{math}z\end{math} are ordered in \begin{math}O\end{math},
it is now clear that \begin{math}w_{j}(y)\end{math} and \begin{math}w_{j}(z)\end{math} have the same question
and thus the order is still preserved.
\end{itemize}
This completes the proof by transfinite induction.
\end{demo}

As we noted in our previous article, questionable representations are well-ordered representations
and it is surprising that any total order admits a well-ordered representation.

We remark that this theorem is not ``space efficient''.
Indeed, \begin{math}\omega + 1\end{math} length is sufficient to represent the order of the real numbers between 0 and 1,
although that order is not countable.
With this theorem, we need the first ordinal of cardinal \begin{math}\mathfrak{c}\end{math} (the cardinality of the continuum),
this ordinal is at least \begin{math}\omega_{1}\end{math}, the first uncountable ordinal.
Since the width of the obtained questionable representation can not be improved,
it leaves the following open problem.
\begin{openproblem}
Can the length of the questionable representation obtained in Theorem~\ref{theorem:total_order_to_qr} be improved?
\end{openproblem}
The proof we gave, let us hope that a more carefully chosen ordinal order on the elements for the transfinite induction
may lengthen the words less frequently, and give better bounds.

We note that for any total order, it is possible to define an invariant (up to isomorphism),
that we name \emph{questionable ordinality}, as the minimum ordinal such that the total order admits
a strict binary questionable representation of this ordinal length.
Proving that, for any ordinal \begin{math}\alpha\end{math},
\begin{math}\Next(\alpha +1, \alpha + 2, \mathcal{O}^{0,1}_{\alpha + 2})\end{math}
can not be embedded into
\begin{math}\Next(\alpha, \alpha + 1, \mathcal{O}^{0,1}_{\alpha + 1})\end{math}
would give a negative answer to the open problem,
showing that the results by Hausdorff and Sierpi\'nski are optimal,
and that for any ordinal, there exists a total order with such questionable ordinality.

\section{Total questionable representations for partial orders}
\label{section:total_questionable_representations_for_partial_orders}

In this section, we study Open problem 6.7 in \cite{DBLP:journals/corr/abs-1809-00954}.
We first note that, thanks to Theorem~\ref{theorem:total_order_to_qr},
a partial order \begin{math}O\end{math} admits a total questionable representation
if and only if it admits a binary total questionable representation,
since we can replace any ``digit'' over some order of cardinal more than two
with the equivalent word given by the theorem.
These digit replacements do not change the fact that the word is ordinal indexed.

We also note that we can assume without loss of generality that the partial order \begin{math}O\end{math}
does not contain elements that are incomparable with all other elements.
Indeed, words that are prefix of each other are incomparable and thus,
we can encode any globally incomparable element with a distinct sequence of zeros.
All these sequences of zeros are incomparable,
and they admit the existence of a sequence of zeros \begin{math}s\end{math}
that is longer than any of these sequences.
If we obtained a questionable representation for \begin{math}O\end{math} without the globally incomparable elements,
we can extend it by prefixing any obtained word with \begin{math}s\end{math},
and associating any globally incomparable element with its distinct sequence of zeros.

We remark that if \begin{math}O\end{math} has a total questionable representation,
then it is also the case for any induced suborder of \begin{math}O\end{math}.
We will try to prove a reciprocal over finite induced suborders: if all finite induced suborders of \begin{math}O\end{math}
have a total questionable representation, then \begin{math}O\end{math} has also a total questionable representation.
This approach will fail, but we will gain structural information from it.

We recall that an antichain of an order is a set of pairwise incomparable elements.
A maximal antichain is an antichain that cannot be extended further 
(because for any element outside of the antichain, some element in this antichain is ordered with it).

Since we consider total questionable representations,
two elements are not ordered if and only if they do not have a question if and only if one is the prefix of the other.
After these remarks, we can list small partial orders that admit or not a total questionable representation.
\begin{itemize}
\item With one or two elements, we only have globally incomparable elements.
These partial orders do admit a total questionable representation.
\item With three elements, thanks to Remark~\ref{remark:inv_qr},
   there is only one partial order to consider: \begin{math}O = (\{a,b,c\}, \{a < b, c < b\})\end{math}.
   We also have a positive answer by associating \begin{math}w(a) = 0, w(b) = 1, w(c) = 00\end{math}.
\item With four elements, we have the following orders to consider
  (we present them with increasing number of levels; this number is the minimal number of antichains that cover the order;
  it will make sense if you draw the directed graph representing each order;
  when there is only one level, all elements are incomparable;
  when the number of levels equals the finite cardinal of the order, it is a total order;
  it leaves us only 2 or 3 levels for partial orders of cardinal 4):
  \begin{itemize}
    \item \begin{math}O = (\{a,b,c,d\}, \{a < b, c < b, d < b\})\end{math}.
    We have a positive answer by associating \begin{math}w(a) = 0, w(b) = 1, w(c) = 00, w(d) = 000\end{math}.
    The same applies to \begin{math}\Inv(O)\end{math}.

    \item \begin{math}O = (\{a,b,c,d\}, \{a < b, c < d\})\end{math}.
    By symmetry, without loss of generality, we may assume that a is a prefix of c.
    But then, since b is more than a, it has a question with a, and since a is a prefix of c,
    c has a question with b, and c is less than b, a contradiction.
    This case exhibits a fact that we can easily deduce: whenever two elements \begin{math}x, y\end{math} are incomparable,
    then one is the prefix of the other, and either all elements less (resp. more) than \begin{math}x\end{math} are less (resp. more) than \begin{math}y\end{math},
    or all elements less (resp. more) than \begin{math}y\end{math} are less (resp. more) than \begin{math}x\end{math}.
    It has the consequence that the orders that admit a total questionable representation
    are connected (in terms of the corresponding directed graph that models the partial order)
    if we forget the elements that are incomparable with all other elements.
    (Only one connected component may be of cardinality more than one.)
    For this case, \begin{math}O \equiv \Inv(O)\end{math} (these orders are isomorphic).

    \item \begin{math}O = (\{a,b,c,d\}, \{a < b, c < d, c < b\}) \equiv \Inv(O)\end{math}.
    We added the relation \begin{math}c < b\end{math} that was missing for previous case.
    However, it is not enough because this time \begin{math}a\end{math} and \begin{math}d\end{math}
    are incomparable but they both are in distinct levels (antichains),
    and they both have an order relation that the other does not have.
    This case exhibits another fact that we can easily deduce:
    if we consider two antichains (of size at least 2; the following affirmation is true but trivial for size 1)
    \begin{math}A,B\end{math} of a partial order
    such that each element of \begin{math}A\end{math} is ordered with at least one element of \begin{math}B\end{math}
    and each element of \begin{math}B\end{math} is ordered with at least one element of \begin{math}A\end{math},
    then either all elements of \begin{math}A\end{math} are less than all elements of \begin{math}B\end{math},
    or all elements of \begin{math}A\end{math} are more than all elements of \begin{math}B\end{math}.
    (Indeed, otherwise, there would be \begin{math}a \in A, b \in B\end{math} that are incomparable,
    or there would be \begin{math}a,a' \in A, b,b' \in B\end{math} such that \begin{math}a < b\end{math} 
    and \begin{math}a' > b'\end{math} with \begin{math}a \neq a', \text{or}\ b \neq b'\end{math}.
    If \begin{math}a \in A, b \in B\end{math} are incomparable, they do not respect the previous fact:
    \begin{math}a\end{math} is ordered with at least one element \begin{math}b' \in B\end{math},
    whilst \begin{math}b \sim b'\end{math} since \begin{math}B\end{math} is an antichain;
    \begin{math}b\end{math} is ordered with at least one element \begin{math}a' \in A\end{math},
    whilst \begin{math}a \sim a'\end{math} since \begin{math}A\end{math} is an antichain; we have a contradiction.
    If there is \begin{math}a,a' \in A, b,b' \in B\end{math} such that \begin{math}a < b\end{math} 
    and \begin{math}a' > b'\end{math} with \begin{math}a \neq a', \text{or}\ b \neq b'\end{math},
    then by previous fact applied to \begin{math}a,a'\end{math},
    we have \begin{math}a' < b\end{math} or \begin{math}a > b'\end{math} or both if \begin{math}a = a'\end{math};
    and thus we have by transitivity \begin{math}b' < b\end{math} in both cases;
    again we have a contradiction.)

    \item \begin{math}O = (\{a,b,c,d\}, \{a < b, a < d, c < d, c < b\}) \equiv \Inv(O)\end{math}.
    We have a positive answer by associating \begin{math}w(a) = 0, w(b) = 1, w(c) = 00, w(d) = 11\end{math}.
    This was the last order on 2 levels to consider.

    \item \begin{math}O = (\{a,b,c,d\}, \{a < b, a < c, b < c, d < b, d < c\})\end{math}
    (a total order on a,b,c with d added in the level of a).
    We have a positive answer by associating \begin{math}w(a) = 0, w(b) = 10, w(c) = 11, w(d) = 00\end{math}.
    The same applies to \begin{math}\Inv(O)\end{math}.

    \item \begin{math}O = (\{a,b,c,d\}, \{a < b, a < c, b < c, a < d, d < c\}) \equiv \Inv(O)\end{math}
    (a total order on a,b,c with d added in the level of b).
    We have a positive answer by associating \begin{math}w(a) = 0, w(b) = 10, w(c) = 11, w(d) = 100\end{math}.
    
    \item \begin{math}O = (\{a,b,c,d\}, \{a < b, a < c, b < c, d < c\})\end{math}
    (a total order on a,b,c with d less than c and incomparable with both a and b).
    We have a positive answer by associating \begin{math}w(a) = 00, w(b) = 010, w(c) = 111, w(d) = 0\end{math}.
    The same applies to \begin{math}\Inv(O)\end{math}.
    This case exhibits the fact that if some element is incomparable with two elements that are ordered,
    then it is a prefix of both elements.
    This was the last order on 3 levels to consider.

  \end{itemize}
\end{itemize}

This case study gave us the two following necessary conditions:
\begin{enumerate}[(i)]
\item Whenever two elements \begin{math}x, y\end{math} are incomparable,
    then either all elements less (resp. more) than \begin{math}x\end{math} are less (resp. more) than \begin{math}y\end{math},
    or all elements less (resp. more) than \begin{math}y\end{math} are less (resp. more) than \begin{math}x\end{math}.
\item If we consider two antichains \begin{math}A,B\end{math} of a partial order
    such that each element of \begin{math}A\end{math} is ordered with at least one element of \begin{math}B\end{math}
    (which is the case if \begin{math}B\end{math} is a maximal antichain),
    and each element of \begin{math}B\end{math} is ordered with at least one element of \begin{math}A\end{math}
    (which is the case if \begin{math}A\end{math} is a maximal antichain),
    then either all elements of \begin{math}A\end{math} are less than all elements of \begin{math}B\end{math},
    or all elements of \begin{math}A\end{math} are more than all elements of \begin{math}B\end{math}.
    (Note that it implies that \begin{math}A \cap B = \emptyset\end{math}.)
\end{enumerate}
If (i) is not satisfied, then there are \begin{math}x, y, z, t\end{math}, such that \begin{math}x \sim y, x \not\sim z, y \not\sim t, x \sim t, y \sim z\end{math}.
(\begin{math}x < z, y < z\end{math} or \begin{math}x > z, y > z\end{math} would not be a contradiction;
\begin{math}x < z, y > z\end{math} or \begin{math}x > z, y < z\end{math} would imply that \begin{math}x \not\sim y\end{math}.
The same applies with \begin{math}t\end{math} instead of \begin{math}z\end{math}.)
Thus we have a finite obstruction isomorphic to order
\begin{math}O_{obs1} = (\{a,b,c,d\}, \{a < b, c < d\}) \equiv \Inv(O_{obs1})\end{math}, \begin{math}z \sim t\end{math},
or order \begin{math}O_{obs2} = (\{a,b,c,d\}, \{a < b, c < d, c < b\}) \equiv \Inv(O_{obs2})\end{math}, \begin{math}z \not\sim t\end{math},
(use graph theory, count the number of possible edges/arcs and compare with the list of orders of cardinal 4,
comparing the degrees of the vertices is sufficient to distinguish between the remaining orders with the same number of edges).
These are our first and second finite obstructions.

If (ii) is not satisfied, then there are \begin{math}x \in A, y \in B\end{math}, such that \begin{math}x \sim y\end{math} (first case),
or there are \begin{math}x \in A, y \in B, z \in A, t \in B\end{math}, such that \begin{math}x < y, z > t\end{math} (second case).
\begin{itemize}
\item If there are \begin{math}x \in A, y \in B\end{math}, such that \begin{math}x \sim y\end{math},
      we also have \begin{math}x' \in A, y' \in B\end{math}, such that \begin{math}x \not\sim y', x' \not\sim y, x \sim x', y \sim y'\end{math}.
      Again, if \begin{math}x' \sim y'\end{math}, we have a finite obstruction \begin{math}O_{obs1}\end{math}.
      If \begin{math}x' \not\sim y'\end{math}, we have a finite obstruction \begin{math}O_{obs2}\end{math}.
\item If there are \begin{math}x \in A, y \in B, z \in A, t \in B\end{math}, such that \begin{math}x < y, z > t\end{math},
      we also have \begin{math}x \sim z, y \sim t\end{math}.
      If \begin{math}x < t\end{math}, then \begin{math}x < z\end{math}, a contradiction.
      If \begin{math}x > t\end{math}, then \begin{math}t < y\end{math}, a contradiction.
      Thus \begin{math}x \sim t\end{math}.
      If \begin{math}y < z\end{math}, then \begin{math}x < z\end{math}, a contradiction.
      If \begin{math}y > z\end{math}, then \begin{math}t < y\end{math}, a contradiction.
      Thus \begin{math}y \sim z\end{math}.
      Again, we have a finite obstruction \begin{math}O_{obs1}\end{math}.
\end{itemize}

Since we have the same obstructions, we proved the following lemma.
\begin{lemma}
  For an order \begin{math}O\end{math}, the following properties are equivalent:
\begin{enumerate}[(i)]
\item Whenever two elements \begin{math}x, y \in O\end{math} are incomparable,
    then either all elements less (resp. more) than \begin{math}x\end{math} are less (resp. more) than \begin{math}y\end{math},
    or all elements less (resp. more) than \begin{math}y\end{math} are less (resp. more) than \begin{math}x\end{math}.

\item If we consider two antichains \begin{math}A,B\end{math} of \begin{math}O\end{math}
    such that each element of \begin{math}A\end{math} is ordered with at least one element of \begin{math}B\end{math},
    and each element of \begin{math}B\end{math} is ordered with at least one element of \begin{math}A\end{math},
    then either all elements of \begin{math}A\end{math} are less than all elements of \begin{math}B\end{math},
    or all elements of \begin{math}A\end{math} are more than all elements of \begin{math}B\end{math}.

\item If we consider two maximal antichains \begin{math}A,B\end{math} of \begin{math}O\end{math},
    then either all elements of \begin{math}A\end{math} are less than all elements of \begin{math}B\end{math},
    or all elements of \begin{math}A\end{math} are more than all elements of \begin{math}B\end{math}.

\item If we consider two antichains \begin{math}A,B\end{math} of \begin{math}O\end{math}
    such that \begin{math}A\end{math} is a maximal antichain,
    and each element of \begin{math}A\end{math} is ordered with at least one element of \begin{math}B\end{math},
    then either all elements of \begin{math}A\end{math} are less than all elements of \begin{math}B\end{math},
    or all elements of \begin{math}A\end{math} are more than all elements of \begin{math}B\end{math}.

\item The order \begin{math}O\end{math} does not contain an induced suborder isomorphic to
    \begin{math}O_{obs1} = (\{a,b,c,d\}, \{a < b, c < d\}) \equiv \Inv(O_{obs1})\end{math},
    or isomorphic to \begin{math}O_{obs2} = (\{a,b,c,d\}, \{a < b, c < d, c < b\}) \equiv \Inv(O_{obs2})\end{math}.
\end{enumerate}
\end{lemma}

\begin{proof}
  It is easy to prove this lemma by demonstrating that all properties (i), (ii), (iii) and (iv) are equivalent to (v).
  When studying the partial orders of cardinal 4, we observed that:
  \begin{itemize}
    \item the negation of (v) implies the negation of (i);
    \item the negation of (v) implies the negation of (ii);
    \item similarly, we can observe that the negation of (v) implies the negation of (iii), and it implies the negation of (iv).
  \end{itemize}
  Just above the lemma, we proved that:
  \begin{itemize}
    \item the negation of (i) implies the negation of (v);
    \item the negation of (ii) implies the negation of (v);
    \item similarly, we can observe that since (iii) and (iv) are implied by (ii),
          then the negation of (iii) (resp. (iv)) implies the negation of (ii) that implies the negation of (v).
  \end{itemize}
\end{proof}

Excluding an induced suborder isomorphic to \begin{math}O_{obs1}\end{math} characterizes the class of interval orders
(for countable orders but what about uncountable orders?);
excluding an induced suborder isomorphic to \begin{math}O_{obs2}\end{math} characterizes the class of series parallel orders;
thus orders with the 5 equivalent properties of previous lemma are the series parallel interval orders.
These orders were studied in the article by \cite{Disanto2010},
but most of the results on them in our article seems to be new.
In previous versions of this article, we named property (itov) the fact of being a series parallel interval order.
We keep the name ``property (itov)'' in most parts of this article,
since it is way shorter than ``being a series parallel interval order'',
and it reminds of the five equivalent properties, which we use,
instead of focusing on interval representation that we do not use,
and series parallel representation that we use later.

Being a series parallel interval order is necessary to obtain a total questionable representation, but it is not sufficient.
Indeed, we remarked how we could deal with globally incomparable elements by adding words of zeros as prefix of all other words.
This approach can be extended by proceeding step by step, at each step considering all elements that have a set of neighbours 
(elements more or less than it) that is minimal for inclusion among all elements that have not yet been processed.
Clearly the first step deals with globally incomparable elements (if there are some).
However, the ``minimal for inclusion'' is not always defined.
We could have an infinite decreasing chain of such elements.

For example, if we consider the positive integers ordered as usual and we add one copy of each positive integer to this order,
such that a ``copy-integer'' of integer \begin{math}i\end{math} 
is less than integers more than \begin{math}i\end{math} and incomparable with integers at most \begin{math}i\end{math},
and any copy-integer is incomparable to any other copy-integer.
It is not hard to see that this example is a series parallel interval order,
 since it does not contain one of the two finite obstructions.
(You can also remark that two elements are incomparable if and only if at least one of them is a copy-integer.
If both are, clearly the directed neighbourhood (taking into account the order) of the copy-integer of \begin{math}i\end{math} contains
the directed neighbourhood of the copy-integer of \begin{math}j\end{math} when \begin{math}i < j\end{math}.
If we observe the integer \begin{math}i\end{math} and the copy-integer of \begin{math}j\end{math} that are incomparable,
then \begin{math}i \leq j\end{math} and again the directed neighbourhood of the integer \begin{math}i\end{math} contains
the directed neighbourhood of the copy-integer of \begin{math}j\end{math}.
)

Clearly, when the directed neigbourhood of some element \begin{math}x\end{math}
is included in the directed neighbourhood of an incomparable element \begin{math}y \sim x\end{math},
then the word associated to \begin{math}x\end{math} must be a prefix of the word associated to \begin{math}y\end{math}.
Thus an ever decreasing chain of such neigbourhoods implies that we have an infinite chain of always smaller prefixes,
which is impossible since we consider words indexed by ordinals. 

Note that the contrapositive of property (i) is that
when the directed neighbourhoods of two elements are not ordered by inclusion,
then these two elements are ordered.
(The reciprocal of the contrapositive,
namely that when two elements are ordered then their directed neighbourhoods are not ordered by inclusion,
is always true.)

\begin{definition}[Neighbourhood order of an order]
Given an order \begin{math}O\end{math}, the \emph{neighbourhood order} of \begin{math}O\end{math},
denoted by \begin{math}\NeigbourhoodOrder(O)\end{math},
is defined as follow:
\begin{itemize}
  \item \begin{math}\Domain(\NeigbourhoodOrder(O)) = \Domain(O)\end{math},
  \item \begin{math}x <_{\NeigbourhoodOrder(O)} y\end{math} if and only if
  \begin{itemize}
    \item \begin{math}\{z <_{O} x, z \in \Domain(O)\} \subset \{z <_{O} y, z \in \Domain(O)\}\end{math}
      and \begin{math}\{z >_{O} x, z \in \Domain(O)\} \subseteq \{z >_{O} y, z \in \Domain(O)\}\end{math},
    \item or \begin{math}\{z <_{O} x, z \in \Domain(O)\} \subseteq \{z <_{O} y, z \in \Domain(O)\}\end{math}
      and \begin{math}\{z >_{O} x, z \in \Domain(O)\} \subset \{z >_{O} y, z \in \Domain(O)\}\end{math},
  \end{itemize}
\end{itemize}
\end{definition}

\begin{lemma}
If \begin{math}\NeigbourhoodOrder(O)\end{math} does not contain an infinite decreasing chain,
then there is a linear extension of \begin{math}\NeigbourhoodOrder(O)\end{math} that is a well-order.
It orders \begin{math}\Domain(O) = \Domain(\NeigbourhoodOrder(O))\end{math}
in bijection with an ordinal denoted by \begin{math}\alpha(\NeigbourhoodOrder(O))\end{math}.
\end{lemma}

\begin{theorem}
\label{theorem:itov_order_to_qr}
Any partial order \begin{math}O\end{math} with property (itov),
% of cardinal \begin{math}\aleph\end{math},
such that \begin{math}\NeigbourhoodOrder(O)\end{math} does not contain an infinite decreasing chain,
 has a total binary questionable representation
 of length at most \begin{math}(2 \times \alpha(\NeigbourhoodOrder(O))) + 1\end{math}.
\end{theorem}
\begin{demo}
We consider the bijection 
\begin{math}f\end{math} between \begin{math}\Domain(O) = \Domain(\NeigbourhoodOrder(O))\end{math}
 and \begin{math}\Domain(\alpha(\NeigbourhoodOrder(O)))\end{math}
This bijection is easily extended to a bijection \begin{math}g\end{math} between the set 
\begin{math}\Domain(O) \times \{0,1\}\end{math} and the set \begin{math}\Domain(2 \times \alpha(\NeigbourhoodOrder(O)))\end{math}.
In order that any element of \begin{math}O\end{math} is associated to two consecutive ordinals in
\begin{math}2 \times \alpha(\NeigbourhoodOrder(O))\end{math}.

We will now consider the sequence of total orders \begin{math}\mathcal{O}^{0,1}_{2 \times \alpha(\NeigbourhoodOrder(O))}\end{math}
that we will order partially with 
\begin{math}\Next(2, 2 \times \alpha(\NeigbourhoodOrder(O)) + 1, \mathcal{O}^{0,1}_{2 \times \alpha(\NeigbourhoodOrder(O)) + 1})\end{math}.
The idea is simply to associate two bits of information to each element of \begin{math}O\end{math}.
One of the bits will be set to 0 on the element and all later elements that are less than it,
and set to 1 on all later elements that are more than it;
and the other bit will be set to 1 on the element and all later elements that are more than it,
and set to 0 on all later elements that are less than it.

Let \begin{math}x\end{math} be the \begin{math}(i + 1)\end{math}th element (element with ordinal rank \begin{math}i\end{math}).
We associate to \begin{math}x\end{math} a word of length \begin{math}2 \times (i + 1)\end{math}
as follow:
Let us define \begin{math}w(x)\end{math} as the word where:
\begin{itemize}
  \item the digit at rank \begin{math}2 \times r\end{math}, \begin{math}r < i\end{math},
        is 0 if \begin{math}x\end{math} is less than the \begin{math}(r+1)\end{math}th element,
        1 otherwise.
  \item the digit at rank \begin{math}(2 \times r) + 1\end{math}, \begin{math}r < i\end{math},
        is 1 if \begin{math}x\end{math} is more than the \begin{math}(r+1)\end{math}th element,
        0 otherwise.
  \item the digit at rank \begin{math}2 \times i\end{math} is 1.
  \item the digit at rank \begin{math}(2 \times i) + 1\end{math} is 0.
\end{itemize}
Equivalently, grouping digits by pairs, we could say that
\begin{math}w(x)[2 \times i,(2 \times i)+1] = 10\end{math},
and, for \begin{math}r < i\end{math},
\begin{math}w(x)[2 \times r,(2 \times r)+1] = 00\end{math} if \begin{math}x < f^{-1}(r)\end{math},
\begin{math}w(x)[2 \times r,(2 \times r)+1] = 11\end{math} if \begin{math}x > f^{-1}(r)\end{math},
\begin{math}w(x)[2 \times r,(2 \times r)+1] = 10\end{math} otherwise.

Now, we prove that \begin{math}\Next(2, 2 \times \alpha(\NeigbourhoodOrder(O)) + 1, \mathcal{O}^{0,1}_{2 \times \alpha(\NeigbourhoodOrder(O)) + 1})\end{math}
applied to the words \begin{math}\{ w(x), x \in \Domain(O)\}\end{math} matches the partial order \begin{math}O\end{math}.

Consider two elements \begin{math}x, y\end{math}  of \begin{math}O\end{math}.
Let \begin{math}i\end{math} be the ordinal rank of \begin{math}x\end{math}
(\begin{math}x\end{math} is the \begin{math}(i+1)\end{math}th element).
Let \begin{math}j\end{math} be the ordinal rank of \begin{math}y\end{math}
(\begin{math}y\end{math} is the \begin{math}(j+1)\end{math}th element).
Assume without loss of generality that \begin{math}j < i\end{math}.

\begin{itemize}
\item If \begin{math}w(x)\end{math} and \begin{math}w(y)\end{math} does not have a question,
in particular \begin{math}w(x)[2 \times j] = w(y)[2 \times j] = 1\end{math},
and \begin{math}w(x)[(2 \times j)+1] = w(y)[(2 \times j)+1] = 0\end{math};
by definition of these four digits, it implies that \begin{math}x\end{math} and \begin{math}y\end{math} are incomparable.

\item We now consider the case when \begin{math}w(x)\end{math} and \begin{math}w(y)\end{math} does have a question.
If the question is at rank \begin{math}2 \times j\end{math} or at rank \begin{math}(2 \times j)+1\end{math}
by definition of these four digits, it implies that \begin{math}x\end{math} and \begin{math}y\end{math}
are ordered exactly as \begin{math}w(x)\end{math} and \begin{math}w(y)\end{math} are ordered by \begin{math}\Next\end{math}.
If the question is at rank \begin{math}2 \times k\end{math} or at rank \begin{math}(2 \times k)+1\end{math},
for some \begin{math}k < j\end{math},
let \begin{math}z\end{math} be the \begin{math}(k+1)\end{math}th element;
by definition of these four digits, we have the following cases to consider:
\begin{itemize}
  \item \begin{math}x \sim z\end{math} and \begin{math}y \not\sim z\end{math},
        in that case \begin{math}w(x)[2 \times k] = w(z)[2 \times k] = 1\end{math},
        and \begin{math}w(x)[(2 \times k)+1] = w(z)[(2 \times k)+1] = 0\end{math}.
        Since \begin{math}x \sim z\end{math} and order \begin{math}O\end{math} has property (itov),
        \begin{math}x\end{math} is ordered with \begin{math}y\end{math}
        exactly like \begin{math}z\end{math} is ordered with \begin{math}y\end{math}.
        The order is preserved;
  \item \begin{math}x \not\sim z\end{math} and \begin{math}y \sim z\end{math},
        that case is similar to previous case;
  \item \begin{math}x \not\sim z\end{math} and \begin{math}y \not\sim z\end{math}.
        Either \begin{math}w(x)[2 \times k,(2 \times k)+1] = 00\end{math},
        or \begin{math}w(x)[2 \times k,(2 \times k)+1] = 11\end{math}.
        Either \begin{math}w(y)[2 \times k,(2 \times k)+1] = 00\end{math},
        or \begin{math}w(y)[2 \times k,(2 \times k)+1] = 11\end{math}.
        Since they have a question on these two digits,
        either \begin{math}w(x)[2 \times k,(2 \times k)+1] = 00\end{math}
        and \begin{math}w(y)[2 \times k,(2 \times k)+1] = 11\end{math},
        or \begin{math}w(x)[2 \times k,(2 \times k)+1] = 11\end{math},
        and \begin{math}w(y)[2 \times k,(2 \times k)+1] = 00\end{math}.
        Thus, either we observe that \begin{math}w(x) <_{\Next} w(z)\end{math},
        \begin{math}w(x) <_{\Next} w(y)\end{math}, \begin{math}w(z) <_{\Next} w(y)\end{math},
        which implies by definition of these digits that \begin{math}x < z\end{math},
        \begin{math}z < y\end{math}; by transitivity we have \begin{math}x < y\end{math} and the order is preserved.
        Or we observe that \begin{math}w(z) <_{\Next} w(x)\end{math},
        \begin{math}w(y) <_{\Next} w(x)\end{math}, \begin{math}w(y) <_{\Next} w(z)\end{math},
        which implies by definition of these digits that \begin{math}z < x\end{math},
        \begin{math}y < z\end{math}; by transitivity we have \begin{math}y < x\end{math} and the order is preserved.
\end{itemize}
\end{itemize}

In all cases \begin{math}x \end{math} and \begin{math}y\end{math}
are ordered exactly as \begin{math}w(x)\end{math} and \begin{math}w(y)\end{math} are ordered by \begin{math}\Next\end{math}.
\end{demo}

The length of the questionable representation can not be improved in the previous theorem for orders of infinite cardinal.
Indeed, we already observed that the words should be prefix of each others according to the order \begin{math}\NeigbourhoodOrder(O)\end{math}.
And, for any ordinal \begin{math}\beta\end{math} of same cardinal than \begin{math}O\end{math},
you can construct another order \begin{math}O'\end{math} of same cardinal than \begin{math}O\end{math},
such that \begin{math}\NeigbourhoodOrder(O')\end{math} contains an infinite ascending chain of length \begin{math}\beta\end{math}.
(Consider \begin{math}O'\end{math} containing the ordinal \begin{math}\beta\end{math} along with copy-elements
such that each copy-element is only ordered with (more than) the elements of the ordinal \begin{math}\beta\end{math} that are less than its ``original value''.
The antichain of copy-elements yields the desired chain in \begin{math}\NeigbourhoodOrder(O')\end{math}
and it is easy to check that \begin{math}O'\end{math} has property (itov).)

\section{Algorithms for finite series parallel interval orders}
\label{section:algorithms_for_finite_orders_with_property_itov}

From (v) of property (itov), we have immediately a polynomial time algorithm
to test if a (partial) order has property (itov).
Indeed, it is sufficient to enumerate all 4-tuples of elements of the partial order
and verify if one of the corresponding induced suborder is isomorphic to one of the two obstructions
\begin{math}O_{obs1} = (\{a,b,c,d\}, \{a < b, c < d\}) \equiv \Inv(O_{obs1})\end{math},
and \begin{math}O_{obs2} = (\{a,b,c,d\}, \{a < b, c < d, c < b\}) \equiv \Inv(O_{obs2})\end{math}.
This yields an \begin{math}O(n^4)\end{math} algorithm,
where \begin{math}n\end{math} is the number of elements of the order \begin{math}P\end{math}.
(We will use \begin{math}P\end{math} to denote partial or total orders in this section,
in order to avoid confusion with asymptotic \begin{math}O\end{math} notation.)

But a faster algorithm to test membership in the family of finite (itov) orders is also possible.
Indeed, from (i) of property (itov), we can test for each pair of elements \begin{math}x,y \in P\end{math}
whether they are incomparable both in the order and in the corresponding neighbourhood order, or not.
This test can be done in \begin{math}O(n)\end{math} time:
\begin{itemize}
\item testing if the elements are incomparable is immediate;
\item then we consider the comparibility matrix restricted to the two rows corresponding to the elements,
we iterate over all columns starting with a state ``equality between \begin{math}x\end{math} and \begin{math}y\end{math}'':
  \begin{itemize}
  \item if some third element is less than \begin{math}x\end{math} and more than \begin{math}y\end{math},
        or the symetric case, we know the comparibility matrix is not transitive
        (since \begin{math}x\end{math} and \begin{math}y\end{math} were found incomparable in it) and stop;
  \item if some third element is less or more than \begin{math}x\end{math} and incomparable with \begin{math}y\end{math},
        or the symetric case,
        then we want to switch to state ``advantage to \begin{math}x\end{math}'' (resp. \begin{math}y\end{math}):
        if we already have ``advantage to \begin{math}y\end{math}'' (resp. \begin{math}x\end{math}), we know it is not an (itov) order and stop;
        otherwise, we set the state and proceed to the next column;
  \item in all other cases, we proceed to the next column.
  \item if we do not stop prematurely, we know that property (i) of (itov) is satisfied for \begin{math}x,y\end{math}.
  \end{itemize}
\end{itemize}
Hence we obtain an \begin{math}O(n^3)\end{math} algorithm for finite (itov) orders membership testing.

We shall try to obtain other algorithmic results, but we need more structural results.
Our structural results will be nice, but the algorithmic results will be much less efficient
than classical results on series parallel orders.

\begin{definition}[Maximum chain, height]
Let \begin{math}P\end{math} be an order,
a chain of \begin{math}P\end{math} is \emph{maximum}
if it is maximal and no other chain of \begin{math}P\end{math} has greater cardinality.
The cardinal of a maximum chain is the \emph{height} of \begin{math}P\end{math},
denoted \begin{math}\Height(P)\end{math}.
When \begin{math}P\end{math} is well-founded\footnote{
Some authors also say Noetherian. In both cases, it means that there is no strictly decreasing infinite sequence.
},
we redefine a maximum chain to be one such that the corresponding ordinal is maximum;
and we redefine its height to be the ordinal corresponding to its maximum chains.
Thus in this case \begin{math}\Height(P)\end{math} denotes an ordinal.
\end{definition}
Note that an infinite order may have no maximum chain, but it always have at least one maximal chain.
When there is no maximum chain, \begin{math}\Height(P)\end{math} is defined as
the supremum cardinal/ordinal of the cardinals/ordinals corresponding to maximal chains.

\begin{definition}[Trunk]
Let \begin{math}P\end{math} be an order,
a \emph{trunk} \begin{math}T\end{math} of \begin{math}P\end{math}
is an induced suborder of \begin{math}P\end{math} such that:
\begin{enumerate}
\item[(trunk:i)] \begin{math}\forall x,y,z \in T, x \sim y \text{ and } y \sim z\end{math} implies \begin{math}x \sim z\end{math}.
\end{enumerate}
Moreover a trunk is said to be \emph{full} if
\begin{enumerate}
\item[(trunk:ii)] \begin{math}T\end{math} contains at least one maximum chain of \begin{math}P\end{math}.
\end{enumerate}
Note that a chain or an antichain are trunks.
A trunk is said to be \emph{maximal} if it is not contained in another trunk.
A trunk is said to be \emph{maximum} if it is the only maximal trunk of \begin{math}P\end{math}.
A full trunk is said to be \emph{relatively maximum} if it is the only maximal full trunk of \begin{math}P\end{math}.
\end{definition}

Note that a maximal (resp. maximum) chain is a trunk (resp. full trunk)
but this trunk may neither be maximal, nor maximum.

Clearly, be a trunk (trunk:i) is an hereditary property, any induced suborder of a trunk is a trunk.
Hence, a trunk is maximal when no individual element can be added to it and still obtain a trunk.
It does not matter whether one tries to add elements one at a time or many at once to a trunk.
Observe that (trunk:i) is equivalent to exclude the induced suborder \begin{math}O_{obst} = (\{a,b,c\}, \{a < b\}) \equiv \Inv(O_{obst})\end{math}.
Thus trunks have (itov) property.

Trunks are the name we gave to weak orders (or preorders),
see the article by \cite{DBLP:journals/dm/Trenk98} for a list of references on this topic.
We did not found a reference to what we name full trunks although.

We would like the reader to read the following lemma with the following example in mind:
Consider an order slightly similar to \begin{math}O_{obs1}\end{math} and \begin{math}O_{obs2}\end{math} with 4 elements:
two at the bottom and two at the top such that elements at the bottom are incomparable,
elements at the top are incomparable, and elements at the bottom are less than elements at the top.
Clearly this order is a (full-)trunk of itself.
Generalize by adding levels of two incomparable elements each between bottom and top elements.
Add one element that is less than one of the bottom elements (and thus is less than intermediate and top elements by transitivity),
to obtain a second order.
Draw a picture of both orders.
Is the first order still a full trunk in the second order?
\begin{lemma}
Let \begin{math}P\end{math} be a well-founded order.
Consider a level decomposition of \begin{math}P\end{math}
as a function \newline \begin{math}\Level: \Domain(P) \rightarrow \Height(P)\end{math}
(\begin{math}\Height(P)\end{math} is an arbitrary ordinal.).
An induced suborder \begin{math}T\end{math} of \begin{math}P\end{math} is a full trunk of \begin{math}P\end{math}
if and only if
\begin{enumerate}
\item[(trunk:i')] two elements of \begin{math}T\end{math} are ordered the way their levels are ordered
\newline
(\begin{math}\forall x,y \in T, (x \sim y \vee x = y) \Leftrightarrow \Level(x) = \Level(y),
 x < y \Leftrightarrow \Level(x) < \Level(y),
 x > y \Leftrightarrow \Level(x) > \Level(y)\end{math}),
\item[(trunk:ii')] and \begin{math}T\end{math} contains an element in each level of \begin{math}P\end{math}.
\end{enumerate}
\end{lemma}
\begin{demo}
Since (trunk:ii) holds, \begin{math}T\end{math} contains at least one maximum chain of \begin{math}P\end{math},
and it is clear that \begin{math}\forall h \in \Height(P), \exists x \in T \text{ such that } \Level(x) = h\end{math}.

Clearly, for any well-founded order
\begin{math}\forall x,y \in P, \Level(x) = \Level(y) \Rightarrow (x \sim y \vee x = y),
 x < y \Rightarrow \Level(x) < \Level(y),
 x > y \Rightarrow \Level(x) > \Level(y)\end{math}.
Thus, we only need to prove
\begin{math}\forall x,y \in T \text{ with } x \neq y, x \sim y \Rightarrow \Level(x) = \Level(y)
 ~ (\text{implying} \Level(x) < \Level(y) \Rightarrow x < y
 \text{ and } \Level(x) > \Level(y) \Rightarrow x > y\end{math}).

Assume for a contradiction that
\begin{math}\exists x,y \in T \text{ such that } x \sim y \text{ and } \Level(x) \neq \Level(y)\end{math}.
Let \begin{math}C\end{math} be a maximum chain of \begin{math}P\end{math} included in \begin{math}T\end{math}.
Let \begin{math}x_0 \in C  \text{ be such that } \Level(x) = \Level(x_0) \end{math},
and \begin{math}y_0 \in C  \text{ be such that } \Level(y) = \Level(y_0) \end{math}.
We have \begin{math}x \sim x_0 \vee x = x_0\end{math} and \begin{math}y \sim y_0 \vee y = y_0\end{math},
hence by (trunk:i), we have \begin{math}x_0 \sim y_0 \end{math},
contradicting the fact that \begin{math}C\end{math} is a chain.

As a consequence, we note that \begin{math}T\end{math} as a graph is connected,
unless \begin{math}P\end{math} is an order where all elements are incomparable.

The reciprocal is clear:
If there is an element in each level and these elements are ordered according to their level,
there is a maximum chain in \begin{math}T\end{math}.
If all elements are ordered according to their level, then (trunk:i) is immediate.
\end{demo}

Consider well-founded orders.
Note that (trunk:i) does not imply (trunk:i'):
Consider an antichain in \begin{math}P\end{math} containing elements in distinct levels.
To sum up, we have:
\begin{itemize}
\item (trunk:i') \begin{math}\Rightarrow\end{math} (trunk:i),
\item (trunk:ii) \begin{math}\Rightarrow\end{math} (trunk:ii')
\end{itemize}
and the other implications are all false.

(full trunk = (trunk:i) and (trunk:ii) \begin{math}\Leftrightarrow\end{math} (trunk:i') and (trunk:ii')) is the previous lemma.\newline
Thus (trunk:i') and (trunk:ii) \begin{math}\Rightarrow\end{math} full trunck,\newline
and (trunk:i) and (trunk:ii) \begin{math}\Rightarrow\end{math} (trunk:i) and (trunk:ii').\newline
But full trunck also equals (trunk:i) and (trunk:ii) and (trunk:i') and (trunk:ii'),
hence full trunck \begin{math}\Rightarrow\end{math} (trunk:i') and (trunk:ii).

Unfortunately, (trunk:i) and (trunk:ii') \begin{math}\not\Rightarrow\end{math} full trunk,
indeed consider an order \begin{math}P\end{math} made of \begin{math}n\end{math} chains of length \begin{math}n - 1\end{math},
\begin{math}n \geq 2\end{math}.
Taking a diagonal of \begin{math}n\end{math} elements,
i.e. an element in each chain, each element at a different level,
yields an antichain that is a trunk intersecting each level but is not a full trunk.
Thus we have:
\begin{itemize}
\item full trunk = (trunk:i) and (trunk:ii) \begin{math}\Leftrightarrow\end{math} (trunk:i') and (trunk:ii') \begin{math}\Leftrightarrow\end{math} (trunk:i') and (trunk:ii),
\item full trunk \begin{math}\Rightarrow\end{math} (trunk:i) and (trunk:ii').
\end{itemize}
All subsets of at least 3 properties among (trunk:i), (trunk:ii), (trunk:i'), and (trunk:ii')
are equivalent to full trunks.
The name trunk comes from the drawing of (induced sub-)orders with property (trunk:i'),
but unfortunately this property cannot be defined outside of well-founded orders.

Order isomorphism and counting linear extensions of an order are two hard problems.
(Order isomorphism is equivalent to graph isomorphism,
 for which neither \np-completeness proof, nor polynomial time algorithm is known.
Counting linear extensions of an order is \diesep-complete, see \cite{DBLP:conf/stoc/BrightwellW91}.)

Finite trunks with \begin{math}n\end{math} elements have an encoding consisting of \begin{math}l\end{math} integers,
where \begin{math}l\end{math} is the number of levels in its decomposition.
Since \begin{math}l\end{math} may be equal to \begin{math}n\end{math},
this encoding has size \begin{math}\Theta(n \times \log(n))\end{math},
which is slightly more compact than \begin{math}\Theta(n^2)\end{math}, for arbitrary orders.

However, it is trivial to see that two finite trunks are isomorphic
if and only if they have the same number of elements in each level.
Thus given two trunks as \begin{math}T = (t_0, \dots, t_l)\end{math} and \begin{math}U = (u_0, \dots, u_{l'})\end{math},
one can decide if they are isomorphic by comparing \begin{math}l\end{math} versus \begin{math}l'\end{math},
and \begin{math}t_i\end{math} versus \begin{math}u_i, 0 \leq i \leq \min(l, l')\end{math} in time \begin{math}\Theta(n \times \log(n))\end{math}
for worst case complexity, which is linear in the size of the encoding.

It is also trivial to count the number of linear extensions of a trunk \begin{math}T = (t_0, \dots, t_l)\end{math}.
Clearly this is equal to \begin{math}\prod_{i=0}^{i=l} (t_i!)\end{math},
which can be computed in time \begin{math}O(n \times \MultiplyCost(n\log(n)))\end{math},
where \begin{math}\MultiplyCost(p)\end{math} denotes the time complexity of multiplication of integers
 of \begin{math}p\end{math} bits.
(Currently the best asymptotic upper bound known for \begin{math}\MultiplyCost(p)\end{math}
is \begin{math}p \times \log(p) \times 4^{\log^{*}(p)}\end{math} (see \cite{DBLP:journals/corr/abs-1802-07932}).
Thus this algorithm is almost quadratic in the size of the encoding.)

We shall try to generalize these results to orders with property (itov) (and even a superclass of (itov) orders).
But first we note that computing a level decomposition of an order can be done in time \begin{math}O(n^3)\end{math}.
Testing if an order given with its level decomposition is a trunk can be done in time \begin{math}O(n^2)\end{math};
indeed it is sufficient to consider the induced suborder made of two levels, for each couple of consecutive levels \begin{math}(i,i+1)\end{math},
and verify that the number of arcs/comparability relationship is equal to \begin{math}t_i \times t_{i+1}\end{math}.
Thus membership in the class of trunks can be tested in time \begin{math}O(n^3)\end{math}.
\cite{DBLP:journals/dm/Trenk98} gives an \begin{math}O(n^2)\end{math} algorithm for this.
A faster algorithm is known, we will mention it at the end of this section.

\begin{theorem}
\label{theorem:isomorphic_maximal_chains}
All maximal chains in a trunk are isomorphic (maximal relatively to the trunk).
In particular they have the same cardinal.
In a full trunk of an order \begin{math}P\end{math}, any element belongs to a chain isomorphic to a maximum chain
(but this chain may not be maximal in \begin{math}P\end{math} if \begin{math}P\end{math} is infinite;
we will give an example after Proposition~\ref{proposition:exists_rmf_trunk}).
\end{theorem}
\begin{demo}
Remark that any element of a trunk that is not in a chain belonging to the trunk
is incomparable with at most one element of the chain by (trunk:i).
Moreover, if the chain is maximal for inclusion,
no element of the trunk outside the chain may be ordered with all its elements.
Hence, any element of a trunk that is not in a maximal chain belonging to the trunk
is incomparable with exactly one element of the chain by (trunk:i).
Thus given a trunk \begin{math}T\end{math}, one of its maximal chains \begin{math}C \subseteq T\end{math},
one can define ``locally to the trunk'' a notion of level.
In the case of orders that are not well-founded, this notion of level is only relative to \begin{math}C\end{math},
and there is nothing we can do in order to obtain an ordinal or a ``negative ordinal'' below some origin
(Think about the trunk \begin{math}\zz + \zz\end{math}).
Let us denote this level by \begin{math}\Level_{T,C}(x), \forall x \in T\end{math}.
From (trunk:i) again, all elements inside a level are incomparable
(\begin{math}x \sim \Level_{T,C}(x), y \sim \Level_{T,C}(y), \Level_{T,C}(x) = \Level_{T,C}(y) \Rightarrow x \sim y\end{math}).
Moreover, if two elements in distinct levels would be incomparable,
then, by (trunk:i) their levels would also be incomparable.
Thus two elements in distinct levels are comparable.
They cannot be ordered differently of their level (assume for a contradiction that
\begin{math}\Level_{T,C}(x) < \Level_{T,C}(y)\end{math} and \begin{math}x > y\end{math};
if \begin{math}\Level_{T,C}(y) < x\end{math}, then \begin{math}\Level_{T,C}(x) < x\end{math}, a contradiction;
but \begin{math}x > y\end{math} and \begin{math}\Level_{T,C}(y) > x\end{math} implies \begin{math}\Level_{T,C}(y) > y\end{math},
a contradiction).
\end{demo}
Thanks to this theorem, we can talk about \begin{math}\Level_{T}(x), \forall x \in T\end{math}
 instead of \begin{math}\Level_{T,C}(x), \forall x \in T\end{math},
it is understood that the value \begin{math}\Level_{T}(x)\end{math} is an equivalence class over elements of \begin{math}T\end{math}.
Of course this equivalence class is one equivalence class of the relation of incomparability
(that is symetric by definition and transitive by (trunk:i)).
We will occasionally use \begin{math}\Level(x)\end{math}
to denote the set of elements with the same level than \begin{math}x\end{math}.
Formally we should write \begin{math}\Level^{-1}(\Level(x))\end{math}
(this is not the identity since \begin{math}\Level\end{math} is not 1-1 outside of total orders).

This theorem also shows that finite weak orders/trunks correspond to the finite orders with Jordan-Dedekind chain condition,
see the article by \cite{Linial1985}.
It is trivial to see that infinite well-orders may have the Jordan-Dedekind chain condition without being weak orders/trunks.

A relatively maximum full trunk may not be a maximum trunk.
Indeed consider a chain of length 3 \begin{math}(x_0, x_1, x_2)\end{math}
with another element \begin{math}y\end{math} less than the maximum element \begin{math}x_2\end{math} of the chain.
Clearly the chain is a relatively maximum full trunk, since it is the only full trunk.
But the suborder induced by \begin{math}\{x_0,x_2,y\}\end{math} is also a maximal trunk.
We prove in the following corollary that if a trunk is maximum, it is a full trunk.

\begin{corollary}
An order \begin{math}P\end{math} has a maximum trunk if and only if it is itself a ((relatively maximum) full) trunk.
\end{corollary}
\begin{demo}
Only the forward implication is almost non-trivial.
Consider the contrapositive.
Let us assume that \begin{math}P\end{math} is not a trunk.
Then there are three elements \begin{math}x,y,z \in P \text{ such that } x \sim y, y \sim z, \text{ but } x < z\end{math}.
Since any (induced sub-)order of size 2 is a trunk, there is at least three maximal trunks,
one containing each pair among \begin{math}x,y,z\end{math}
and excluding the third element.
The fact that a maximum trunk is a full trunk thus follows, since all maximal chains of the order are isomorphic,
and define the same cardinal/ordinal, hence are maximum chains.
Thus, a maximum trunk is a relatively maximum full trunk.
\end{demo}

\begin{remark}
\label{remark:max_chains_in_trunk}
A relatively maximum full trunk contains all maximum chains of an order.
\end{remark}

From Theorem~\ref{theorem:isomorphic_maximal_chains} and Remark~\ref{remark:max_chains_in_trunk},
we deduce:
\begin{corollary}
\label{corollary:relatively_maximum_full_trunk_as_an_union}
A relatively maximum full trunk is the union of all maximum chains of an order.
An order has a relatively maximum full trunk if and only if the union of all its maximum chains is a trunk.
\end{corollary}
\begin{demo}
Any element \begin{math}x \in T\end{math} belongs to a maximal chain \begin{math}C\end{math} relatively to the full trunk \begin{math}T\end{math},
and \begin{math}C\end{math} is isomorphic to a maximum chain \begin{math}C'\end{math}of the order;
thus \begin{math}C\end{math} is included in some maximum chain \begin{math}C''\end{math},
and \begin{math}x \in C''\end{math}.
\end{demo}

Let us denote \begin{math}\RMFTrunk(P)\end{math} the relatively maximum full trunk of an order \begin{math}P\end{math},
if it exists.
\begin{proposition}
\label{proposition:exists_rmf_trunk}
If a finite order \begin{math}P\end{math} has property (itov),
then it has a relatively maximum full trunk \begin{math}\RMFTrunk(P)\end{math}.
\end{proposition}
\begin{demo}
Clearly, an order that is an antichain has property (itov) and has a relatively maximum full trunk since it is a trunk.
Hence, we can assume for the rest of the proof that the orders are not antichains,
and that maximum chains in them have length at least one.

Let \begin{math}P\end{math} be an order with property (itov),
we already know that at most one connected component is not reduced to a singleton.
We will prove that the union of its maximum chains is a trunk,
and by Corollary~\ref{corollary:relatively_maximum_full_trunk_as_an_union} conclude that
it has a relatively maximum full trunk.
Let \begin{math}T\end{math} be a maximal full trunk, it is an union of maximum chains
(in a finite order, a chain isomorphic to a maximum chain is a maximum chain).
Assume, for a contradiction, that there exists \begin{math}x \in P \setminus T\end{math},
and \begin{math}x \in C\end{math} a maximum chain of \begin{math}P\end{math}.
Since \begin{math}x\end{math} is not in \begin{math}T\end{math},
and \begin{math}T\end{math} is maximal,
there must exist \begin{math}y \in T\end{math}
such that \begin{math}x \sim y\end{math} and \begin{math}\Level(x) \neq \Level(y)\end{math}.
Let \begin{math}C'\end{math} be a maximum chain of \begin{math}P\end{math} included in \begin{math}T\end{math} containing \begin{math}y\end{math}.
Consider \begin{math}\{x'\} = (\Level(x) \cap C')\end{math}, and \begin{math}\{y'\} = (\Level(y) \cap C)\end{math}.
We observe that \begin{math}x,x',y,y'\end{math} are all distinct
(\begin{math}x,x'\end{math} and \begin{math}y,y'\end{math} are in distinct levels.
\begin{math}x\end{math} is outside of \begin{math}T\end{math}
but \begin{math}x'\end{math} is in \begin{math}C'\end{math} included in \begin{math}T\end{math}.
\begin{math}y \neq y'\end{math} because \begin{math}x \sim y\end{math} by definition of \begin{math}y\end{math},
and \begin{math}x,y' \in C\end{math} so \begin{math}x \not\sim y'\end{math}.).
We have \begin{math}x \sim y, x \sim x', y \sim y'\end{math},
and \begin{math}x \not\sim y'\end{math}, \begin{math}x' \not\sim y\end{math}.
Hence, clearly, the induced suborder \begin{math}x,x',y,y'\end{math} is isomorphic to one of the two obstructions:
\begin{math}O_{obs1} = (\{a,b,c,d\}, \{a < b, c < d\}) \equiv \Inv(O_{obs1})\end{math} if and only if \begin{math}x' \sim y'\end{math},
or \begin{math}O_{obs2} = (\{a,b,c,d\}, \{a < b, c < d, c < b\}) \equiv \Inv(O_{obs2})\end{math} if and only if \begin{math}x' \not\sim y'\end{math}.
It contradicts the fact that \begin{math}P\end{math} has property (itov);
hence \begin{math}T\end{math} is a maximal full trunk that contains all maximum chains
and we have the desired relatively maximum full trunk.
\end{demo}

Clearly the previous proposition is false for infinite (well-)orders.
Indeed \nn~ together with an element \begin{math}a\end{math} that is less than integers at least 2,
and incomparable with 0 and 1 has property (itov) since the neighbourhood of \begin{math}a\end{math}
is included in the neighbourhood of 0 and the neighbourhood of 1.
However, it contains three maximal full trunks, namely \nn~ (a chain),
\begin{math}\{a,1,2,3,...\}\end{math} (the union of two chains),
 and \begin{math}\{a,0,2,3,...\}\end{math} (the union of two chains).

\begin{openproblem}
With the previous counter-example, it is clear that \nn~ is more legitimate as a candidate for
relatively maximum full trunk than \begin{math}\{a,1,2,3,...\}\end{math}, because
\begin{math}\{a,1,2,3,...\}\end{math} contains the chain \begin{math}\{1,2,3,...\}\end{math},
that is maximal in \begin{math}\{a,1,2,3,...\}\end{math} but is not maximal in \begin{math}\nn \cup \{a\}\end{math}.
It suggests that, although all the maximal chains in a trunk are isomorphic,
the choice of the maximum chain(s) of the order that will be included in a full trunk are not equivalent,
and can probably be ordered such that there are full trunks that are ``relatively more maximum''
at least in some certain classes of orders.
We let as an open problem the choice of the good definition(s) for ``relatively more maximum'',
and the definition of the corresponding classes of orders.
\end{openproblem}

Relatively maximum full trunks are an important element of structure
but alone they say nothing about the rest of the order,
since it is possible to have such a trunk like a chain incomparable with any induced suborder,
provided that this trunk is sufficiently high to mask the suborder.
Hence, we will study the nature of the join of full trunks with other elements in (itov) orders.

\begin{definition}[regular to a trunk]
Let \begin{math}T\end{math} be a trunk of an order \begin{math}P\end{math},
an element \begin{math}x \in P\end{math} is said to be \emph{regular to} \begin{math}T\end{math},
when \begin{math}\forall y,z \in T\end{math} such that \begin{math}\Level_T(y) = \Level_T(z)\end{math}:
\begin{itemize}
\item \begin{math}x < y \Leftrightarrow x < z\end{math},
\item \begin{math}x > y \Leftrightarrow x > z\end{math},
\item \begin{math}x \sim y \vee x = y\Leftrightarrow x \sim z \vee x = z\end{math}.
\end{itemize}
An element regular to a trunk is ordered uniformly with each level of the trunk.
If \begin{math}x \not\in T\end{math}, it is equivalent to \begin{math}\OrderFunction(P)(x,y) = \OrderFunction(P)(x,z)\end{math}.
We could also express it in all cases by \begin{math}\Identify(\OrderFunction(P)(x,y), =, \sim) = \Identify(\OrderFunction(P)(x,z), =, \sim)\end{math},
in order to consider equality as a case of incomparability,
where the \begin{math}\Identify\end{math} function is the identity function of the first parameter
everywhere except for when the first parameter is equal to the second parameter,
in which case it is mapped to the third parameter.
\end{definition}
Clearly, each element of a trunk is regular to the trunk.
Moreover, an element regular to a trunk partitions the levels of the trunk
in at most three consecutives sets: the levels that are below the element,
the levels that are incomparable with the element,
and the levels that are above the element.

Not all elements must be regular to any maximal trunk in an (itov) order
(there is an (itov) order with 3 elements demonstrating this,
the reader should be able to find it :P).

\begin{lemma}
If a well-founded order has property (itov), then any element is regular to any full trunk.
\end{lemma}
\begin{demo}
Let \begin{math}P\end{math} be a well-founded order with property (itov),
and \begin{math}T\end{math} be a full trunk of \begin{math}P\end{math}.
Let us now assume for a contradiction that \begin{math}P \setminus T\end{math}
contains an element \begin{math}x\end{math} that is not regular to \begin{math}T\end{math}.
Let \begin{math}y,z \in T, y \neq z\end{math} be such that \begin{math}\Level(y) = \Level(z)\end{math}
and \begin{math}\OrderFunction(P)(x,y) \neq \OrderFunction(P)(x,z)\end{math}.
If \begin{math}x \not\sim y \text{ and } x \not\sim z\end{math},
then by transitivity \begin{math}y < z \text{ or } z < y\end{math}, a contradiction.
Thus, without loss of generality, \begin{math}x \not\sim y \text{ and } x \sim z\end{math},
and \begin{math}\Level(x) \neq \Level(y) = \Level(z)\end{math}.
Let \begin{math}x' \in (\Level(x) \cap T)\end{math} (this is where we use the fact that the trunk is full).
We have \begin{math}\OrderFunction(P)(x,y) = \OrderFunction(P)(x',y)\end{math},
because \begin{math}x \not\sim y \text{ and } x' \not\sim y\end{math},
and the levels of \begin{math}x,x'\end{math} are equals.
\newline Moreover \begin{math}\OrderFunction(P)(x',y) = \OrderFunction(P)(x',z)\end{math},
and \begin{math}x' \not\sim z\end{math}.
Clearly, the induced suborder \begin{math}x,x',y,z\end{math} is isomorphic to the obstruction
\begin{math}O_{obs2}\end{math}.
\end{demo}

The reciprocal is false, as demonstrates \begin{math}O_{obs1}\end{math}.
Equivalently, one can observe that the property that any element is regular to any full trunk
is not hereditary (consider an order \begin{math}P\end{math} with two connected components: a chain of size 3,
and \begin{math}O_{obs2}\end{math}; if you remove an element in the chain,
then full trunks in \begin{math}O_{obs2}\end{math} become full trunks in \begin{math}P\end{math}).
Hence we shall seek another regularity to characterize well-founded (itov) order.

Observe that in the proofs of the two preceding proposition and lemma,
we found \begin{math}O_{obs1}\end{math} and \begin{math}O_{obs2}\end{math} obstructions
with an additional constraint: each level of the obstruction was included
in a level of the bigger order,
 and no two levels of the obstruction were included in the same level of the bigger order.
A little thought about this shows that there are at least 4 kinds of suborders:
\begin{itemize}
\item A suborder \begin{math}P'\end{math} of an order \begin{math}P\end{math}
      is such that \begin{math}\Domain(P') \subseteq \Domain(P)\end{math},
      \begin{math}\forall x,y \in \Domain(P'), x <_{P'} y \Rightarrow x <_{P} y\end{math},
      and \begin{math}\forall x,y \in \Domain(P'), x >_{P'} y \Rightarrow x >_{P} y\end{math}.
      Since any partial order is a suborder of a total order and we cannot say that total orders do not have structure,
      it is clear that the general notion of suborder is too weak.
\item An induced suborder \begin{math}P'\end{math} of an order \begin{math}P\end{math}
      is such that \begin{math}\Domain(P') \subseteq \Domain(P)\end{math},
      and \begin{math}\forall x,y \in \Domain(P'), \OrderFunction(P')(x,y) = \OrderFunction(P)(x,y)\end{math}.
      This is what we used so far (in preliminary versions of this article,
      we just wrote suborder instead of induced suborder, because the remark at the end of the previous item was obvious,
      and we did not think that it would be better to separate suborder and induced suborder,
      like it is customary to distinguish subgraph and induced subgraph in graph theory).
\item A \emph{level-induced suborder} \begin{math}P'\end{math} of a well-founded order \begin{math}P\end{math}
      is such that \begin{math}\Domain(P') \subseteq \Domain(P)\end{math},
      \begin{math}\forall x,y \in \Domain(P'), \OrderFunction(P')(x,y) = \OrderFunction(P)(x,y)\end{math},
      and \begin{math}\forall x,y \in \Domain(P'), \Level_{P'}(x) = \Level_{P'}(y) \Leftrightarrow \Level_{P}(x) = \Level_{P}(y)\end{math}.
      This is the appropriate notion of suborder for the two previous proposition and lemma.
      (Note that we could also define two other kinds of level-induced suborder with
      \begin{math}\forall x,y \in \Domain(P'), \Level_{P'}(x) = \Level_{P'}(y) \Rightarrow \Level_{P}(x) = \Level_{P}(y)\end{math}
      \begin{math}\forall x,y \in \Domain(P'), \Level_{P'}(x) = \Level_{P'}(y) \Leftarrow \Level_{P}(x) = \Level_{P}(y)\end{math}.
      There is a simple proof by transfinite induction on the levels of \begin{math}P'\end{math} showing that
      \begin{math}\forall x,y \in \Domain(P'), \Level_{P'}(x) = \Level_{P'}(y) \Rightarrow \Level_{P}(x) = \Level_{P}(y)\end{math}
      implies \begin{math}\forall x,y \in \Domain(P'), \Level_{P'}(x) = \Level_{P'}(y) \Leftarrow \Level_{P}(x) = \Level_{P}(y)\end{math}.
      Moreover, the same proof shows that 
      \begin{math}\forall x,y \in \Domain(P'), \Level_{P'}(x) < \Level_{P'}(y) \Leftrightarrow \Level_{P}(x) < \Level_{P}(y)\end{math}.
      Thus only two kinds of level-induced suborder exists
      \begin{math}\forall x,y \in \Domain(P'), \Level_{P'}(x) = \Level_{P'}(y) (\Leftrightarrow \text{ or } \Rightarrow) \Level_{P}(x) = \Level_{P}(y)\end{math},
      and \begin{math}\forall x,y \in \Domain(P'), \Level_{P'}(x) = \Level_{P'}(y) \Leftarrow \Level_{P}(x) = \Level_{P}(y)\end{math}.  )
\item A \emph{consecutive level-induced suborder} \begin{math}P'\end{math} of a well-founded order \begin{math}P\end{math}
      is such that \begin{math}\Domain(P') \subseteq \Domain(P)\end{math},
      \begin{math}\forall x,y \in \Domain(P'), \OrderFunction(P')(x,y) = \OrderFunction(P)(x,y)\end{math},
      \begin{math}\forall x,y \in \Domain(P'), \Level_{P'}(x) = \Level_{P'}(y) \Leftrightarrow \Level_{P}(x) = \Level_{P}(y)\end{math},
      and \begin{math}\forall x,y \in \Domain(P'), \Level_{P'}(x) + 1 = \Level_{P'}(y) \Leftrightarrow \Level_{P}(x) + 1 = \Level_{P}(y)\end{math}.
\end{itemize}

Let us restate the previous results.
\setcounter{definition}{7}
\begin{proposition}
If a finite order \begin{math}P\end{math} has no level-induced suborder isomorphic to
\begin{math}O_{obs1}\end{math} or \begin{math}O_{obs2}\end{math},
then it has a relatively maximum full trunk \begin{math}\RMFTrunk(P)\end{math}.
\end{proposition}
\setcounter{definition}{10}
\begin{lemma}
If a well-founded order has no level-induced suborder isomorphic to \begin{math}O_{obs2}\end{math},
then any element is regular to any full trunk.
\end{lemma}

We continue with another notion of regularity for well-founded orders.

\begin{definition}[up-regular in a well-founded order]
Let \begin{math}P\end{math} be a well-founded order,
an element \begin{math}x \in P\end{math} is said to be \emph{up-regular in} \begin{math}P\end{math},
 relatively to a subset \begin{math}S\end{math} of \begin{math}P\end{math},
when \begin{math}\forall y,z \in S\end{math} such that \begin{math}\Level(x) < \Level(y) = \Level(z)\end{math}:
\begin{displaymath}\OrderFunction(P)(x,y) = \OrderFunction(P)(x,z).\end{displaymath}
If \begin{math}S = P\end{math}, we just say \emph{up-regular in} \begin{math}P\end{math}.
We say that a well-founded order \begin{math}P\end{math} is \emph{up-regular} if any element of \begin{math}P\end{math} is \emph{up-regular in} \begin{math}P\end{math}.
\end{definition}

\begin{lemma}
A well-founded order is up-regular if and only if
it excludes level-induced suborders isomorphic to
\begin{math}O_{obs1}\end{math} or \begin{math}O_{obs2}\end{math}.
\end{lemma}
\begin{demo}
Let \begin{math}P\end{math} be a well-founded order.
We first prove that if it excludes level-induced suborders isomorphic to
\begin{math}O_{obs1}\end{math} or \begin{math}O_{obs2}\end{math},
then it is up-regular.
By contrapositive, assume that \begin{math}x \in \Domain(P)\end{math} is not up-regular.
Without loss of generality, let \begin{math}y,z \in \Domain(P)\end{math} be such that \begin{math}\Level(x) < \Level(y) = \Level(z)\end{math},
and \begin{math}x < y, x \sim z\end{math}.
There exists a chain \begin{math}C\end{math} containing \begin{math}z\end{math} and intersecting each level below \begin{math}\Level(z)\end{math}.
Since \begin{math}x \sim z\end{math}, \begin{math}x \not\in C\end{math}.
Let \begin{math}\{x'\} = (C \cap \Level(x))\end{math}.
Clearly, \begin{math}x,y,z,x'\end{math} are all distinct, \begin{math}x < y, x' < z, x \sim x', y \sim z, x \sim z\end{math}.
Thus we have one of the two obstructions, whatever the choice between \begin{math}x' < y\end{math} or \begin{math}x' \sim y\end{math}.

We now prove that if \begin{math}P\end{math} is up-regular,
then it excludes level-induced suborders isomorphic to
\begin{math}O_{obs1}\end{math} or \begin{math}O_{obs2}\end{math}.
Again by contrapositive, assume that there is a level-induced suborder isomorphic to
\begin{math}O_{obs1}\end{math} or \begin{math}O_{obs2}\end{math},
it is trivial to see that at least one element in the bottom level of this level-induced
suborder is not up-regular to the top level of this level-induced
suborder, hence not-regular to the corresponding level of \begin{math}P\end{math}.
\end{demo}

\begin{corollary}
If a well-founded order has property (itov), then it is up-regular.
\end{corollary}

The reciprocal is false. Indeed the following order is up-regular:
\begin{math}(\{b_1, b_1', b_2, t_1, t_1', t_2\},\newline
      \{b_1 < t_1, b_1 < t_1', b_1 < t_2, b_1' < t_1, b_1' < t_1', b_1' < t_2, t_1' < t_2, b_2 < t_2\})\end{math}.
The first three elements are on level 0, the last element is on level 2, and the two others are on level 1.
\begin{math}\{b_1, b_2, t_1,t_2\}\end{math} yields the desired obstruction.
A close look at this example suggests the following lemma.

\begin{lemma}
If a well-founded order is up-regular, then it does not have an induced suborder isomorphic to \begin{math}O_{obs1}\end{math}.
\end{lemma}
\begin{demo}
Let \begin{math}P\end{math} be a well-founded order that is up-regular.
Assume for a contradiction that it contains an obstruction \begin{math}O_{obs1}\end{math}.
Since the bottom elements \begin{math}b_1,b_2\end{math} are not regular
to the top elements \begin{math}t_1, t_2\end{math} in \begin{math}O_{obs1}\end{math},
both top elements are in distinct levels.
Assume without loss of generality that \begin{math}b_1 < t_1, b_2 < t_2, \Level(t_1) < \Level(t_2), b_1 \sim b_2, t_1 \sim t_2\end{math}.
We must have \begin{math}b_1 \sim t_2\end{math} in \begin{math}O_{obs1}\end{math} but we must also have \begin{math}b_1 < t_2\end{math}
since by up-regularity \begin{math}b_1 < t_1'\end{math},
where \begin{math}t_1'\end{math} belongs to a maximal chain between level 0 and \begin{math}\Level(t_2)\end{math} containing \begin{math}t_2\end{math}.
\end{demo}

The reciprocal is false, as demonstrates \begin{math}O_{obs2}\end{math}.
Together with the example before this lemma, it proves that up-regular is not an hereditary property.
There are induced suborders of up-regular orders that are not up-regular.

We observe that the previous lemma is true, although there is an order with an induced suborder isomorphic
to \begin{math}O_{obs1}\end{math}, but no level-induced suborder isomorphic to \begin{math}O_{obs1}\end{math}:
\begin{math}(\{b_1, b_2, b_3, m_2, t_1, t_2, u_3\},\newline
      \{b_1 < t_1, b_2 < m_2, b_2 < t_1, b_2 < t_2, b_2 < u_3,
        m_2 < t_1, m_2 < t_2, m_2 < u_3,
        t_2 < u_3, b_3 < u_3
       \})\end{math}
(letters give level b,m,t,u, indices give column 1,2,3).

\begin{openproblem}
Characterize finite orders that are induced suborders of any well-founded order if and only if
they are (consecutive) level-induced suborders of this well-founded order.
Examples: chains, antichains of size 1 and 2.
Counter-examples: antichains of size at least 3.
\end{openproblem}

Clearly, if we do not require that \begin{math}\Level(x) < \Level(y) = \Level(z)\end{math}, and define
\emph{regular} with \begin{math}\Level(x) \neq \Level(y) = \Level(z)\end{math},
or \emph{down-regular} with \begin{math}\Level(x) > \Level(y) = \Level(z)\end{math},
 the (down-) regular well-founded orders are the well-founded trunks.
Thus, (down-)regular is hereditary whilst up-regular is not.

Up-regular orders have another nice characterization.
\begin{theorem}
A well-founded order \begin{math}P\end{math} is up-regular
if and only if for any trunk \begin{math}T\end{math},
if there is a highest level \begin{math}L\end{math} of \begin{math}P\end{math}
that \begin{math}T\end{math} intersects,
then \begin{math}T \cup L\end{math} is a trunk.
In particular, in an up-regular-order \begin{math}P\end{math},
any maximal trunk \begin{math}T\end{math} contains the highest level \begin{math}L\end{math} of \begin{math}P\end{math}
that it intersects, if such a level is defined.
\end{theorem}
\begin{proof}
For the ``if'' part, we prove the contrapositive;
assume that \begin{math}P\end{math} is not up-regular:
\begin{math}x < y\end{math} and \begin{math}x \sim z\end{math} 
with \begin{math}\Level(y) = \Level(z)\end{math},
then \begin{math}T = \{x,y\}\end{math} is a trunk,
and clearly \begin{math}T \cup \Level(y)\end{math} is not a trunk
(in \begin{math}T \cup \Level(y)\end{math}, both \begin{math}x\end{math} and \begin{math}z\end{math}
would belong to the bottom level and \begin{math}y\end{math} belong to the top level
but \begin{math}y\end{math} would not be regular to the bottom level then).

Now we prove the ``only-if'' part.
Assume that \begin{math}P\end{math} is up-regular.
Let \begin{math}T\end{math} be a trunk,
\begin{math}L\end{math} be the highest level of \begin{math}P\end{math}
that it intersects,
and assume for a contradiction that \begin{math}T \cup L\end{math}
is not a trunk.
Let \begin{math}x,y,z \in (T \cup L)\end{math} be such that
\begin{math}x \sim y, y \sim z, x < z\end{math}.
Either one or two elements among \begin{math}x,y,z\end{math} belong to \begin{math}L\end{math}
(both \begin{math}T\end{math} and \begin{math}L\end{math} are trunks),
moreover \begin{math}x\end{math} cannot belong to \begin{math}L\end{math}
since it is less than \begin{math}z \in (T \cup L)\end{math}.
\begin{itemize}
\item If \begin{math}y \in L\end{math}, and \begin{math}x, z \in (T \setminus L)\end{math},
      then \begin{math}x \sim y, y \sim z, x < z\end{math} implies
      by up-regularity that \begin{math}x \sim y', y' \sim z\end{math}
      for any element \begin{math}y' \in (T \cap L)\end{math}
      contradicting the fact that \begin{math}T\end{math} is a trunk.
\item If \begin{math}z \in L\end{math}, and \begin{math}x, y \in (T \setminus L)\end{math},
      then \begin{math}x \sim y, y \sim z, x < z\end{math} implies
      by up-regularity that \begin{math}y \sim z', x < z'\end{math}
      for any element \begin{math}z' \in (T \cap L)\end{math}
      contradicting the fact that \begin{math}T\end{math} is a trunk.
\item If \begin{math}y,z \in L\end{math}, and \begin{math}x \in (T \setminus L)\end{math},
      then \begin{math}x \sim y, y \sim z, x < z\end{math} contradicts up-regularity
      of \begin{math}x\end{math} with respect to level \begin{math}L\end{math}.
\end{itemize}
\end{proof}

It is tempting to hope that more structural results on maximal trunks holds in (itov) or up-regular orders.
(Like ``Either \begin{math}T\end{math} has height 1 (all its elements are incomparable),
or \begin{math}T = \{x \in P \text{ such that } x < y \in L\} \cup L\end{math}.'')
The following example limits such structural results, even for (itov) orders:
Consider the order with 4 elements on level 0 (\begin{math}x_{0,0}\end{math} to \begin{math}x_{0,3}\end{math}),
3 elements on level 1 (\begin{math}x_{1,0}\end{math} to \begin{math}x_{1,2}\end{math}),
2 elements on level 2 (\begin{math}x_{2,0}\end{math} to \begin{math}x_{2,1}\end{math})
1 element on level 3 (\begin{math}x_{3,0}\end{math}),
such that \begin{math}x_{i,j} < x_{k,l}\end{math}
if and only if \begin{math}i < k\end{math} and \begin{math}j > 0\end{math}.
The following subsets are maximal trunks:
\begin{math}\{x_{0,0}, x_{1,0}, x_{2,0}, x_{2,1}\}\end{math},
and \begin{math}\{x_{0,1}, x_{0,2}, x_{0,3}, x_{1,0}, x_{2,0}, x_{2,1}\}\end{math}.

\begin{corollary}
For a well-founded order \begin{math}P\end{math}, the following properties are equivalent:
\begin{enumerate}[(i)]
\item \begin{math}P\end{math} is up-regular,
\item \begin{math}P\end{math} excludes level-induced suborders isomorphic to
\begin{math}O_{obs1}\end{math} and excludes level-induced suborders isomorphic to \begin{math}O_{obs2}\end{math},
\item \begin{math}P\end{math} excludes induced suborders isomorphic to
\begin{math}O_{obs1}\end{math} and excludes level-induced suborders isomorphic to \begin{math}O_{obs2}\end{math},
\item for any trunk \begin{math}T\end{math},
if there is a highest level \begin{math}L\end{math} of \begin{math}P\end{math}
that \begin{math}T\end{math} intersects,
then \begin{math}T \cup L\end{math} is a trunk.
\end{enumerate}
\end{corollary}

The two following results are corollaries of the previous results stated in terms of level-induced suborders.
\begin{proposition}
\label{proposition:exists_rmf_trunk2}
If a finite order \begin{math}P\end{math} is up-regular,
then it has a relatively maximum full trunk \begin{math}\RMFTrunk(P)\end{math}.
\end{proposition}

\begin{lemma}
If a well-founded order is up-regular, then any element is regular to any full trunk.
\end{lemma}
%\begin{demo}
%Let \begin{math}P\end{math} be a well-founded order that is up-regular,
%and \begin{math}T\end{math} be a full trunk of \begin{math}P\end{math}.
%Let us now assume for a contradiction that \begin{math}P \setminus T\end{math}
%contains an element \begin{math}x\end{math} that is not (down-)regular to \begin{math}T\end{math}.
%Let \begin{math}y,z \in T, y \neq z\end{math} be such that \begin{math}\Level(x) > \Level(y) = \Level(z)\end{math}
%and \begin{math}\OrderFunction(P)(x,y) \neq \OrderFunction(P)(x,z)\end{math}.
%Without loss of generality, \begin{math}x > y \text{ and } x \sim z\end{math}.
%Let \begin{math}x' \in (\Level(x) \cap T)\end{math} (this is where we use the fact that the trunk is full).
%Clearly, we have \begin{math}x \sim x', x' > y \text{ and } x' > z\end{math}.
%Hence the induced suborder \begin{math}x,x',y,z\end{math} is isomorphic to the obstruction
%\begin{math}O_{obs2}\end{math}.
%\end{demo}

We found the following classes of orders:
\begin{itemize}
\item general case: (trunks \begin{math}\subset\end{math} (itov) orders \begin{math}\subset\end{math}
      orders without obstruction \begin{math}O_{obs1}\end{math}),
      orders where any element is regular to any full trunk denoted \rft,
      and orders with a relatively maximum full trunk denoted \rmft;
      it is immediate to prove that trunks \begin{math}\subset\end{math} \rft,
      and trunks \begin{math}\subset\end{math} \rmft;
      we gave counter-examples to show that none of the other inclusion holds,
      except for (itov) orders \begin{math}\not\subseteq\end{math} \rft;
      we shall give the missing counter-example now:
        consider \zz~ together with a copy of
        \begin{math}\nn c = \{0c, 1c, 2c, \dots\}\end{math}
        and an additional element \{-0.5\}
        such that \begin{math}ic < j \Leftrightarrow i < j\end{math}
        except for \begin{math}-0.5 \sim 0c\end{math},
        clearly, it has (itov) property, excluding \{-0.5\} yields a full trunk,
        and -0.5 is not regular to it.
\item well-founded orders: (trunks \begin{math}\subset\end{math} (itov) orders \begin{math}\subset\end{math}
      up-regular orders \begin{math}\subset\end{math} orders without obstruction \begin{math}O_{obs1}\end{math}),
      trunks \begin{math}\subset\end{math} \rmft, up-regular orders \begin{math}\subset\end{math} \rft;
      all the (non-)inclusion holds also for well-orders since we did not use any infinite antichain in our counter-examples.
\item finite orders: (trunks \begin{math}\subset\end{math} (itov) orders \begin{math}\subset\end{math}
      up-regular orders \begin{math}\subset\end{math} orders without obstruction \begin{math}O_{obs1}\end{math}),
      up-regular orders \begin{math}\subset\end{math} \rmft, up-regular orders \begin{math}\subset\end{math} \rft.
\end{itemize}

The following figures are faithful, except maybe for the relations between intersections
of some of the considered classes of orders and some other classes.
\begin{figure}[h!]
  \centering
  \begin{tikzpicture}[scale=1.5]
    %\draw (0,0) node {origin};
    %\draw (1,1) node {1,1};
    %\draw (0,1) node {0,1};
    %\draw (1,0) node {1,0};

    %\draw (0cm,.5cm) ellipse (.25cm and .5cm);
    %est equivalent a
    %\draw (0,.5) ellipse (.25 and .5);

    %\draw (1,1) ellipse (.25 and .5);
    %\draw[shift={(1,1)},rotate=45] (0,0) ellipse (.25 and .5);
    %\draw[rotate=90] (1,1) ellipse (.25 and .5);

    \draw (0,.5) ellipse (.3 and .5);
    \draw (0,.5) node {trunks};

    \draw (0,.8) ellipse (.4 and .8);
    \draw (0,1.2) node {(itov)};

    \draw (0,1.4) ellipse (.5 and 1.4);
    \draw (0,2.5) node {\begin{math}O_{obs1}\end{math} excluded};

    \draw[shift={(0.5,1.235)},rotate=-30] ellipse (.6 and 1.4);
    \draw (0.9,1.7) node {\rft};

    \draw[shift={(-0.5,1.235)},rotate=30] ellipse (.6 and 1.4);
    \draw (-0.9,1.7) node {\rmft};
  \end{tikzpicture}
  \mycaption{Inclusion of some classes of orders}
\end{figure}

\begin{figure}[h!]
  \centering
  \begin{tikzpicture}[scale=1.5]
    \draw (0,.5) ellipse (.3 and .5);
    \draw (0,.5) node {trunks};

    \draw (0,.8) ellipse (.4 and .8);
    \draw (0,1.2) node {(itov)};

    \draw (0,1.4) ellipse (.5 and 1.4);
    \draw (0,2.5) node {up-regular};

    \draw (0,2.4) ellipse (.7 and 2.4);
    \draw (0,4) node {\begin{math}O_{obs1}\end{math} excluded};

    \draw[shift={(0.8,2)},rotate=-30] ellipse (1.4 and 2.4);
    \draw (1.2,1.7) node {\rft};

    \draw[shift={(-0.95,2)},rotate=30] ellipse (.6 and 2.4);
    \draw (-1.2,1.7) node {\rmft};
  \end{tikzpicture}
  \mycaption{Inclusion of some classes of well-founded orders}
\end{figure}

\begin{figure}[h!]
  \centering
  \begin{tikzpicture}[scale=1.5]
    \draw (0,.5) ellipse (.3 and .5);
    \draw (0,.5) node {trunks};

    \draw (0,.8) ellipse (.4 and .8);
    \draw (0,1.2) node {(itov)};

    \draw (0,1.4) ellipse (.5 and 1.4);
    \draw (0,2.5) node {up-regular};

    \draw (0,2.4) ellipse (.7 and 2.4);
    \draw (0,4) node {\begin{math}O_{obs1}\end{math} excluded};

    \draw[shift={(0.8,2)},rotate=-30] ellipse (1.4 and 2.4);
    \draw (1.2,1.7) node {\rft};

    \draw[shift={(-0.8,2)},rotate=30] ellipse (1.4 and 2.4);
    \draw (-1.2,1.7) node {\rmft};
  \end{tikzpicture}
  \mycaption{Inclusion of some classes of finite orders}
\end{figure}

\newpage

\begin{theorem}[decomposition of finite (itov) orders]
\label{theorem:finite_itov_decomposition}
A finite order \begin{math}P\end{math} has (itov) property if and only if
\begin{itemize}
  \item it has a relatively maximum full trunk \begin{math}\RMFTrunk(P)\end{math},
  \item \begin{math}P \setminus \RMFTrunk(P)\end{math} has (itov) property,
  \item each element of \begin{math}P \setminus \RMFTrunk(P)\end{math} is regular to \begin{math}\RMFTrunk(P)\end{math},
  \item and there is no obstruction \begin{math}O_{obs1}\end{math} or \begin{math}O_{obs2}\end{math},
        intersecting both \begin{math}\RMFTrunk(P)\end{math} and \begin{math}P \setminus \RMFTrunk(P)\end{math},
        with at least two elements in \begin{math}P \setminus \RMFTrunk(P)\end{math}.
\end{itemize}
\end{theorem}
\begin{demo}
The forward implication is a consequence of the previous results.

Clearly the backward implication is true if \begin{math}P\end{math} is an antichain.
Thus we can suppose that \begin{math}P\end{math} and \begin{math}\RMFTrunk(P)\end{math}
contain at least one comparability.

Assume for a contradiction that \begin{math}P\end{math} is an order without property (itov),
but it has a relatively maximum full trunk \begin{math}\RMFTrunk(P)\end{math},
\begin{math}P \setminus \RMFTrunk(P)\end{math} has property (itov),
and each element of \begin{math}P \setminus \RMFTrunk(P)\end{math} is regular to \begin{math}\RMFTrunk(P)\end{math}.
Clearly, \begin{math}P\end{math} contains an obstruction \begin{math}O_{obs}\end{math}
that is neither contained in \begin{math}\RMFTrunk(P)\end{math} nor contained in \begin{math}P \setminus \RMFTrunk(P)\end{math}.

Assume that \begin{math}O_{obs} \cap \RMFTrunk(P)\end{math} has size 3.
Let \begin{math}\{x\} = (O_{obs} \setminus \RMFTrunk(P))\end{math}.
We must have that both top elements (if \begin{math}x\end{math} is a bottom element)
or both bottom elements (if \begin{math}x\end{math} is a top element) of
\begin{math}O_{obs}\end{math} are in the same level of \begin{math}\RMFTrunk(P)\end{math},
since they are incomparable.
But then \begin{math}x\end{math} and the third element in \begin{math}O_{obs} \cap \RMFTrunk(P)\end{math}
must be regular to the two top/bottom elements,
since any element in \begin{math}P\end{math} is supposed to be regular to \begin{math}\RMFTrunk(P)\end{math}.
It is easy to see that in both \begin{math}O_{obs1}\end{math} and \begin{math}O_{obs2}\end{math},
one element at least on each level is not regular to the other level, a contradiction.
%
%The case \begin{math}O_{obs} \cap \RMFTrunk(P)\end{math} has size 2 or 1 is in contradiction with the last requirement of the theorem.
\end{demo}

It is tempting to hope to obtain a similar decomposition theorem for up-regular orders.
However up-regular is not hereditary, it is not even ``relatively maximum full trunk hereditary''.
In order to see this fact, we need to adapt the counter-example showing that up-regular is not hereditary.
The following order is up-regular:
\begin{math}(\{b_1, b_1', b_2, t_1, t_1', t_2, m_1', t_2', u_1'\},\newline
      \{b_1 < t_1, b_1 < t_1', b_1 < t_2, b_1 < t_2', b_1 < u_1',\newline
        b_1' < m_1', b_1' < t_1, b_1' < t_1', b_1' < t_2, b_1' < t_2', b_1' < u_1',\newline
        m_1' < t_1, m_1' < t_1', m_1' < t_2, m_1' < t_2', m_1' < u_1',\newline
        t_1' < t_2, t_1' < t_2', t_1' < u_1',\newline
        b_2 < t_2, b_2 < t_2', b_2 < u_1',\newline
        t_2' < u_1'\})\end{math};
It is clear that \begin{math}(b_1', m_1', t_1', t_2', u_1')\end{math} is the only maximum chain of this order,
hence it is its relatively maximum full trunk;
but removing it yields\begin{math}\{b_1, b_2, t_1,t_2\}\end{math} which is not up-regular.

%\begin{falsetheorem}[decomposition of finite up-regular orders]
%\label{falsetheorem:finite_upregular_decomposition}
%A finite order \begin{math}P\end{math} is up-regular if and only if
%\begin{itemize}
%  \item it has a relatively maximum full trunk \begin{math}\RMFTrunk(P)\end{math},
%  \item \begin{math}P \setminus \RMFTrunk(P)\end{math} is up-regular,
%  \item each element of \begin{math}P \setminus \RMFTrunk(P)\end{math} is regular to \begin{math}\RMFTrunk(P)\end{math},
%  \item each element of \begin{math}\RMFTrunk(P)\end{math} is up-regular in \begin{math}P\end{math} to \begin{math}P \setminus \RMFTrunk(P)\end{math}.
%\end{itemize}
%\end{falsetheorem}

Unfortunately, we did not found more regularity result:
\begin{math}P \setminus \RMFTrunk(P)\end{math} doesn't need to be (down-)regular;
elements in \begin{math}\RMFTrunk(P)\end{math} don't need to be (down-)regular
 in \begin{math}P\end{math} relatively to \begin{math}P \setminus \RMFTrunk(P)\end{math};
elements in \begin{math}\RMFTrunk(P)\end{math} don't need to be (down-)regular
 in \begin{math}P\end{math} relatively to \begin{math}\RMFTrunk(P \setminus \RMFTrunk(P))\end{math}.
Finding the appropriate examples is a simple exercise for the reader.

Still, there is something we can improve.
With Theorem~\ref{theorem:finite_itov_decomposition},
we may expect a \begin{math}O(n)\end{math} number of steps of recursive decomposition.
We will now show that it is not the case, and that at most \begin{math}\lceil\lg(n)\rceil\end{math} steps are sufficient.

\begin{lemma}
Consider a finite up-regular order \begin{math}P\end{math}
and its relatively maximum full trunk \begin{math}\RMFTrunk(P)\end{math}.
\begin{math}\Height(P \setminus \RMFTrunk(P)) \leq \frac{1}{2} \times \Height(P)\end{math}.
\end{lemma}
\begin{demo}
Consider \begin{math}C_{rmft}\end{math} a maximum chain of \begin{math}P\end{math},
the cardinal \begin{math}|C_{rmft}|\end{math} of \begin{math}C_{rmft}\end{math}
equals \begin{math}\Height(P)\end{math}.
Consider \begin{math}C_{out}\end{math} a maximum chain of \begin{math}P \setminus \RMFTrunk(P)\end{math},
\begin{math}|C_{out}| = \Height(P \setminus \RMFTrunk(P))\end{math}.
Assume for a contradiction that \begin{math}|C_{out}| > \frac{1}{2} \times |C_{rmft}|\end{math}.
Clearly,
\begin{itemize}
\item either \begin{math}C_{out}\end{math} intersects the highest level of \begin{math}P\end{math},
      but then the element in this intersection belongs to a maximum chain,
      and thus should be in \begin{math}\RMFTrunk(P)\end{math},
      which is impossible since \begin{math}\RMFTrunk(P)\end{math} and \begin{math}C_{out}\end{math} are disjoint;
\item or there are two consecutive levels of \begin{math}P\end{math} containing elements of \begin{math}C_{out}\end{math},
      and these two levels are not the highest level of \begin{math}P\end{math}.
      Let \begin{math}x < y \in C_{out}\end{math} be the corresponding elements.
      Let \begin{math}\{x'\} = (C_{rmft}\cap \Level(x))\end{math} and \begin{math}\{y'\} = (C_{rmft}\cap \Level(y))\end{math}.
      Clearly, \begin{math}x,y,x',y'\end{math} are all distinct, \begin{math}x < y, x' < y', x \sim x', y \sim y'\end{math}.
      Since, \begin{math}P\end{math} is up-regular, we must also have \begin{math}x < y', x' < y\end{math}.
      But then it is trivial to see that \begin{math}x\end{math} belongs to a maximum chain of \begin{math}P\end{math}, the desired contradiction.
\end{itemize}
\end{demo}

This lemma is optimal as the following family \begin{math}(P^{mrh}_i)_{i \in \nn}\end{math}
of (itov) orders shows (mrh stands for maximum recursive height):
\begin{itemize}
\item \begin{math}P^{mrh}_0 = (\{rmft_{0,1},rmft_{0,2},rmft_{0,3},rmft_{0,4},x_{0,1},x_{0,3}\},\newline
      \{rmft_{0,i} < rmft_{0,j} \text{ when } i < j, x_{0,1} < x_{0,3}, x_{0,1} < rmft_{0,3},
      x_{0,1} < rmft_{0,4}, rmft_{0,1} < x_{0,3}, rmft_{0,2} < x_{0,3}\})\end{math},
\item \begin{math}P^{mrh}_1 = P^{mrh}_0 \cup (\{rmft_{1,3},rmft_{1,4},x_{1,3}\},\newline
      \{rmft_{0,p} < rmft_{1,q}, rmft_{1,3} < rmft_{1,4},
      x_{0,p} < x_{1,3}, x_{0,3} < rmft_{1,3},
      rmft_{0,4} < x_{1,3}\})\end{math}, and the transitive closure,
      hence we have \begin{math}\{rmft_{0,p} < x_{1,3}, x_{0,p} < rmft_{1,q}\})\end{math},
\item and so on: \begin{math}P^{mrh}_{i+1} = P^{mrh}_{i} \cup (\{rmft_{i+1,3},rmft_{i+1,4},x_{i+1,3}\},\newline
      \{rmft_{i+1,3} < rmft_{i+1,4}, rmft_{j,p} < rmft_{i+1,q} \text{ when } j < i+1,
      x_{j,p} < x_{i+1,3} \text{ when } j < i+1, x_{i,3} < rmft_{i+1,3},
      rmft_{i,4} < x_{i+1,3}\})\end{math}, and the transitive closure,
      hence we have \begin{math}\{rmft_{j,p} < x_{i+1,3} \text{ when } j < i+1, x_{j,p} < rmft_{i+1,3} \text{ when } j < i+1,
      x_{j,p} < rmft_{i+1,4} \text{ when } j < i\})\end{math}.
\end{itemize}
The rmft elements form the relatively maximum full trunk of these orders, it is a chain lengthened at each step.
The x elements form the lengthiest chain that does not intersect the relatively maximum full trunk.
We expect the reader to draw the first three \begin{math}P^{mrh}_i\end{math}.
For \begin{math}P^{mrh}_0\end{math}, the second subscript indicates the level of each element.

After this structural study, it seems reasonable to choose the class of orders that will be at the same time larger and with better algorithmic results.
It is clear that all results hereafter (class membership testing, order isomorphism, couting the number of linear extensions)
may be generalized to orders for which each connected component (\begin{math}P\end{math} is seen as a directed graph)
belongs to the class studied.
Indeed, one can compute in \begin{math}O(n^2)\end{math} the connected components of \begin{math}P\end{math};
order isomorphism can be solved by pairing the isomorphic connected components;
counting the number of linear extensions \begin{math}|\LinearExtensions(P)|\end{math} of an order \begin{math}P\end{math},
with \begin{math}k\end{math} connected components \begin{math}CC_i\end{math} with \begin{math}n_i, 1 \leq i \leq k,\end{math} elements,
can be done with the following formula
\begin{displaymath}|\LinearExtensions(P)| = (\prod_{i=1}^{i=k} |\LinearExtensions(CC_i)|) \times
                   (\prod_{i=2}^{i=k} \Fusion(\sum_{j=1}^{j=i-1} n_j, n_i)),\end{displaymath}
where \begin{math}\Fusion(p,q)\end{math} counts the number of ways to obtain a linear extension
 of an order made of a chain of size \begin{math}p\end{math} and a chain of size \begin{math}q\end{math}.
\begin{math}\Fusion(p,q) = \Fusion(q,p)\end{math} can easily be computed recursively:
\begin{itemize}
\item \begin{math}\Fusion(p,0) = 1\end{math},
\item \begin{math}\Fusion(p,1) = p+1\end{math},
\item \begin{math}\Fusion(p,q) = \sum_{i=0}^{i=p} (\Fusion(p-i,q-2) \times \Fusion(i,1)) = \sum_{i=0}^{i=p} (\Fusion(p-i,q-2) \times (i+1))\end{math}.
\end{itemize}

The simplest superclass of (itov) orders that we encoutered are the up-regular orders.
It is easy to see that testing membership in this class can be done in \begin{math}O(n^3)\end{math}.
Indeed it is the complexity of computing the level decomposition of an order.
Given a level decomposition, testing if an element \begin{math}x\end{math} is up-regular can be done in \begin{math}O(n^2)\end{math},
since we only need to compute an array of at most \begin{math}n\end{math} integers:
in the ith cell, there is a counter for the number of elements of the ith level that are more than \begin{math}x\end{math};
in \begin{math}O(n^2)\end{math} we compute this array and then in \begin{math}O(n \times \log(n))\end{math},
we can check that the ith cell contains either 0 or the number of elements of the ith level, for i more than the level of \begin{math}x\end{math}.
Hence, testing that the order is up-regular can be done in time \begin{math}O(n^3)\end{math}.

Testing order isomorphism between two up-regular orders is also easy.
Attach to each element \begin{math}x\end{math} a label \begin{math}\Label(x)\end{math} equal to the smallest integer
such that \begin{math}x\end{math} is less than elements of the level corresponding to this integer.
If no element is more than \begin{math}x\end{math}, we can give the label \begin{math}\Height(P)\end{math}
corresponding to a virtual level above all elements.
Clearly in an up-regular order this integer is sufficient since all elements in levels above
must be more than \begin{math}x\end{math} by transitivity.
The size of \begin{math}\Label(x)\end{math} is \begin{math}O(\log(n))\end{math},
and it can be computed in time \begin{math}O(n)\end{math}.
Thus computing the labels can be done in time \begin{math}O(n^2)\end{math}, given a level decomposition.
Two up-regular orders are isomorphic if the number of elements in each level of their level decomposition with the same label are the same.
Clearly, for two elements in the same level with the same label, there is an automorphism that exchanges these two elements.
Ordering the labels in each level, can be done in time
\begin{math}O(n \times \log(n)^2 \times \log(\log(n)))\end{math},
then comparing the ordered lists can be done in time \begin{math}O(n \times \log(n))\end{math}.
Thus testing isomorphism has complexity \begin{math}O(n^3)\end{math} given by the cost of level decomposition.

We will try to compute the number of linear extensions of (itov) orders,
we start with a more simple class of (itov) orders but more complex than trunks.
A \emph{cedar} is an (itov) order \begin{math}P\end{math} with a
relatively maximum full trunk \begin{math}\RMFTrunk(P)\end{math},
such that \begin{math}P \setminus \RMFTrunk(P)\end{math} has no comparability relationship (hence is a trunk),
and no element of \begin{math}P \setminus \RMFTrunk(P)\end{math} is less than an element of \begin{math}\RMFTrunk(P)\end{math}.
For example, take a chain of height 7 (7 elements) for \begin{math}\RMFTrunk(P)\end{math},
add 5 elements on level 3 (levels start at 0, add three elements to the right of \begin{math}\RMFTrunk(P)\end{math},
and add two to the left), add 4 elements on level 4, add one element on level 5,
draw the arcs without drawing those obtained by transitivity,
and draw a smooth curve around each level that has more than one element.
It should look like a cedar.

Clearly computing the number of linear extensions of \begin{math}\RMFTrunk(P)\end{math}
and of \begin{math}P \setminus \RMFTrunk(P)\end{math} is not enough to merge these linear extensions.
To each element of \begin{math}P \setminus \RMFTrunk(P)\end{math} we can associate a label corresponding to its level
(or equivalently to the highest level of the trunk that is less than it).
Then we can define a profile of any linear extension of \begin{math}P \setminus \RMFTrunk(P)\end{math}
as the ordered sequence of labels corresponding to elements of the linear extension from bottom to top.
There can be at most \begin{math}\Height(\RMFTrunk(P)) - 1\end{math} distinct labels,
and thus there can be an exponential number like \begin{math}(\Height(\RMFTrunk(P)) - 1)!\end{math}
of distinct profiles for linear extensions of \begin{math}P \setminus \RMFTrunk(P)\end{math}.
The approach of splitting an order between its relatively maximum full trunk and the rest of it does not work.

There is an approach that works for cedars.
Indeed consider the elements of a cedar by increasing level, i.e.
consider a linear extension of the cedar where all elements of any level
are more than elements of inferior levels and less than elements of superior levels.
Hereafter, this linear extension will be denoted by \begin{math}le(P)\end{math}.
Given a level decomposition of a cedar, computing some \begin{math}le(P)\end{math} can be done in linear time.
We can compute the number of linear extensions of the suborder of the cedar 
induced by the first \begin{math}i\end{math} elements according to \begin{math}le\end{math}.
Let us denote \begin{math}P_i\end{math} this suborder.
Let us denote \begin{math}\LinearExtensions(P, i, l_1, j_1, l_2, j_2)\end{math}
the set of linear extensions of \begin{math}P_i\end{math}
such that for each linear extension \begin{math}lea\end{math} in \begin{math}\LinearExtensions(P, i, l_1, j_1, l_2, j_2)\end{math}:
\begin{itemize}
 \item \begin{math}j_1 = \max(\{\Level_{lea}(x) ; x \in (\Level^{-1}_P(l) \cap P_i \cap \RMFTrunk(P)), 0 \leq l \leq l_1\} \cup \{-1\})\end{math},
 \item \begin{math}j_2 = \max(\{\Level_{lea}(x) ; x \in (\Level^{-1}_P(l) \cap P_i \cap \RMFTrunk(P)), 0 \leq l \leq l_2\})\end{math}.
\end{itemize}
Since \begin{math}\LinearExtensions(P, i, l_1, j_1, l_2, j_2) = \emptyset\end{math},
unless \begin{math}1 \leq i \leq n, -1 \leq l_1, l_2, j_1, j_2 < i\end{math},
we only have to consider a polynomial number of such sets.
Since we consider a cedar, we will only use the case \begin{math}l_2 = l_1 + 1\end{math}.
Clearly \begin{math}\LinearExtensions(P, 1, -1, -1, 0, 0) = \LinearExtensions(P, 1)\end{math}.
(\begin{math}|\LinearExtensions(P, 1, -1, 0, 0, 0)| = 0\end{math}.)
Assume that we have computed \begin{math}|\LinearExtensions(P, i, \Level(x_i) - 1, j_1, \Level(x_i), j_2)|\end{math},
for all \begin{math}-1 \leq j_1 < i, 0 \leq j_2 < i\end{math} with \begin{math}j_1 \leq j_2\end{math},
all other values for \begin{math}j_1, j_2\end{math} may be considered to be 0.
We have four cases to consider:
\begin{itemize}
\item \begin{math}x_{i+1}\end{math} does not belong to \begin{math}\RMFTrunk(P)\end{math},
  and \begin{math}\Level(x_{i+1}) = \Level(x_i)\end{math}.
  It is clear that for any linear extension in \begin{math}\LinearExtensions(P, i, \Level(x_i) - 1, j_1, \Level(x_i), j_2)\end{math},
  we obtain a linear extension in:
  \begin{itemize}
    \item \begin{math}\LinearExtensions(P, i + 1, \Level(x_{i+1}) - 1, j_1, \Level(x_{i+1}), j_2)\end{math},
          if we add \begin{math}x_{i+1}\end{math} higher than position \begin{math}j_2\end{math},
          there is \begin{math}i - j_2\end{math} such choices,
    \item \begin{math}\LinearExtensions(P, i + 1, \Level(x_{i+1}) - 1, j_1, \Level(x_{i+1}), j_2 + 1)\end{math},
          if we add \begin{math}x_{i+1}\end{math} higher than position \begin{math}j_1\end{math} and lower than position \begin{math}j_2\end{math},
          there is \begin{math}j_2 - j_1\end{math} such choices.
  \end{itemize}
  Note that we cannot add \begin{math}x_{i+1}\end{math} lower than position \begin{math}j_1\end{math}.
  Thus \begin{math}|\LinearExtensions(P, i + 1, \Level(x_{i+1}) - 1, j_1, \Level(x_{i+1}), j_2)| = \newline
  (i - j_2) \times |\LinearExtensions(P, i, \Level(x_i) - 1, j_1, \Level(x_i), j_2)|
  + (j_2 - j_1 - 1) \times |\LinearExtensions(P, i, \Level(x_i) - 1, j_1, \Level(x_i), j_2 - 1)|
  \end{math} 
  for all \begin{math}-1 \leq j_1 < i + 1, 0 \leq j_2 < i + 1\end{math} with \begin{math}j_1 \leq j_2\end{math}.
\item \begin{math}x_{i+1}\end{math} does not belong to \begin{math}\RMFTrunk(P)\end{math},
  and \begin{math}\Level(x_{i+1}) = \Level(x_i) + 1\end{math}.
  It is clear that for any linear extension in \begin{math}\LinearExtensions(P, i, \Level(x_i) - 1, j_1, \Level(x_i), j_2)\end{math},
  we obtain a linear extension in:
  \begin{itemize}
    \item \begin{math}\LinearExtensions(P, i + 1, \Level(x_{i+1}) - 1, j_2, \Level(x_{i+1}), j_2)\end{math},
          if we add \begin{math}x_{i+1}\end{math} higher than position \begin{math}j_2\end{math},
          there is \begin{math}i - j_2\end{math} such choices.
  \end{itemize}
  Note that we cannot add \begin{math}x_{i+1}\end{math} lower than position \begin{math}j_2\end{math}.
  It is clear that 
  \begin{math}|\LinearExtensions(P, i + 1, \Level(x_{i+1}) - 1, j_1, \Level(x_{i+1}), j_2)| = 0\end{math} 
  if \begin{math}j_1 \neq j_2\end{math}.
  Similarly \begin{math}|\LinearExtensions(P, i + 1, \Level(x_{i+1}) - 1, j_2, \Level(x_{i+1}), j_2)| = \newline
  \sum_{j_1 = -1}^{j_1 = j_2} ((i - j_2) \times |\LinearExtensions(P, i, \Level(x_i) - 1, j_1, \Level(x_i), j_2)|)
  \end{math}
  for all \begin{math}0 \leq j_2 < i + 1\end{math}.
\item \begin{math}x_{i+1}\end{math} belongs to \begin{math}\RMFTrunk(P)\end{math},
  and \begin{math}\Level(x_{i+1}) = \Level(x_i)\end{math}.
  It is clear that for any linear extension in \begin{math}\LinearExtensions(P, i, \Level(x_i) - 1, j_1, \Level(x_i), j_2)\end{math},
  we obtain a linear extension in:
  \begin{itemize}
    \item \begin{math}\LinearExtensions(P, i + 1, \Level(x_{i+1}) - 1, j_1, \Level(x_{i+1}), j_3)\end{math},
          with \begin{math}j_3 > j_2 + 1\end{math},
          if we add \begin{math}x_{i+1}\end{math} at position \begin{math}j_3\end{math},
          there is only one such choice for \begin{math}j_3\end{math} given (assuming \begin{math}j_2 < i - 1\end{math}),
    \item \begin{math}\LinearExtensions(P, i + 1, \Level(x_{i+1}) - 1, j_1, \Level(x_{i+1}), j_2 + 1)\end{math},
          if we add \begin{math}x_{i+1}\end{math} higher than position \begin{math}j_1\end{math} and lower than position \begin{math}j_2\end{math},
          or at position \begin{math}j_2 + 1\end{math},
          there is \begin{math}j_2 - j_1 + 1\end{math} such choices.
  \end{itemize}
  Note that we cannot add \begin{math}x_{i+1}\end{math} lower than position \begin{math}j_1\end{math}.
  Thus \begin{math}|\LinearExtensions(P, i + 1, \Level(x_{i+1}) - 1, j_1, \Level(x_{i+1}), j_2)| = \newline
  (\sum_{j_2' = 0}^{j_2' = j_2 - 2} |\LinearExtensions(P, i, \Level(x_i) - 1, j_1, \Level(x_i), j_2')|)
  + (j_2 - j_1) \times |\LinearExtensions(P, i, \Level(x_i) - 1, j_1, \Level(x_i), j_2 - 1)|
  \end{math} 
  for all \begin{math}-1 \leq j_1 < i + 1, 0 \leq j_2 < i + 1\end{math} with \begin{math}j_1 \leq j_2\end{math}.
\item \begin{math}x_{i+1}\end{math} belongs to \begin{math}\RMFTrunk(P)\end{math},
  and \begin{math}\Level(x_{i+1}) = \Level(x_i) + 1\end{math}.
  It is clear that for any linear extension in \begin{math}\LinearExtensions(P, i, \Level(x_i) - 1, j_1, \Level(x_i), j_2)\end{math},
  we obtain a linear extension in:
  \begin{itemize}
    \item \begin{math}\LinearExtensions(P, i + 1, \Level(x_{i+1}) - 1, j_2, \Level(x_{i+1}), j_3)\end{math},
          with \begin{math}j_3 \geq j_2 + 1\end{math},
          if we add \begin{math}x_{i+1}\end{math} at position \begin{math}j_3\end{math},
          there is only one such choice for \begin{math}j_3\end{math} given.
  \end{itemize}
  Note that we cannot add \begin{math}x_{i+1}\end{math} lower than position \begin{math}j_2\end{math}.
  It is clear that 
  \begin{math}|\LinearExtensions(P, i + 1, \Level(x_{i+1}) - 1, j_1, \Level(x_{i+1}), j_2)| = \newline
  \sum_{j_1' = -1}^{j_1' = j_1} (|\LinearExtensions(P, i, \Level(x_i) - 1, j_1', \Level(x_i), j_1)|)
  \end{math}
  for all \begin{math}-1 \leq j_1 < i + 1, 0 \leq j_2 < i + 1\end{math} with \begin{math}j_1 < j_2\end{math}
  (it is 0 if \begin{math}j_1 = j_2\end{math}).
\end{itemize}
Clearly, 
\begin{math}
|\LinearExtensions(P_i)| = \sum_{j_1 = -1}^{j_1 = i-1} \sum_{j_2 = j_1}^{j_2 = i-1} (|\LinearExtensions(P, i, \Level(x_i) - 1, j_1, \Level(x_i), j_2)|)
\end{math}, for all \begin{math}i\end{math}.
Thus \begin{math}|\LinearExtensions(P)| = |\LinearExtensions(P_n)|\end{math}
can be computed in polynomial time, for cedars.
(There is at most a linear number of additions and two multiplications in order to compute some 
\begin{math}|\LinearExtensions(P, i + 1, \Level(x_{i+1}) - 1, j_1, \Level(x_{i+1}), j_2)|\end{math}.
There is at most \begin{math}O(n^3)\end{math} such numbers to compute,
so the algorithm runs in \begin{math}O(n^5)\end{math}, which is bad for such a simple class of orders.)
We did not succeed to generalize this approach for all (itov) orders, because it would require to keep track
of positions for a linear number of levels, and thus would yield an exponential time algorithm.

However, there is a very simple way to compute the number of linear extensions of (itov) orders.
It really shows that nice structural results, such as those above, may be completely unefficient.
After we presented version 4 of this article at the ENS Lyon, St\'ephan Thomass\'e made the following remark:
"When you exclude induced suborders isomorphic to \begin{math}O_{obs2}\end{math},
since you are considering orders, you exclude \begin{math}P_4\end{math}.
(itov) orders look like cographs where the disjoint union has all but one subgraphs that are isolated vertices.
Thus counting the number of linear extensions should be trivial."

Indeed, finite orders excluding induced suborders isomorphic to \begin{math}O_{obs2}\end{math}
are the transitive series parallel digraphs of \cite{Lawler1976}
(see also the articles by \cite{DBLP:conf/stoc/ValdesTL79} and \cite{DBLP:journals/dam/CorneilLB81}).
They can be recognized in linear time (see the articles by \cite{DBLP:conf/stoc/ValdesTL79} and \cite{DBLP:journals/dam/CrespelleP06}, for example),
and even obtain the corresponding modular decomposition with the same time complexity
(here linear time means \begin{math}O(n+m)\end{math}, linear in the sum of the number of elements and comparability arcs, not \begin{math}O(n)\end{math}).
Given the modular decomposition using disjoint sum and ``order composition'' nodes,
computing the number of linear extensions is trivial since it is using \begin{math}\Fusion\end{math} defined above
in the first case, and a multiplication in the second case.
It is easy to see that disjoint sum where at least two suborders are not reduced to isolated vertices yields an obstruction \begin{math}O_{obs1}\end{math},
and it is also easy to see that such an obstruction cannot be constructed by order composition
or by disjoint sum where at most one suborder has comparability relationship (has arcs).
Thus, it truly characterizes finite (itov) orders.
Similarly, it is easy to see that modular decomposition where all disjoint sums are applied only to suborders without comparability relationship
(without arcs) characterizes finite trunks.
Instead of using \begin{math}\Fusion\end{math}, we can use a simple multiplication by the cardinal plus one
of the previous suborders in the disjoint sum, when the suborder added is an isolated vertex.
Thus, counting the number of linear extensions requires at most a linear number of additions and multiplications,
its complexity is \begin{math}O((n+m) + n \times n \times \log(n) + n \times \MultiplyCost(n \times \log(n)))
  = O(m + n \times \MultiplyCost(n \times \log(n)))
  = O(n \times \MultiplyCost(n \times \log(n)))\end{math}, since \begin{math}m \in O(n^2)\end{math}.
We note that the classes of series parallel interval/(itov) orders and weak orders/trunks are not listed in the recent article by \cite{DBLP:journals/corr/abs-1907-00801}
on subclasses of directed co-graphs.

\begin{openproblem}
Find the exact complexity of counting the linear extensions of up-regular orders.
\end{openproblem}

\section{(Strict) questionable representations of width 2 for partial orders}
\label{section:questionable_representations_width2_for_partial_orders}

If we do not ask that the questionable representation of width 2 is total,
then the prefix condition between two incomparable elements is no longer valid.
However, it is easy to see that neither \begin{math}O^{0,1}\end{math},
nor \begin{math}O^{a,b} = (\{a,b\}, \{a \sim b\})\end{math} can decompose \begin{math}O_{obs2}\end{math}.
In this section, we prove that excluding induced suborders isomorphic to \begin{math}O_{obs2}\end{math}
is the necessary and sufficient condition for the existence of a questionable representation of width 2.
The result is easy and partially already known, but we give the proof to show that it does hold also for infinite orders.

\begin{lemma}
If an order excludes \begin{math}O_{obs2}\end{math},
then either there is a bipartition of its elements such that no element of the first part is comparable to some element of the second part,
or there is a bipartition of its elements such that all elements of the first part are less than all elements of the second part.
\end{lemma}
\begin{proof}
(We recommend drawing the sets as we define them along the proof.)
By contrapositive, assume that there is no bipartition of its elements such that no element of the first part is comparable to some element of the second part,
and there is no bipartition of its elements such that all elements of the first part are less than all elements of the second part.
For any element \begin{math}x\end{math}, if all other elements are incomparable with
 (resp. less than, resp. more than) \begin{math}x\end{math},
 then \begin{math}x\end{math} and its complement yields
one of the sought bipartitions, a contradiction.
\begin{itemize}
\item If all elements are comparable with \begin{math}x\end{math}, then all elements that are less than \begin{math}x\end{math}
are less than all elements that are more than \begin{math}x\end{math}, by transitivity.
Thus we have one of the sought bipartitions (whatever part we put \begin{math}x\end{math} in), a contradiction.
\item If all elements are either incomparable or more than \begin{math}x\end{math},
let \begin{math}Ix\end{math} and \begin{math}Mx\end{math} be the corresponding sets.
By transitivity, there is no element of \begin{math}Mx\end{math} less than some element of \begin{math}Ix\end{math}.
If there is no element of \begin{math}Mx\end{math} more than some element of \begin{math}Ix\end{math},
then \begin{math}Ix\end{math} and the rest of the order yields one of the sought bipartition, a contradiction.
(We mark this possibility as ``end 1'' for later use.)

Hence, let us split \begin{math}Ix = Ix\_LMx \sqcup Ix\_IMx\end{math} between the elements that are less than some element of \begin{math}Mx\end{math}
and the elements that are incomparable with all elements of \begin{math}Mx\end{math}.
If \begin{math}Ix\_IMx\end{math} is not empty, then
  \begin{itemize}
  \item either all its elements are incomparable with elements in \begin{math}Ix\_LMx\end{math},
  and \begin{math}Ix\_IMx\end{math} yields a bipartition (we mark this possibility as ``end 2'' for later use),
  \item or by transitivity there is an arc from \begin{math}z \in Ix\_LMx\end{math} to \begin{math}t \in Ix\_IMx\end{math},
  and then we have an induced suborder isomorphic to \begin{math}O_{obs2}\end{math} with \begin{math}x,z,t\end{math}, and some element \begin{math}y \in Mx\end{math}
  such that \begin{math}y > z\end{math}.
  \end{itemize}
Thus \begin{math}Ix = Ix\_LMx\end{math}.

Hence, let us split \begin{math}Mx = Mx\_MIx \sqcup Mx\_IIx\end{math} between the elements that are more than some element of \begin{math}Ix\end{math}
and the elements that are incomparable with all elements of \begin{math}Ix\end{math}.
(We already noted that \begin{math}Mx\_MIx\end{math} is not empty.)
If \begin{math}Mx\_IIx\end{math} is not empty, then
  \begin{itemize}
  \item either at least one of its elements \begin{math}y\end{math}
  is incomparable with at least one element \begin{math}z\end{math} in \begin{math}Mx\_MIx\end{math},
  there is some element \begin{math}t \in Ix = Ix\_LMx, t < z\end{math},
  and we have an induced suborder isomorphic to \begin{math}O_{obs2}\end{math},
  \item or by transitivity all elements of \begin{math}Mx\_IIx\end{math} are less than all elements of \begin{math}Mx\_MIx\end{math}.
  \end{itemize}
Thus, we continue the proof assuming that \begin{math}Mx\_IIx\end{math} is empty,
or all elements of \begin{math}Mx\_IIx\end{math} are less than all elements of \begin{math}Mx\_MIx\end{math}.
If all elements of \begin{math}Ix = Ix\_LMx\end{math} are less than all elements of \begin{math}Mx\_MIx\end{math},
then \begin{math}Mx\_MIx\end{math} yields a bipartition.
(We mark this possibility as ``end 3'' for later use.)
If two elements \begin{math}z, z' \in Ix = Ix\_LMx, z < z'\end{math}, then the elements \begin{math}y \in Mx\_MIx\end{math}
that are more than \begin{math}z'\end{math} are more than \begin{math}z\end{math} by transitivity.
The elements \begin{math}y \in Mx\_MIx\end{math}
that are more than \begin{math}z\end{math} are also more than \begin{math}z'\end{math},
because otherwise the suborder induced by \begin{math}x,y,z,z'\end{math} is isomorphic to \begin{math}O_{obs2}\end{math}.
Hence, if there is one connected component in \begin{math}Ix = Ix\_LMx\end{math},
then all elements of \begin{math}Ix = Ix\_LMx\end{math} are less than all elements of \begin{math}Mx\_MIx\end{math},
and \begin{math}Mx\_MIx\end{math} yields a bipartition.
(We mark again this possibility as ``end 3'' for later use.)
Let \begin{math}z, z' \in Ix = Ix\_LMx\end{math} be in two connected components,
and assume that there are some elements \begin{math}y, y' \in Mx\_MIx\end{math},
such that the connected component of \begin{math}z\end{math} is less than \begin{math}y\end{math},
incomparable with \begin{math}y'\end{math},
and the connected component of \begin{math}z'\end{math} is less than \begin{math}y'\end{math},
incomparable with \begin{math}y\end{math}.
If \begin{math}y,y'\end{math} are comparable,
then the suborder induced by \begin{math}z,y,y',z'\end{math} is isomorphic to \begin{math}O_{obs2}\end{math}.
If \begin{math}y,y'\end{math} are incomparable,
then the suborder induced by  \begin{math}y,x,y',z'\end{math} is isomorphic to \begin{math}O_{obs2}\end{math}.
Hence, the subsets of \begin{math}Mx\_MIx\end{math} more than some connected component of \begin{math}Ix = Ix\_LMx\end{math}
are ordered by inclusion, their intersection \begin{math}Core(Mx\_MIx)\end{math} is non-empty.
Let us assume that \begin{math}y \in Core(Mx\_MIx), y' \in Mx\_MIx \setminus Core(Mx\_MIx)\end{math}.
If \begin{math}y,y'\end{math} are incomparable,
then the suborder induced by \begin{math}y',x,y,z\end{math} is isomorphic to \begin{math}O_{obs2}\end{math}
for some correctly chosen \begin{math}z\end{math} in a connected component of \begin{math}Ix = Ix\_LMx\end{math},
that is not less than \begin{math}y'\end{math}.
Moreover by transitivity, if \begin{math}y < y'\end{math}, \begin{math}y' \in Core(Mx\_MIx)\end{math}, a contradiction.
Thus, \begin{math}Core(Mx\_MIx)\end{math} yields the sought bipartition, at last.
(We mark this possibility as ``end 4'' again for later use.)

\item If all elements are either incomparable or less than \begin{math}x\end{math},
let \begin{math}Ix\end{math} and \begin{math}Lx\end{math} be the corresponding sets.
This case is symetric to the previous case.
(Consider the inverse order and observe that \begin{math}O_{obs2} \equiv \Inv(O_{obs2})\end{math},
and that sought bipartitions are preserved by taking the inverse order.)
\end{itemize}

Hence, for any element \begin{math}x\end{math}, we have elements \begin{math}i_x, l_x, g_x, x \sim i_x, l_x < x, x < g_x\end{math}.
(Thus, this is the end of the proof if the order is finite or well-founded.)
We can use the two symetric previous cases and observe that when we found an induced suborder isomorphic to \begin{math}O_{obs2}\end{math}
in the suborder induced by  \begin{math}\{x\} \sqcup Ix \sqcup Mx\end{math}, it was also an induced suborder of the bigger considered orders.
Thus, we have 4 cases to consider for \begin{math}\{x\} \sqcup Ix \sqcup Mx\end{math} (resp. \begin{math}\{x\} \sqcup Ix \sqcup Lx\end{math}):
\begin{itemize}
\item[(end 1)] \begin{math}Ix\end{math} is incomparable with \begin{math}\{x\} \sqcup Mx\end{math} (resp. \begin{math}\{x\} \sqcup Lx\end{math}),
\item[(end 2)] \begin{math}Ix\_IMx\end{math} (resp. \begin{math}Ix\_ILx\end{math}) is incomparable with
               \begin{math}\{x\} \sqcup Mx \sqcup Ix\_LMx\end{math}
              (resp. \begin{math}\{x\} \sqcup Lx \sqcup Ix\_MLx\end{math}),
\item[(end 3)] \begin{math}Mx\_MIx\end{math} (resp. \begin{math}Lx\_LIx\end{math}) is more (resp. less)
               than \begin{math}\{x\} \sqcup Ix \sqcup (Mx \setminus Mx\_MIx)\end{math}
                 (resp. \begin{math}\{x\} \sqcup Ix \sqcup (Lx \setminus Lx\_LIx)\end{math}),
\item[(end 4)] \begin{math}Core(Mx\_MIx)\end{math} (resp. \begin{math}Core(Lx\_LIx)\end{math})
                 is more (resp. less) than \begin{math}\{x\} \sqcup Ix \sqcup (Mx \setminus Core(Mx\_MIx))\end{math}
                 (resp. \begin{math}\{x\} \sqcup Ix \sqcup (Lx \setminus Core(Lx\_LIx))\end{math}).
\end{itemize}
(end 3) and (end 4) may be merged by considering that in (end 3) \begin{math}Core(Mx\_MIx) = Mx\_MIx\end{math}
(resp. \begin{math}Core(Lx\_LIx) = Lx\_LIx\end{math}).
We note (end i') when the case applies to \begin{math}\{x\} \sqcup Ix \sqcup Lx\end{math} instead of
\begin{math}\{x\} \sqcup Ix \sqcup Mx\end{math}.
It is easy to see that if (end 3,4) (resp. (end 3,4')) holds, then by transitivity,
\begin{math}Core(Mx\_MIx)\end{math} is more than \begin{math}\{x\} \sqcup Ix \sqcup (Mx \setminus Core(Mx\_MIx))\end{math}
implies that \begin{math}Core(Mx\_MIx)\end{math} is more than \begin{math}Lx\end{math},
and \begin{math}Core(Mx\_MIx)\end{math} (resp. \begin{math}Core(Lx\_LIx)\end{math}) yields the sought bipartition.
It remains the following combinations:
\begin{itemize}
\item (end 1) and (end 1'), \begin{math}Ix\end{math} yields a bipartition,
\item ((end 1) and (end 2')) or ((end 2) and (end 1')), \begin{math}Ix\_ILx\end{math} (resp. \begin{math}Ix\_IMx\end{math}) yields a bipartition,
\item (end 2) and (end 2'), if  \begin{math}Ix\_IMx \subseteq Ix\_ILx\end{math} (resp. \begin{math}Ix\_ILx \subseteq Ix\_IMx\end{math}),
                            then clearly \begin{math}Ix\_IMx\end{math} (resp. \begin{math}Ix\_ILx\end{math}) yields a bipartition,
                            otherwise let \begin{math}y \in Ix\_IMx \setminus Ix\_ILx\end{math}, \begin{math}z \in Ix\_ILx \setminus Ix\_IMx\end{math},
                            \begin{math}y' \in Lx, y' < y\end{math}, \begin{math}z' \in Mx, z' > z\end{math},
                            the suborder induced by \begin{math}y,y',z',z\end{math} is isomorphic to \begin{math}O_{obs2}\end{math}.
\end{itemize}
This ends the proof.
\end{proof}

\begin{theorem}
\label{theorem:order_without_obs2_to_qr}
Any partial order \begin{math}O\end{math} of cardinal \begin{math}\aleph\end{math}, without induced suborder isomorphic to \begin{math}O_{obs2}\end{math},
 has a questionable representation of width 2
 and length at most \begin{math}\alpha\end{math},
where \begin{math}\alpha\end{math} is the ordinal of cardinal \begin{math}\aleph\end{math},
if \begin{math}\aleph\end{math} is finite,
and \begin{math}\alpha = \omega_{\beta + 1}\end{math},
if \begin{math}\aleph = \aleph_{\beta}\end{math} is infinite.
\end{theorem}
\begin{demo}
We first consider a questionable representation of width 3 using only the partial order \begin{math}O^{0,1,a} = (\{0,1,a\}, \{0 < 1, a \sim 0, a \sim 1\})\end{math}.
Using the previous lemma, we can see that we can recursively split any induced suborder of \begin{math}O\end{math}
in two incomparable or ordered induced suborders.
Thus, we obtain a questionable representation of width 3 and unknown length.
Clearly, we can do so that each induced suborder ``of rank \begin{math}i\end{math}'' is split exactly
in two non empty parts by the element-items of rank \begin{math}i\end{math} in the words associated to its elements.
(\begin{math}O\end{math} is the order of rank 0.
Each induced suborder ``of rank \begin{math}i\end{math}'' uses exactly two element-items among \begin{math}\{0,1,a\}\end{math}.
We can choose \begin{math}\{0,1\}\end{math} for order composition and \begin{math}\{0,a\}\end{math} for disjoint union.
)

Clearly, the number of elements in \begin{math}O\end{math} cannot be less than the cardinal of the ordinal length of a word associated to an element.
Thus, the length of the questionable representation is at most \begin{math}\alpha = \omega_{\beta + 1}\end{math}
(remember the strict inequality in the definition of the ordinal length of a questionable representation).
This infinite length is not changed by replacing each order-item \begin{math}O^{0,1,a}\end{math}
by two consecutive order-items \begin{math}O^{0,1}\end{math} and \begin{math}O^{a,b}\end{math}.
(The element-items \begin{math}0,1,a\end{math} are replaced by \begin{math}(0,b),(1,b),(0,a)\end{math} respectively.)
\end{demo}

\begin{openproblem}
Can the length of infinite questionable representations obtained in Theorem~\ref{theorem:order_without_obs2_to_qr} be improved?
\end{openproblem}

\section{Questionable representations for partial orders}
\label{section:questionable_representations_for_partial_orders}

We first remark that there is no need to distinguish between strict and nonstrict, nontotal questionable representations of partial orders.
Indeed let us assume that a partial order \begin{math}O\end{math} admits a nonstrict questionable representation
using the finite partial order \begin{math}O'\end{math}.
Let us name \begin{math}u,v \in O'\end{math} two elements that are incomparable.
For any element \begin{math}x \in O\end{math} associated to the word \begin{math}w_x\end{math},
such that
\begin{itemize}
\item some element \begin{math}y \in O\end{math} associated to the word \begin{math}w_y\end{math} is incomparable with it,
\item \begin{math}w_x\end{math} and \begin{math}w_y\end{math} do not have a question,
\item \begin{math}\Length(w_x) \leq \Length(w_y)\end{math},
\end{itemize}
we can lengthen all the words of length at least \begin{math}\Length(w_x)\end{math} by one:
\begin{itemize}
\item \begin{math}w'_x[i] = w_x[i]\end{math}, \begin{math}i \in \Length(w_x)\end{math},
\begin{math}w'_x[\Length(w_x)] = u\end{math},
\item \begin{math}w'_z[i] = w_z[i]\end{math}, \begin{math}i \in \Length(w_x)\end{math},
\begin{math}w'_z[\Length(w_x)] = v\end{math},
\begin{math}w'_z[i+1] = w_z[i]\end{math}, \begin{math}i \in (\Length(w_z) \setminus \Length(w_x))\end{math},
for all \begin{math}z \in (O \setminus \{x\})\end{math} such that \begin{math}\Length(w_x) \leq \Length(w_z)\end{math}
(\begin{math}y\end{math} is such a \begin{math}z\end{math}).
\end{itemize}
Hence, by transfinite induction, we can remove all incomparable relationship that is not explicitely in the questionable representation,
and we end up with a strict questionable representation.

We also remark that we may consider that the partial order \begin{math}O\end{math} is connected.
Indeed, let us consider the set \begin{math}CC\end{math} of the connected components of \begin{math}O\end{math},
we may start all the words of the questionable representation by prefixes of the same length
where all digits are \begin{math}u,v\end{math} two elements that are incomparable.
Giving the same prefix to all elements of the same connected component,
and such that for any two components, at least one digit of their respective prefix differs.
In particular, if the number of connected components is finite,
the length of the prefixes may be a binary logarithm of this number of component.

Let us consider the \emph{questionable-width} of an order as the minimum cardinal such that there exists
a questionable representation of this order using only orders of this cardinal.
We first note that no bound exists on the questionable-width of all orders.
Indeed let us define a ``zigzag'' of ordinal length \begin{math}\alpha\end{math},
denoted \begin{math}Zigzag_\alpha\end{math},
as an order of height 2, made of two sequences of elements
\begin{math}S = (s_i)_{i \in \alpha}, S' = (s'_i)_{i \in \alpha}\end{math} of length \begin{math}\alpha\end{math},
such that the comparability relationships are
\begin{itemize}
\item \begin{math}s_i < s'_i, i \in \alpha\end{math},
\item \begin{math}s_i < s'_{i+1}, i,i+1 \in \alpha\end{math},
\item \begin{math}s_j < s'_i, j < i \in \alpha\end{math}, and \begin{math}i\end{math} is a limit ordinal.
\end{itemize}

\begin{lemma}
Let \begin{math}\aleph\end{math} be the cardinal of \begin{math}Zigzag_\alpha\end{math},
then no questionable representation of \begin{math}Zigzag_\alpha\end{math} has width less than \begin{math}\aleph\end{math}.
\end{lemma}
\begin{demo}
Consider a questionable representation \begin{math}w\end{math} of \begin{math}Zigzag_\alpha\end{math}.
Assume for a contradiction that \begin{math}w(x)[0] \in O, \forall x \in \Domain(Zigzag_\alpha)\end{math}
and that \begin{math}O\end{math} has cardinal strictly less than \begin{math}\aleph\end{math}.
Then, there is at least two elements \begin{math}x,y \in \Domain(Zigzag_\alpha)\end{math}
such that \begin{math}w(x)[0] = w(y)[0]\end{math}.
Let \begin{math}X = \{z \in \Domain(Zigzag_\alpha) | w(z)[0] = w(x)[0]\}\end{math}.

\begin{enumerate}[1)]
\item If \begin{math}X\end{math} contains both \begin{math}s_i, s'_j, \text{ for some } i,j \in \alpha\end{math},
      then \begin{math}s'_i > s_i \wedge s'_i \sim s'_j\end{math}, thus \begin{math}s'_i \in X\end{math}.
\item If \begin{math}X\end{math} contains both \begin{math}s_i, s_j, \text{ for some } i,j \in \alpha\end{math},
      \begin{itemize}
      \item if \begin{math}i\end{math} is a successor ordinal,
        when \begin{math}j \neq i - 1\end{math}, \begin{math}s'_i > s_i \wedge s'_i \sim s_j\end{math}, thus \begin{math}s'_i \in X\end{math};
        when \begin{math}j = i - 1\end{math}, \begin{math}s'_j > s_j \wedge s'_j \sim s_i\end{math}, thus \begin{math}s'_j \in X\end{math};
      \item if \begin{math}i\end{math} is a limit ordinal,
        when \begin{math}j > i\end{math}, \begin{math}s'_i > s_i \wedge s'_i \sim s_j\end{math}, thus \begin{math}s'_i \in X\end{math};
        when \begin{math}j < i\end{math}, \begin{math}s'_j > s_j \wedge s'_j \sim s_i\end{math}, thus \begin{math}s'_j \in X\end{math}.
      \end{itemize}
\item If \begin{math}X\end{math} contains both \begin{math}s'_i, s'_j, \text{ for some } i,j \in \alpha\end{math},
      \begin{itemize}
      \item if \begin{math}j\end{math} is a successor ordinal,
        when \begin{math}j \neq i + 1\end{math}, \begin{math}s_i < s'_i \wedge s_i \sim s'_j\end{math}, thus \begin{math}s_i \in X\end{math};
        when \begin{math}j = i + 1\end{math}, \begin{math}s_j < s'_j \wedge s_j \sim s'_i\end{math}, thus \begin{math}s_j \in X\end{math};
      \item if \begin{math}j\end{math} is a limit ordinal,
        when \begin{math}j < i\end{math}, \begin{math}s_i < s'_i \wedge s_i \sim s'_j\end{math}, thus \begin{math}s_i \in X\end{math};
        when \begin{math}j > i\end{math}, \begin{math}s_j < s'_j \wedge s_j \sim s'_i\end{math}, thus \begin{math}s_j \in X\end{math};
      \end{itemize}
\end{enumerate}
In all cases, we obtain a pair \begin{math}s_i, s'_i \in X, \text{ for some } i \in \alpha\end{math}.

But then \begin{math}s'_{i+1} > s_i \wedge s'_{i+1} \sim s'_i\end{math},
thus \begin{math}s'_{i+1} \in X\end{math},
and \begin{math}s_{i+1} < s'_{i+1} \wedge s_{i+1} \sim s_i\end{math},
thus \begin{math}s_{i+1} \in X\end{math}.
If there is a limit ordinal \begin{math}j \in \alpha, j > i\end{math},
then \begin{math}s'_j > s_i \wedge s'_j \sim s'_i\end{math},
thus \begin{math}s'_j \in X\end{math},
and \begin{math}s_j < s'_j \wedge s_j \sim s_i\end{math},
thus \begin{math}s_j \in X\end{math}.
Hence, \begin{math}X\end{math} is the union of all \begin{math}\{s_i, s'_i\}\end{math},
for \begin{math}i\end{math} in a final segment of \begin{math}\alpha\end{math}.

But then, if \begin{math}i\end{math} is a successor ordinal,
\begin{math}s_{i-1} < s'_i \wedge s_{i-1} \sim s_i\end{math},
thus \begin{math}s_{i-1} \in X\end{math},
and \begin{math}s'_{i-1} > s_{i-1} \wedge s'_{i-1} \sim s'_i\end{math},
thus \begin{math}s'_{i-1} \in X\end{math}.
If \begin{math}i\end{math} is a limit ordinal,
then \begin{math}\forall j < i\end{math} \begin{math}s_j < s'_i \wedge s_j \sim s_i\end{math},
thus \begin{math}s_j \in X\end{math},
and \begin{math}s'_j > s_j \wedge s'_j \sim s'_i\end{math},
thus \begin{math}s'_j \in X\end{math}.
Hence, \begin{math}X\end{math} is the union of all \begin{math}\{s_i, s'_i\}\end{math},
for \begin{math}i\end{math} in an initial segment of \begin{math}\alpha\end{math}.

Thus \begin{math}X = \Domain(Zigzag_\alpha)\end{math}.
The desired contradiction.
\end{demo}

In light of this lemma, it may appear that questionable representation and questionable-width
are rather weak compared to tree-decomposition/tree-width and clique-decomposition/clique-width.
However things are not that simple, there are orders with
questionable-width 2 and arbitrary high tree-width.
%\begin{itemize}
%\item questionable-width 2 and arbitrary high tree-width,
%\item questionable-width 2 and arbitrary high clique-width.
%\end{itemize}
First we note that tree-width or clique-width may be used to mesure finite orders in two fashions:
using the directed graph of the comparability relation,
or using the directed graph of the cover relation.
(An element \begin{math}x\end{math} covers an element \begin{math}y\end{math}
if and only if \begin{math}x < y\end{math} and there is no element \begin{math}z\end{math}
with \begin{math}x < z < y\end{math}.
The directed graph of the cover relation is known as Hasse diagram
and is usually drawn with edges instead of arcs assuming that the orientation is from bottom to top.)

The comparability relation is the transitive closure of the cover relation
but the cover relation is not suitable for infinite orders.

If we use the comparability relation, the tree-width of finite total orders is not bounded whilst the questionable-width is 2.
If we use the cover relation, consider orders on two levels such that any element in the bottom level is less than elements in the top level:
the tree-width of such complete bipartite graphs is not bounded whilst the questionable-width is 2 since they are (itov) orders.
Hence, we cannot say that tree-width is worse or better than questionable-width.
They are different.

For clique-width, the comparability graph has bounded clique-width if and only if the cover graph has bounded clique-width
(we recommend reading the book by \cite{DBLP:books/daglib/0030804}).

Note that trunks, cedars, (itov) orders, orders without induced suborder isomorphic to \begin{math}O_{obs2}\end{math},
have clique-width 2, as directed co-graphs, but unbounded tree-width.
See \cite{DBLP:conf/esa/EibenGKO16}, \cite{DBLP:conf/ijcai/KangasHNK16},
and \cite{DBLP:conf/iwpec/KangasKS18} for the complexity of counting linear extensions of orders of bounded tree-width.

%Preuve 2
Consider the following orders: a ``trunk with woodpeckers'' of rank \begin{math}k\end{math},
denoted by \begin{math}TW_{k}\end{math},
is an order with a trunk/chain made of \begin{math}k\end{math} levels/elements,
together with ``woodpeckers'';
a woodpecker \begin{math}x\end{math} is an element that is regular to the trunk/chain,
and its ``legs'' are the arcs between the highest level of the trunk with elements less than \begin{math}x\end{math},
and the arcs between the lowest level of the trunk with elements more than \begin{math}x\end{math};
these two levels are denoted \begin{math}level_i(x)\end{math} and \begin{math}level_s(x)\end{math}.
We take such orders such that \begin{math}(level_i(x), level_s(x)) \neq (level_i(y), level_s(y))\end{math},
whenever \begin{math}x,y\end{math} are two distinct woodpeckers.

Moreover, two woodpeckers \begin{math}wp, wp'\end{math} have no cover relationship.
Note that such orders are not (itov):
if the trunk is a chain of height 6 (\begin{math}t_0, t_1, t_2, t_3, t_4, t_5\end{math}),
with the 4 woodpeckers \begin{math}wp_0, wp_1, wp_2, wp_3\end{math},
and \begin{math}(level_i(wp_0), level_s(wp_0)) = (t_0, t_2), (level_i(wp_1), level_s(wp_1)) = (t_1, t_3),
(level_i(wp_2), level_s(wp_2)) = (t_2, t_4), (level_i(wp_3), level_s(wp_3)) = (t_3, t_5)\end{math}),
then the suborder induced by the woodpeckers is isomorphic to \begin{math}O_{obs2}\end{math},
and the order is not (itov).
We have \begin{math}level_i(x) + 2 \leq level_s(x)\end{math}
(if \begin{math}level_i(x) + 2 = level_s(x)\end{math}, then the woodpecker belongs to the relatively maximum full trunk,
but we still say it is a woodpecker).
Thus, there is at most \begin{math}\frac{(k-1)(k-2)}{2}\end{math} woodpeckers.
We take the maximum number of woodpeckers,
and the minimum number of elements in the trunk.
Thus the trunk is a chain of \begin{math}k\end{math} elements,
and there is \begin{math}\frac{k^2-3k}{2} + 1\end{math} woodpeckers:
\begin{math}|\Domain(TW_k)| = \frac{k^2-k}{2} + 1\end{math}.

%Preuve 1
%%and we take \begin{math}\lceil(k-3)/2\rceil\end{math} elements in each level of the trunk.
%%Thus there is \begin{math}\frac{k^2-3k}{2} + 1\end{math} woodpeckers,
%%and between \begin{math}\frac{k^2-3k}{2}\end{math} and \begin{math}\frac{k^2-3k}{2} + \frac{k}{2}\end{math} elements in the trunk.
%%If \begin{math}k \geq 5\end{math} is odd, there is exactly one woodpecker more than elements in the trunk.

We observe that \begin{math}TW_{k}\end{math} is not up-regular,
but it does not contain an induced suborder isomorphic to \begin{math}O_{obs1}\end{math}.
Indeed, at least two elements of such obstruction should be outside of the trunk/chain,
otherwise there would be a chain of height 3 in \begin{math}O_{obs1}\end{math}.
\begin{itemize}
\item if the trunk contains two elements, then this is a chain of height 2,
hence the other chain of height 2 in \begin{math}O_{obs1}\end{math} is made of woodpeckers;
if the trunk contains one element, then at least one chain of height 2 in \begin{math}O_{obs1}\end{math} is made of woodpeckers;
but two woodpeckers have no cover relationship and thus, they are ordered if and only if
the upper leg of one is at most the lower leg of the other;
since any trunk element is above this upper leg or below this lower leg,
any trunk element is comparable with at least one element of the chain of 2 woodpeckers, a contradiction.
\item if all elements are woodpeckers, then we have one upper leg below some lower leg for the first chain,
and another upper leg below another lower leg for the second chain.
But then, assume without loss of generality that the first upper leg is below the second upper leg,
then the first upper leg is below the second lower leg, and we have at least three comparability;
again, we obtain a contradiction.
\end{itemize}

\begin{lemma}
For \begin{math}k \geq 8\end{math}, the Hasse diagram of \begin{math}TW_k\end{math} has clique-width at least \begin{math}\lceil\frac{k}{13}\rceil\end{math}.
\end{lemma}
\begin{proof}
Consider an optimal clique-decomposition \begin{math}cd\end{math} of \begin{math}TW_k\end{math}.
There is an internal node \begin{math}n_b\end{math} of \begin{math}cd\end{math} distinct from the root,
such that the leaves below define at least one third and at most two thirds of the elements of \begin{math}TW_k\end{math}.
Hence the same is true of the leaves that are not below this node,
we have the same inequalities for the cardinal of the two sides of this bipartition.
Let us denote %\begin{math}A \subseteq \Domain(TW_k)\end{math} these elements,
\begin{math}T \subseteq \Domain(TW_k)\end{math} the elements of the trunk/chain,
and \begin{math}W \subseteq \Domain(TW_k)\end{math} the woodpeckers.
Let us consider \begin{math}S_T\end{math} one side  of the bipartition with at least one half of \begin{math}T\end{math},
i.e \begin{math}|S_T \cap T| \geq \frac{k}{2}\end{math}.
The other side \begin{math}S_W\end{math} of the bipartition contains at least
\begin{math}\frac{1}{3} \times |\Domain(TW_k)| - \frac{k}{2} = \frac{k^2-k}{6} + \frac{1}{3} - \frac{k}{2}
                                                             = \frac{k^2-4k}{6} + \frac{1}{3}\end{math}
elements of \begin{math}W\end{math}.
%There is at most \begin{math}\lfloor\frac{k}{2}\rfloor\end{math} of the \begin{math}k\end{math} levels of \begin{math}T\end{math}
%that are contained in \begin{math}S_W\end{math}.
The number of elements in \begin{math}S_W \cap W\end{math} which have cover relationship
only with elements in \begin{math}S_W \cap T\end{math} cannot be more than
\begin{math}\frac{(\frac{k}{2} - 1)(\frac{k}{2} - 2)}{2} = \frac{k^2}{8} - \frac{3k}{4} + 1\end{math}.
Thus, there is at least \begin{math}
  \frac{k^2-4k}{6} + \frac{1}{3} - \frac{k^2}{8} + \frac{3k}{4} - 1
  = \frac{k^2}{24} + \frac{k}{12} - \frac{2}{3}
\end{math}
woodpeckers in \begin{math}S_W\end{math} that covers or are covered by trunk elements in \begin{math}S_T\end{math}.

\begin{itemize}
\item If \begin{math}S_W\end{math} is the side of the bipartition below \begin{math}n_b\end{math}.
Since \begin{math}(level_i(x), level_s(x)) \neq (level_i(y), level_s(y))\end{math},
whenever \begin{math}x,y\end{math} are two distinct woodpeckers,
it is clear that if \begin{math}(level_i(x), level_s(x)) \in (S_T \times S_T)\end{math},
for some woodpecker \begin{math}x \in S_W\end{math}, then the label of this woodpecker
at node \begin{math}n_b\end{math} must be distinct of the labels of all other woodpeckers
in \begin{math}S_W\end{math} at node \begin{math}n_b\end{math}.
Thus, we minimize the number of labels by assuming that any woodpecker 
\begin{math}x \in S_W\end{math} has only one cover relationship with \begin{math}S_T\end{math}.
At most \begin{math}\frac{k}{2}\end{math} woodpeckers may have the same cover relationship,
since the other cover relationships with \begin{math}S_W \cap T\end{math} must be distinct.
Hence, there is at least \begin{math}
  \frac{2}{k} \times (\frac{k^2}{24} + \frac{k}{12} - \frac{2}{3}) = \frac{k}{12} + \frac{1}{6} - \frac{4}{3k} 
%                                                                 > \frac{k}{12} + \frac{1}{6} - \frac{4}{9}
%                                                                 = \frac{k}{12} + \frac{3}{18} - \frac{8}{18}
%                                                                 > \frac{k}{12} + \frac{1}{6} - \frac{4}{12}
%                                                                 = \frac{k}{12} + \frac{2}{12} - \frac{4}{12}
%                                                                 > \frac{k}{12} + \frac{1}{6} - \frac{4}{15}
%                                                                 = \frac{k}{12} + \frac{5}{30} - \frac{8}{30}
                                                                 \geq \frac{k}{12} + \frac{1}{6} - \frac{4}{24}
                                                                 = \frac{k}{12}
\end{math} woodpeckers that must have distinct labels, since \begin{math}k \geq 8\end{math}.

\item If \begin{math}S_T\end{math} is the side of the bipartition below \begin{math}n_b\end{math}.
Since \begin{math}(level_i(x), level_s(x)) \neq (level_i(y), level_s(y))\end{math},
whenever \begin{math}x,y\end{math} are two distinct woodpeckers,
it is clear that elements of \begin{math}S_T \cap T\end{math}
must have distinct labels at node \begin{math}n_b\end{math} unless no woodpecker
in \begin{math}S_W\end{math} has a cover relationship with them,
or if only one woodpecker has a cover relationship with two trunk elements of same label.

Clearly, if two elements in \begin{math}S_W \cap W\end{math} have exclusive cover relationship
with four distinct elements in \begin{math}S_T \cap T\end{math} (2 labels for the four elements),
then we do not increase the number of labels needed by assuming
instead that they both have cover relationship only with the same element in \begin{math}S_T \cap T\end{math},
and that the three remaining elements in \begin{math}S_T \cap T\end{math} have no more
cover relationship with \begin{math}S_W \cap W\end{math} (1 label for the three elements).

Clearly, if one element \begin{math}x\end{math} in \begin{math}S_W \cap W\end{math} has exclusive cover relationship
with two distinct elements in \begin{math}S_T \cap T\end{math},
\begin{itemize}
\item if some element in \begin{math}S_T \cap T\end{math} has no cover relationship
with \begin{math}S_W \cap W\end{math},
then we do not increase the number of labels needed by assuming
instead that \begin{math}x\end{math} has cover relationship only 
with one element in \begin{math}S_T \cap T\end{math},
and that the remaining element in \begin{math}S_T \cap T\end{math} has no more
cover relationship with \begin{math}S_W \cap W\end{math}.

\item if all elements in \begin{math}S_T \cap T\end{math} have cover relationship
with \begin{math}S_W \cap W\end{math}, then since there is at most one woodpecker with exclusive cover relationship,
the number of distinct labels needed is at least \begin{math}\lceil\frac{k}{2}\rceil - 2 + 1\end{math}.
\end{itemize}
Thus, we may assume that no woodpecker in \begin{math}S_W \cap W\end{math}
has exclusive cover relationship with two distinct elements in \begin{math}S_T \cap T\end{math}.

At most \begin{math}\lfloor\frac{k}{2}\rfloor\end{math} woodpeckers may have the same cover relationship,
since the others cover relationships with \begin{math}S_W \cap T\end{math} must be distinct.
Hence, we have that the maximum number of woodpeckers such that the number of labels is two is
\begin{math}\lfloor\frac{k}{2}\rfloor\end{math}
(there is one label for the covering or covered trunk element and one label for all other trunk elements);
the maximum number of woodpeckers such that the number of labels is three is
\begin{math}2 \times \lfloor\frac{k}{2}\rfloor + 1\end{math},
since we can add one extra woodpecker with cover relationship with the two trunk elements
already adjacent to the other woodpeckers, etc.
The maximum number of woodpeckers such that the number of labels is \begin{math}p < \lceil\frac{k}{2}\rceil\end{math}
is \begin{math}(p - 1) \times \lfloor\frac{k}{2}\rfloor + \frac{(p-2)(p-1)}{2}\end{math}.
For \begin{math}p = \lfloor\frac{k}{q}\rfloor\end{math}, we obtain
\begin{math}(\lfloor\frac{k}{q}\rfloor - 1) \times \lfloor\frac{k}{2}\rfloor + \frac{(\lfloor\frac{k}{q}\rfloor-2)(\lfloor\frac{k}{q}\rfloor-1)}{2}
  \leq (\frac{k}{q} - 1) \times \frac{k}{2} + \frac{(\frac{k}{q}-2)(\frac{k}{q}-1)}{2}
  = \frac{k^2}{2q} - \frac{k}{2} + \frac{k^2}{2q^2} - \frac{3k}{2q} + 1
  = \frac{k^2(q+1)}{2q^2} - \frac{k(q+3)}{2q} + 1
\end{math}.
Since there is at least \begin{math}\frac{k^2}{24} + \frac{k}{12} - \frac{2}{3}\end{math} woodpeckers
with cover relationship with \begin{math}S_T \cap T\end{math},
the difference is
\begin{math}
  \frac{k^2}{24} + \frac{k}{12} - \frac{2}{3} - (\frac{k^2(q+1)}{2q^2} - \frac{k(q+3)}{2q} + 1)
  = \frac{k^2(q^2 - 12q - 12)}{24q^2} + \frac{k(q + 6q + 18)}{12q} - \frac{5}{3}
  > \frac{k^2(q^2 - 12q - 12)}{24q^2} + \frac{k(7q + 18)}{12q} - 2
\end{math}.
For \begin{math}q = 13\end{math}, we obtain
\begin{math}\frac{k^2}{24 \times 13 ^2} + \frac{109k}{12 \times 13} - 2\end{math}
which is increasing with \begin{math}k\end{math}, and positive for \begin{math}k \geq 3\end{math}.
\end{itemize}
In both cases, we obtain the sought bound.

\end{proof}

\begin{lemma}
For \begin{math}k \geq 4\end{math}, \begin{math}TW_k\end{math} has questionable-width at least \begin{math}3 \times k - 8\end{math}.
\end{lemma}
\begin{proof}
We now consider the following orders:
A chain of height \begin{math}k - 2\end{math} together with \begin{math}k - 3\end{math} down-elements
(each element of the chain except for the lowest has cover relationship with exactly one of these down-elements,
it is more than this down-element),
and \begin{math}k - 3\end{math} up-elements
(each element of the chain except for the highest has cover relationship with exactly one of these up-elements,
it is less than this up-element).
It is clear that this order is an induced suborder of the trunks with woodpeckers
(remove both extremal elements of the trunk and keep only the woodpeckers that had exactly one leg with the two extremal elements).
These orders have unbounded questionable-width.
Consider the first mapping of the questionable representation when \begin{math}k - 2 = 2\end{math}.
It is easy to see that if the two trunk elements have the same image \begin{math}x\end{math},
then the up-element and the down-element must also have the same image \begin{math}x\end{math}.
Similarly, if the up-element and the down-element have the same image \begin{math}x\end{math},
 then both trunk elements have the same image \begin{math}x\end{math}.
If one trunk element and the corresponding up-element (resp. down-element) have the same image \begin{math}x\end{math},
then the other trunk element must have the same image \begin{math}x\end{math}, and we reuse a previous case.
If one trunk element and an incomparable up-element (resp. down-element) have the same image \begin{math}x\end{math},
then the down-element (resp. up-element) must have the same image \begin{math}x\end{math}, and we reuse the previous case.
Thus the questionable-width is 4. (This order is isomorphic to \begin{math}Zigzag_2\end{math} and \begin{math}O_{obs2}\end{math}.)

Before, we can give a proof by induction, we need to strengthen our hypothesis.
Namely, we must prove that for any order made of \begin{math}i\end{math} groups ``trunk+up+down'',
where cover relationship is only between elements of the same group and between trunk elements,
if two elements are mapped to the same element, then all elements are mapped to the same element.
Above, we proved the case (trunk+up, trunk+down).
It is easy to see that if one adds an up and/or a down element,
then if two elements have the same image:
\begin{itemize}
\item if these elements are part of the previous case, then all elements of the previous case have the same image,
and the additional up/down element must also have the same image;
\item if these elements are the new up and the new down element,
 then both trunk elements must have the same image, and thus all elements have the same image;
\item if these elements are, without loss of generality,
 the new up element and some other previous element,
 then it is easy to check the four subcases and see that at least another previous element must have the same image,
 and, again, all elements have the same image.
\end{itemize}
Thus the induction hypothesis is true for \begin{math}i = 2\end{math}.

Now by induction assume that for \begin{math}k - 2 < i, i \geq 3\end{math},
 when two elements have the same image then all elements have the same image.
If \begin{math}k - 2 = i\end{math} and two elements have the same image, then clearly,
there is an induced suborder containing these two elements that is isomorphic to
the case \begin{math}i - 1\end{math}.
(If \begin{math}i \geq 3\end{math}, there is \begin{math}i\end{math} groups ``trunk+up+down'',
 there is at least one group disjoint from the two elements that can be removed.)
But then, since all the elements of the induced suborder isomorphic to
the case \begin{math}i - 1\end{math} must have the same image,
it also implies by considering two such elements in the same group,
that they have the same image than the elements of the group we previously excluded,
since they are all contained in another induced suborder isomorphic to
the case \begin{math}i - 1\end{math}.
(We obtain the bound with \begin{math}4 + 3 \times (i - 2)\end{math}  elements,
in \begin{math}i\end{math} groups (trunk+up, trunk+up+down repeated \begin{math}i  - 2\end{math} times, trunk+down),
since these orders are induced suborders of the orders used in the proof.)
\end{proof}
It shows that there exist orders excluding the induced suborder \begin{math}O_{obs1}\end{math}
with arbitrary high tree/clique/questionable-width.

For edge-weighted graphs, clique-width and tree-width are incomparable.
(See Annexe B in \cite{Lyaudet2007}, for an example/proof that some planar edge-weighted graphs of tree-width 2
have unbounded weighted clique-width.)

However, we can show that when the questionable-width is finite with a questionable representation of finite length,
then the number of labels used in an optimal clique-decomposition is at most the questionable-width.
(Of course, questionable representation of finite width and length implies that the graph is finite
if the questionable representation is strict : any two elements have a question.)
\begin{lemma}
If a finite or countable graph/order has a questionable representation of finite width \begin{math}k\end{math}
and finite length \begin{math}l\end{math},
it has a (highly symetric) clique-decomposition using \begin{math}k\end{math} labels.
If the graph is finite of \begin{math}n\end{math} elements,
the clique-decomposition has depth at most \begin{math}l \times (k(k-1) + \lceil\lg(k-1)\rceil) + \lceil\lg(n)\rceil + (l-1)(k-1)\end{math}.
\end{lemma}
\begin{proof}
  Consider such a questionable representation from left to right,
  the corresponding clique-decomposition will have its root on the left.
  For any index from 0 to the length of the questionable representation,
  associate injectively, to each element of the graph of size at most \begin{math}k\end{math} at this index,
  a label between 1 and \begin{math}k\end{math}.
  We look at the questions on index 0, if these questions add edges,
  we can have at most \begin{math}k(k-1)\end{math} operations of adding edges
  (this is one of the points where we need that the width is finitely bounded to avoid an infinite path in the syntactic tree of the term,
  such an infinite path would not allow to see actual elements of the graph after some bounded depth
  as is required by clique-width for countable graphs).
  Below these edge-adding operations is a tree of at most \begin{math}k-1\end{math} disjoint sums
  (of depth \begin{math}\lg(k-1)\end{math} if the tree is balanced),
  splitting the graph in the induced subgraphs,
  corresponding each to the vertices mapped to the same element
  (there is at most \begin{math}k\end{math} such disjoint induced subgraphs).
  Below each disjoint sum, either there is two other disjoint sums,
  or there is at least one path of \begin{math}k-1\end{math} consecutive renaming operations to give
  to all elements of an induced subgraph the same label.
  At each step, the induced subgraphs are further partitioned into smaller induced subgraphs.
  Whenever an induced subgraph is a singleton, we add a vertex adding node, otherwise we repeat the construction.
  If after all steps, there are still induced subgraphs of cardinality more than one,
  for each of them, we had a tree of disjoint sums for internal nodes and vertex creation for leaves.
  Since the length is bounded, we obtain the sought clique-decomposition.
\end{proof}

We could/should have wrote this proof by induction on the length of the questionable representation,
starting from right to left, from the vertex leaves and gluing together until we have the root of the global clique-decomposition.
However the proof we gave emphasizes the fact that we stop/give concrete vertices,
once the partitioning process generates singletons (or the length is bounded).
This cannot be achieved for infinite countable (resp. uncountable) graphs
when the questionable-width is finite (resp. countable).
Dealing with infinite structures using the first difference principle
is conceptually simpler than terms and clique-decompositions.

These results of (in)comparability suggest that we may define a width of binary structures
(structures with relations or functions of arity 2)
that generalizes tree/clique/questionable-width.
Consider a binary signature \begin{math}\mathcal{S}\end{math} of binary relations and functions.
Given a set \begin{math}S\end{math},
an \emph{\begin{math}(\mathcal{S},S,k)\end{math}-mapping-run} is an (ordinal-indexed) sequence of mappings
from \begin{math}S\end{math} to \begin{math}\mathcal{S}\end{math}-structures of cardinality at most \begin{math}k\end{math}.
\begin{definition}
Let \begin{math}X\end{math} be an \begin{math}\mathcal{S}\end{math}-structure.
A \emph{\begin{math}(k,\alpha,\beta)\end{math}-tree questionable decomposition} of \begin{math}X\end{math} is
a triple \begin{math}(T, ll, nl)\end{math}
such that:
\begin{itemize}
\item \begin{math}T\end{math} is a rooted tree.
      There are many definitions that generalize rooted trees in the infinite setting,
      here is the one that fits well.
      An ``infinite rooted tree'' is a well-founded order where the minimal elements/nodes correspond to the leaves,
      such that there is a maximum element/node (corresponding to the root),
      and for any two nodes \begin{math}n,n'\end{math}, the initial sections\footnote{
          Recall that an initial section, also called an ideal or a down-set, is a subset of the domain of an order \begin{math}I \subseteq \Domain(O)\end{math}
          (or the corresponding induced suborder),
          such that \begin{math}\forall x \in I, \forall y \in \Domain(O), y < x \text{ implies } y \in I\end{math}.
          (In a total order, an initial section is called an initial segment.)
          Recall that a final section, also called a filter or an up-set, is a subset of the domain of an order \begin{math}F \subseteq \Domain(O)\end{math}
          (or the corresponding induced suborder),
          such that \begin{math}\forall x \in F, \forall y \in \Domain(O), y > x \text{ implies } y \in F\end{math}.
          (In a total order, a final section is called a final segment.)
          (It should be pretty easy to remember that ideal goes with initial and filter goes with final.)
        }
      generated by \begin{math}n,n'\end{math}
       (sets of elements less than \begin{math}n\end{math}, resp. less than \begin{math}n'\end{math})
      are either disjoint (disjoint subtrees), or one is contained in the other
       (indicating that \begin{math}n\end{math} is an ancestor or a descendant of \begin{math}n'\end{math}).
      Note that for any node/leaf, the set of nodes greater than it forms a chain/path with a maximum element corresponding to the root.
      The ``inner nodes'' are thus the non-minimal elements of this order.
\item leaves are mapped surjectively to elements of \begin{math}X\end{math}
      (at least one leaf for each element) by function \begin{math}ll\end{math},
      if exactly one leaf is mapped to any element,
      then the \begin{math}(k,\alpha,\beta)\end{math}-tree questionable decomposition
      is said to be bijective,
\item thus to each internal node \begin{math}node\end{math} is associated the set of elements of \begin{math}X\end{math}
      corresponding to elements \begin{math}ll(l)\end{math} for any leaf \begin{math}l\end{math} below \begin{math}node\end{math},
      defining \begin{math}ll(node)\end{math},
\item \begin{math}nl\end{math} is a mapping from internal nodes, such that \begin{math}nl(node)\end{math} 
      is a \begin{math}(\mathcal{S},ll(node),k)\end{math}-mapping-run,
\item hence, to each element corresponds a subtree of \begin{math}T\end{math},
      and since the intersection of two trees is a tree, we also have a tree corresponding to a couple of elements \begin{math}(x,y)\end{math}.
      For each path of this tree directed from leaves to the root (but this tree may not contain leaves),
      we can define the \begin{math}(\mathcal{S},\{x,y\},k)\end{math}-mapping-run obtained by concatenating
      the \begin{math}(\mathcal{S},ll(node),k)\end{math}-mapping-runs restricted to \begin{math}\{x,y\}\end{math},
      we impose that this mapping-run is a questionable representation of \begin{math}X\end{math} restricted to 
      \begin{math}\{x,y\}\end{math},
      thus the corresponding words associated to \begin{math}x,y\end{math} have a question respecting the structure \begin{math}X\end{math}.
      If \begin{math}ll\end{math} is not bijective, then it entails that all such questions yield the same ``adjacency type''.
\item \begin{math}\alpha\end{math} is the depth of the tree \begin{math}T\end{math}, \emph{i.e.} the supremum of all ordinals contained in the rooted tree.
      Thus for questionable representations, \begin{math}\alpha = 2\end{math}.
\item \begin{math}\beta\end{math} is the depth of the expanded tree \begin{math}T'\end{math} obtained by replacing each internal node with a chain of nodes
       (one for each mapping in the mapping run corresponding to the original node), \emph{i.e.} the supremum of all ordinals obtained by adding the ordinals
      corresponding to the mapping-runs on a path.

\end{itemize}
\begin{math}k\end{math} is called the width of the decomposition;
\begin{math}\alpha\end{math} is called the structural depth of the decomposition.
\begin{math}\beta\end{math} is called the logical depth of the decomposition.
\end{definition}

For finite structures with \begin{math}n \geq 2\end{math} elements,
 \begin{math}k \leq n, \alpha \leq n\end{math}, and \begin{math}\beta \leq n^2\end{math}.

We say that a decomposition is linear if each internal node has exactly two sons,
and for all internal node at least one son is a leaf.
A ``linear infinite rooted tree'' is such an order with the constraints that
all inner nodes are ancestor/descendant related,
and such that there is a ``lowest inner node'' that generates an initial section made exactly of two leaves and itself,
and all other inner nodes generates an initial section that is the union of the initial sections generated by inner nodes below it,
together with the disjoint union of one more leaf.

\begin{lemma}
Any binary structure with at least two elements has a bijective linear \begin{math}(2,\alpha, \leq \beta)\end{math}-tree questionable decomposition,
where \begin{math}\alpha\end{math} is the first ordinal of same cardinal than \begin{math}X\end{math},
and \begin{math}\beta\end{math} is an ordinal of same cardinal than \begin{math}X^2\end{math}.
\end{lemma}

Thus we see that tree-questionable-width is too powerful.
Usually, linear tree-width a.k.a path-width and linear clique-width are strictly less powerful than tree-width, resp. clique-width.

In order to limit this but still obtain a width more powerful than tree-width and clique-width for finite structures,
we use well-known balancing results.
If a finite (weighted)-graph/binary-structure has tree-width \begin{math}k\end{math},
it has a tree-decomposition of width \begin{math}3k-1\end{math} and depth \begin{math}3\lg(n)\end{math}
where nodes have at most two sons
(see \cite{DBLP:conf/wg/Bodlaender88}).
If a finite (weighted)-graph/binary-structure has clique-width \begin{math}k\end{math},
it has a clique-decomposition of width \begin{math}k \times 2^k\end{math} and depth \begin{math}3\lg(n)\end{math}
(see \cite{DBLP:journals/dam/CourcelleV03}).

\begin{lemma}
If a finite binary structure has a tree-decomposition of width \begin{math}k\end{math}
and depth \begin{math}d\end{math},
it has a \begin{math}(k+2,d+1,d)\end{math}-tree questionable decomposition.
\end{lemma}
\begin{proof}
  Consider such a rooted tree decomposition,
  we may easily add to it new leaves associated with elements such that any ``node-bag'' of the tree-decomposition
  is contained in the set of elements corresponding to leaves below it (leaves of the tree-decomposition become internal nodes).
  Then it is trivial to see that the ``adjacency-type'' between two elements is the same in all the bags of the internal nodes,
  and thus to each internal node we associate a mapping run of one mapping corresponding to the substructure of the bag
  with another element added such that this element \begin{math}g\end{math} has the ``default adjacency type'' (no adjacency for graphs)
  with all other elements. All elements of the bag are mapped to themselves, and the other elements in \begin{math}ll(node)\end{math}
  are mapped to \begin{math}g\end{math}.
\end{proof}

\begin{lemma}
If a finite binary structure has a compact clique-decomposition of width \begin{math}k\end{math}
and depth \begin{math}d\end{math},
it has a bijective \begin{math}(2k,d,d-1)\end{math}-tree questionable decomposition.
\end{lemma}
\begin{proof}
  This is really trivial, since all adjacencies are added just after disjoint sum of graphs.
  And thus, at each node we only need one mapping in the mapping-run with 
  \begin{math}k\end{math} elements for the values of the ``left son'' elements by the mapping
  and \begin{math}k\end{math} elements for the values of the ``right son'' elements by the mapping.
\end{proof}

If the proofs of the two previous lemmas were unclear to you, again we strongly suggest reading the book by \cite{DBLP:books/daglib/0030804}.

We draw tree-questionable-decompositions with the leaves on the left, the root on the right,
and all mapping-runs from left to right.
With the previous lemma, the clique decomposition had thus its root on the right.
One may be surprised by the fact that the root was on the left for the lemma giving a clique-decomposition from a questionable representation,
where the mapping-run was also from left to right.

We consider that a tree-questionable-decomposition is balanced 
if its structural depth is at most logarithmic in the size of the graph/structure decomposed.
(It makes sense for a class of graphs/binary structures
where we can say that all graphs in this class have a 
tree-questionable-decomposition of width less than \begin{math}k\end{math}, for some fixed \begin{math}k\end{math},
and structural depth less than some fixed function in \begin{math}O(\log(n))\end{math}.)

We know that a class of graphs has
decidable monadic second-order logic:
\begin{itemize}
\item with edge set quantifications only if it
has bounded tree-width (See \cite{DBLP:journals/apal/Seese91});
\item without edge set quantification but with even cardinality predicates only if it
has bounded clique-width (See \cite{DBLP:journals/jct/CourcelleO07}).\footnote{In both cases, it is ``if and only if''
  when considering certain ``regular'' classes of graphs defined with HR or VR grammars.}
\end{itemize}
The same question is still open for first-order logic.

\begin{openproblem}
Do classes of graphs have decidable first-order theory only if
they have bounded balanced tree-questionable-width ?
Do the class of all graphs with bounded balanced \begin{math}(k,f(n),n^2)\end{math}-tree-questionable-width
for some computable function \begin{math}f \in O(\log(n))\end{math} have decidable first-order theory ?
\end{openproblem}

The same open problem can be wrote for monadic second-order logic without edges sets quantification
and (even) cardinality predicates.
However, bounded balanced tree-questionable-width is not the good measure for this problem,
since grids have undecidable monadic second-order logic,
and we shall see soon that grids have bounded bijective balanced tree-questionable-width.

The following lemma proves that bijective balanced tree-questionable-width can be more powerful than tree-width and clique-width.

\begin{lemma}
\begin{math}p \times q\end{math}-grids have a bijective
\begin{math}(3, \leq \lceil\log_{\frac{5}{3}}(p)\rceil + \lceil\log_{\frac{3}{2}}(q)\rceil, \leq 2 \times \lceil\log_{\frac{5}{3}}(p)\rceil + (\lceil\lg(p)\rceil + 1) \times \lceil\log_{\frac{3}{2}}(q)\rceil)\end{math}-tree questionable decomposition.
\end{lemma}
\begin{proof}
(We do not repeat in this proof that all tree-questionable decompositions are bijective.)
If the grid is a path, then the proof is by induction on the length of the path.
A path of length 2 has a \begin{math}(2, 2, 1)\end{math}-tree questionable decomposition.
A path of length 3 has a \begin{math}(2, 2, 2)\end{math}-tree questionable decomposition (or \begin{math}(3, 2, 1)\end{math}-tree questionable decomposition).
A \begin{math}(3, \leq \lceil\log_{\frac{5}{3}}(p)\rceil, \leq 2 \times \lceil\log_{\frac{5}{3}}(p)\rceil)\end{math}-tree questionable decomposition
for a path of length \begin{math}p \geq 4\end{math} can be obtained by adding a root node
above the root nodes of the two
\begin{math}(3, \leq \lceil\log_{\frac{5}{3}}(\lceil\frac{p}{2}\rceil)\rceil, \leq 2 \times \lceil\log_{\frac{5}{3}}(\lceil\frac{p}{2}\rceil)\rceil)\end{math}-tree questionable decomposition
and \begin{math}(3, \leq \lceil\log_{\frac{5}{3}}(\lfloor\frac{p}{2}\rfloor)\rceil, \leq 2 \times \lceil\log_{\frac{5}{3}}(\lfloor\frac{p}{2}\rfloor)\rceil)\end{math}-tree questionable decomposition of the two ``half'' subpaths.
This root node contains a mapping run of two mappings:
the first mapping associates the two adjacent vertices of the subpaths to two adjacent vertices,
and all other vertices are mapped to an isolated vertex;
the second mapping associates all vertices of the first subpath to the same vertex
 and all vertices of the second subpath to another vertex, image-vertices are non adjacent.
(Alternatively, one can have a mapping run of only one mapping with image set of cardinal 4,
where the two adjacent vertices of the subpaths are mapped to two adjacent vertices,
and all other vertices are mapped to two isolated vertices, one for each subpath.
Thus, one obtains a \begin{math}(4, \leq \lceil\log_{\frac{5}{3}}(p)\rceil, \leq \lceil\log_{\frac{5}{3}}(p)\rceil)\end{math}-tree questionable decomposition.)

Again by induction, one can obtain a
\begin{math}(3, \leq \lceil\log_{\frac{5}{3}}(p)\rceil + \lceil\log_{\frac{3}{2}}(q)\rceil,
                \leq 2 \times \lceil\log_{\frac{5}{3}}(p)\rceil + (\lceil\lg(p)\rceil + 1) \times \lceil\log_{\frac{3}{2}}(q)\rceil)\end{math}-tree questionable decomposition
for a \begin{math}p \times q\end{math}-grid by adding a root node
above the root nodes of the two
\begin{math}(3, \leq \lceil\log_{\frac{5}{3}}(p)\rceil + \lceil\log_{\frac{3}{2}}(\lceil\frac{q}{2}\rceil)\rceil,
                \leq 2 \times \lceil\log_{\frac{5}{3}}(p)\rceil + (\lceil\lg(p)\rceil + 1) \times \lceil\log_{\frac{3}{2}}(\lceil\frac{q}{2}\rceil)\rceil)\end{math}
and
\begin{math}(3, \leq \lceil\log_{\frac{5}{3}}(p)\rceil + \lceil\log_{\frac{3}{2}}(\lfloor\frac{q}{2}\rfloor)\rceil,
                \leq 2 \times  \lceil\log_{\frac{5}{3}}(p)\rceil + (\lceil\lg(p)\rceil + 1) \times \lceil\log_{\frac{3}{2}}(\lfloor\frac{q}{2}\rfloor)\rceil)\end{math}-tree questionable decomposition of the two ``half'' subgrids.
This root node contains a mapping run of \begin{math}\lceil\lg(p)\rceil + 1\end{math} mappings:
in the first \begin{math}\lceil\lg(p)\rceil\end{math} mappings, there are two isolated image-vertices \begin{math}x_0, x_1\end{math},
and each adjacent pair of vertices of the ``ladder'' are mapped to the same vertex;
either \begin{math}x_0\end{math} or \begin{math}x_1\end{math} in the \begin{math}i\end{math}th mapping,
if the \begin{math}i\end{math}th bit of the binary representation of the level of this edge of the ladder is 0 or 1.
Thus these mappings define the non-adjacency of all vertices,
and the last mapping uses only two adjacent image-vertices,
giving distinct images to previously paired elements.
\end{proof}

This is a funny situation where bijective balanced tree-questionable-width is more powerful than clique-width,
but maybe is incomparable with tree-width (for edge-weighted graphs).
In all cases, balanced tree-questionable-width is more powerful than tree/clique/questionable-width.

\begin{openproblem}
Do classes of (edge-weighted) graphs of bounded balanced tree-questionable-width
have \emph{bijective} bounded balanced tree-questionable-width ?
If not, do edge-weighted graphs of bounded tree-width have bijective bounded balanced tree-questionable-width ?
\end{openproblem}

We believe that the triangulated grids (in each square is added a diagonal edge from bottom-left to top-right, for example)
have unbounded bijective balanced tree-questionable-width,
since cutting as we did for the grids would give something close to a ``zigzag'' when joining both parts.
Thus, it is probable that bijective balanced tree-questionable-width is too weak to characterize first-order logic decidability.
Moreover, we believe that these triangulated grids, and also any class of graphs embeddable in some fixed surface,
have bounded non-bijective balanced tree-questionable-width, since it should be possible to cut the graph on the surface
 ``along stripes of width 2'', the number of stripes being bounded by a logarithm of the number of vertices
 plus some constant depending on the surface.
If our intuitions are correct, non-bijective balanced tree-questionable-width would be too powerful for first-order logic decidability,
and non-bijective balanced tree-questionable-width would be strictly stronger than bijective balanced tree-questionable-width.
This is crude reasoning, it may also be that fine tuning of the functions bounding structural and logical depths
of the tree-questionable-decompositions, or maybe additional constraints on the rooted tree, may yield
 the sought width for first-order logic decidability.

We note that other attempts at defining hierarchic decompositions can be done along the same lines.
Indeed, instead of a questionable representation for the bipartite structure obtained when joining two substructures,
one may consider a path/tree/clique-decomposition of the bipartite structure
(or \begin{math}k\end{math}-partite structure if we join more than two substructures at once).
Thus, defining ``second-order recursive path/tree/clique-decomposition''.
(We do not list all variants of base decompositions that can be used instead of questionable representations,
or path/tree/clique-decomposition.)
For these decompositions, the same lemma that any binary structure can be decomposed with a
``linear second-order recursive path/tree/clique-decomposition'' of linear depth holds.
One may generalize, by defining ``third-order recursive path/tree/clique-decomposition'',
where each node of the decomposition is associated to a ``second-order recursive path/tree/clique-decomposition''
of the \begin{math}k\end{math}-partite structure,
and so on.
Most of these ``recursive hierarchic decompositions'', maybe all, will be too powerful to obtain any
useful result. However, maybe first-order logic decidability will correspond to such recursive hierarchic decomposition
for an appropriate base decomposition and appropriate additional constraints.

%\section{Open problems}
%\label{section:open_problems}

%\begin{openproblem}
%Is it true that any (countable) partial order \begin{math}PO\end{math}
%has a (finite) partial questionable representation?
%\end{openproblem}

%\section{Conclusion}

\section*{Acknowledgements}
\label{section:acknowledgements}

We thank God: Father, Son, and Holy Spirit. We thank Maria.
They help us through our difficulties in life.
We thank St\'ephan Thomass\'e for his very useful remark.

\nocite{*}
\bibliographystyle{abbrvnat}
% use the following instead if you encounter problems 
%\bibliographystyle{alpha}
\bibliography{LL2019OrdreRepresentations}
\label{section:bibliography}

\end{document}